\documentclass[10pt,a4paper,english,reqno]{amsart}
\usepackage{amsmath}
\usepackage{a4}
\usepackage{amssymb}
\usepackage{amsthm}
\usepackage{bbm}
\usepackage{bm}
\usepackage{dsfont}
\usepackage{mathtext}
\usepackage[T1]{fontenc}

\usepackage[utf8]{inputenc}
\usepackage[scaled=0.86]{helvet}
\usepackage{graphicx}
\usepackage{array}
\usepackage{multirow}
\usepackage{tabularx}
\usepackage{tabulary}
\usepackage{multirow,bigdelim}

\usepackage[english]{babel}
\usepackage{latexsym}
\usepackage{mathtools}
\usepackage{tikz}
\usepackage{cite}
\usepackage{float}
\usepackage{placeins}
\usepackage{caption}

\usepackage{enumerate}
\usepackage{enumitem}
\usepackage{longtable}
\usepackage{cases}
\usepackage[unicode=true,pdfusetitle,
 bookmarks=true,bookmarksnumbered=false,bookmarksopen=false,
 breaklinks=false,pdfborder={0 0 1},backref=section,colorlinks=false]
 {hyperref}
\newcommand{\N}{\mathbb N}

\makeatletter

\pdfpageheight\paperheight
\pdfpagewidth\paperwidth

\theoremstyle{plain}
\newtheorem{thm}{Theorem}[section] 
\theoremstyle{definition}
\newtheorem{Def}[thm]{\protect\definitionname}

\theoremstyle{plain}

\theoremstyle{plain}
\newtheorem{Prop}[thm]{Proposition}
\newtheorem{lemma}[thm]{Lemma}
\newtheorem{Cor}[thm]{Corollary}
\theoremstyle{definition}

\theoremstyle{remark}
\newtheorem{Rem}[thm]{Remark}

\usepackage{a4wide}
\usepackage{bbm}
\usepackage[arrow, matrix, curve]{xy}
\usepackage{tikz}

\newcommand\upperleft[3]{\prescript{#1}{}{\mathrlap{\smash{#2#3}}\phantom{#3}}}

\DeclareMathOperator{\cor}{Cor}
\DeclareMathOperator{\dist}{dist}

\DeclareMathOperator{\id}{id}
\DeclareMathOperator{\Image}{Im}

\DeclareMathOperator{\spec}{spec}

\renewcommand{\tilde}{\widetilde}

\renewcommand{\emptyset}{\varnothing}
\renewcommand{\bar}{\overline}
\usepackage{amssymb}
\usepackage{mathtools}

\makeatother

  \providecommand{\definitionname}{Definition}

\begin{document}
\title{Strong laws of large numbers for intermediately trimmed Birkhoff sums of observables with infinite mean}

\author[Kesseb\"ohmer]{Marc Kesseb\"ohmer}
  \address{Universit\"at Bremen, Fachbereich 3 -- Mathematik und Informatik, Bibliothekstr. 1, 28359 Bremen, Germany}
  \email{\href{mailto:mhk@math.uni-bremen.de}{mhk@math.uni-bremen.de}}
\author[Schindler]{Tanja Schindler}
\address{Australian National University, Research School Finance, Actuarial Studies and Statistics, 26C Kingsley St,
Acton ACT 2601, Australia}
  \email{\href{mailto:tanja.schindler@anu.edu.au}{tanja.schindler@anu.edu.au}}

\date{\today}

\date{\today}

\begin{abstract}
We consider dynamical systems on a finite measure space fulfilling a spectral gap property
and Birkhoff sums of a non-negative, non-integrable observable.
For such systems we generalize strong laws of large numbers for intermediately trimmed sums only known for independent random variables.
The results split up in trimming statements for general distribution functions and for regularly varying tail distributions.
In both cases the trimming rate can be chosen in the same or almost the same way as in the i.i.d.\ case. 
As an example we show that piecewise expanding interval maps fulfill the necessary conditions for our limit laws.
As a side result we obtain strong laws of large numbers for truncated Birkhoff sums.
\end{abstract}
\maketitle

\section{Introduction and statement of main results}
We consider an ergodic dynamical system $\left(\Omega,\mathcal{A},T, \mu\right)$ with $\mu$ a probability measure  
and a stochastic processes given by the Birkhoff sums $\mathsf{S}_n\chi\coloneqq\sum_{k=1}^n\chi\circ  T^{k-1}$ with $\mathsf{S}_0\chi=0$ for some measurable function $\chi:\Omega\to \mathbb{R}_{\geq 0}$. 
Regarding strong laws of large numbers there is a crucial difference between $\int\chi\mathrm{d}\mu$ being finite or not.
In the finite case we obtain by Birkhoff's ergodic theorem that $\mu$-almost surely (a.s.)  
\begin{align*}
 \lim_{n\to\infty}\frac{\mathsf{S}_n\chi}{n}=\int\chi\mathrm{d}\mu,
\end{align*}
i.e.\  the strong law of large numbers is fulfilled, whereas 
in the case of an observable with infinite expectation, 
Aaronson showed in \cite{aaronson_ergodic_1977} that for
all positive sequences $\left(d_{n}\right)_{n\in\mathbb{N}}$ we have
$\mu$-a.s.
\[
\limsup_{n\rightarrow\infty}\frac{\mathsf{S}_{n}\chi}{d_{n}}=+\infty\text{ \,\,\,\,\ or \,\,\,\,\,}\liminf_{n\rightarrow\infty}\frac{\mathsf{S}_{n}\chi}{d_{n}}=0.
\]
However, there might be a strong law of large numbers after deleting a number
of the largest summands from the partial $n$-sums. 
More precisely,
for each $n\in\mathbb{N}$ we chose a permutation $\sigma\in\mathcal{S}_{n}$
of $\left\{ 0,\ldots,n-1\right\} $ such that $\chi\circ  T^{\sigma\left(0\right)}\geq \chi\circ  T^{\sigma\left(1\right)}\geq\ldots\geq \chi\circ  T^{\sigma\left(n-1\right)}$
and for given $\left(b_{n}\right)\in\mathbb{N}_{0}^{\mathbb{N}}$ we define 
\begin{align*}
\mathsf{S}_{n}^{b_{n}}\chi & \coloneqq\sum_{k=b_n}^{n-1}\chi\circ  T^{\sigma\left(k\right)}.
\end{align*}
If $b_{n}=r\in\mathbb{N}$ is fixed for all $n\in\mathbb{N}$ then $\left(\mathsf{S}_{n}^{r}\chi\right)$
is called a \emph{lightly trimmed sum} \emph{process}.
If we allow the sequence $\left(b_{n}\right)\in\mathbb{N}_0^{\mathbb{N}}$
to diverge to infinity such that $b_{n}=o\left(n\right)$, i.e.\ $\lim_{n\rightarrow\infty}b_{n}/n=0$,
then $\left(\mathsf{S}_{n}^{r}\chi\right)$ is called
an \emph{intermediately} (also \emph{moderately}) \emph{trimmed sum process}.
If there exist $r\in\N$ and a sequence of constants $(d_n)$ such that 
$\lim_{n\to\infty}\mathsf{S}_n^r\chi/d_n=1$ a.s.\ we refer to it as a \emph{lightly trimmed strong law} and similarly
if for an intermediately trimmed sum $\mathsf{S}_n^{b_n}\chi$ there exists $(d_n)$ such that $\lim_{n\to\infty}\mathsf{S}_n^{b_n}\chi/d_n=1$ a.s.\
we refer to it as an \emph{intermediately trimmed strong law}.

Trimming results for independent, identically distributed (i.i.d.) random variables are well studied. 
Mori developed
in \cite{mori_strong_1976} and \cite{mori_stability_1977} general conditions
for a lightly trimmed strong law to hold. These results have been
generalized by Kesten and Maller, see \cite{maller_relative_1984},
\cite{kesten_ratios_1992}, and \cite{kesten_effect_1995}.

It becomes clear from a result by
Kesten, see \cite{kesten_convergence_1993}, that light trimming is not always sufficient.
His result implies in particular that a weak law of large numbers for a lightly trimmed sum of i.i.d.\ random variables holds
if and only if it also holds for the untrimmed sum. 
This rules out the possibility for a lightly trimmed strong law if a weak law of large numbers does not hold.

A special case in this context is the case of \emph{regularly varying} tail variables with index between $-1$ and $0$.
To state more precisely the situation we require that the distribution function $F$ (i.e.\ in the dynamical systems setting we have $F\left(x\right)=\mu\left(\chi\leq x\right)$)
fulfills $1-F\left(x\right)\sim x^{-\alpha}L\left(x\right)$
with $0<\alpha<1$ and $L$ a slowly varying function.
Here, $u\left(x\right)\sim w\left(x\right)$ means that $u$ is
asymptotic to $w$ at infinity, that is $\lim_{x\rightarrow\infty}u\left(x\right)/w\left(x\right)=1$
and $L$ being \emph{slowly varying} means that for every $c>0$ we
have $L\left(cx\right)\sim L\left(x\right)$.

It can be easily deducted from \cite[VII.7 Theorem 2 and VIII.9 Theorem 1]{feller_introduction_1971} that in the just mentioned case 
for i.i.d.\ random variables no weak law of large numbers and thus no lightly trimmed strong law holds.
Furthermore, the case of i.i.d.\ random variables with regularly varying tails is treated by Haeusler and Mason in \cite{haeusler_laws_1987} and Haeusler in \cite{haeusler_nonstandard_1993}, in which a law of  an iterated logarithm is established.
As we are here considering sums of non-negative random variables instead of Birkhoff sums, we will
denote in these cases the sum trimmed by the $r$ maximal terms by $S_n^r$, 
i.e.\ we distinguish the trimmed Birkhoff sum $\mathsf{S}_n\chi$ from a trimmed sum of i.i.d.\ random variables $S_n^r$.
With this notation and setting $F^{\leftarrow}:\left[0,1\right]\to\mathbb{R}_{\geq 0}$ denoting the \emph{generalized inverse
function} of $F$, i.e.\ 
$F^{\leftarrow}\left(y\right) \coloneqq\inf\left\{ x\in\mathbb{R}\colon F\left(x\right)\geq y\right\}$,
then a combination of the results by Haeusler and Mason imply that there exists a non-stochastic $\gamma$ depending on $n$ and the trimming function $b_n$ such that
\begin{align*}
\limsup_{n\to\infty} \pm \frac{S_n^{b_n}-n\cdot \int_0^{1-b_n/n}F^{\leftarrow}\left(s\right)\mathrm{d}s}{\gamma\left(n,b_n\right)}
\end{align*}
almost surely equals $1$ if $\lim_{n\to\infty}b_n/\log\log n=\infty$, see \cite{haeusler_laws_1987}, 
and, almost surely equals a constant $M$ if $b_n\sim c\cdot \log\log n$,  see \cite{haeusler_nonstandard_1993}.
By comparing the asymptotic behavior of the norming and centering sequences $\gamma\left(n,b_n\right)$ and $n\cdot\int_0^{1-b_n/n}F^{\leftarrow}\left(s\right)\mathrm{d}s$ referring to \cite[Section 4]{haeusler_nonstandard_1993} one can conclude that $\lim_{n\to\infty}S_n^{b_n}/(n\cdot \int_0^{1-b_n/n}F^{\leftarrow}\left(s\right)\mathrm{d}s)=1$ almost surely
if and only if $\lim_{n\to\infty}b_n/\log\log n=\infty$, i.e.\ an intermediately trimmed strong law follows as a special case.

The results for the lightly trimmed case have been generalized to dependent random variables from different contexts.
One of the first investigated examples of this situation is the unique continued
fraction expansion of an irrational $x\in\left[0,1\right]$ given by
\[x\coloneqq\frac{1}{a_{1}\left(x\right)+\cfrac{1}{a_{2}\left(x\right)+\ddots}}.
\]
In this case we consider the space $\Omega\coloneqq \left[0,1\right]\backslash\mathbb{Q}$,  
the Gauss measure $\mu$ given by $d\mu(x)\coloneqq1/\left(\log2\left(1+x\right)\right) d\lambda\left(x\right)$ with  $\lambda$ denoting the Lebesgue measure restricted to $\left[0,1\right]$, and the Gauss map $T$ defined as $ T x\coloneqq \left\{1/x\right\}\coloneqq 1/x-\left\lfloor 1/x\right\rfloor$.
The observable $\chi\colon\Omega\to\mathbb{N}$ with $\chi (x)\coloneqq\left\lfloor 1/x\right\rfloor$ and $\left\lfloor x\right\rfloor \coloneqq\max\left\{ n\in\mathbb{N}\colon n\leq x\right\}$ 
gives then rise to the  stationary (dependent, but $\bm{\psi}$-mixing) process $\chi\circ  T^{n-1}=a_{n}$, $n\in\mathbb{N}$, of the $n$-th continued fraction digit. 
Even though a strong law of large numbers can not
hold for $\mathsf{S}_{n}\chi$, Diamond and Vaaler showed in \cite{diamond_estimates_1986}
that under light trimming with $r=1$ we have a.s.
\[
\lim_{n\rightarrow\infty}\frac{\mathsf{S}_{n}^{1}\chi}{n\log n}=\frac{1}{\log2}.
\]
These results were generalized in \cite{nakada_metric_2002} and \cite{nakada_metrical_2003} to other continued fraction expansions.

In \cite{haynes_quantitative_2014} Haynes gave a quantitative strong law of large numbers under trimming for a certain class of observables $\chi$.
He considered the sum $\widetilde{\mathsf{S}}_n$ with $\widetilde{\mathsf{S}}_n\chi(x)\coloneqq\mathsf{S}_n\chi(x)-\delta\left(n,x\right)\max_{0\leq i\leq n-1}\chi\circ  T^i(x)$
with $\delta\left(n,x\right)\in\left\{0,1\right\}$, i.e.\ the sum is trimmed by the maximal element, but depending on $x$. Then he gave an error term for $\widetilde{\mathsf{S}}_n\chi-d_n$
with $(d_n)$ the norming sequence as above.

One way to study limit theorems in the dynamical systems setting is to prove limit theorems 
for mixing random variables as dynamical system are often $\bm{\psi}$- or at least $\bm{\phi}$- or $\bm{\phi}_{\textrm{rev}}$-mixing.
For a precise definitions of different kinds of mixing see \cite{bradley_basic_2005}.
One approach in this direction is a result by Aaronson and Nakada extending the results by Mori to $\bm{\psi}$-mixing random variables, see in \cite{aaronson_trimmed_2003}, 
i.e.\ they gave sufficient conditions for a lightly trimmed strong law to hold.

In this paper we will study intermediately trimmed strong laws for dynamical systems fulfilling a spectral gap property with the exact assumptions on the system given as Property $\mathfrak{D}$ in Definition \ref{def: Prop D}.
The random variables of our main example given in Section \ref{Examples}
are at least exponentially $\bm{\phi}_{\textrm{rev}}$-mixing, which was proven in \cite{aaronson_mixing_2005}.
One approach to prove an intermediately trimmed strong law would thus be to prove a limit theorem for $\bm{\phi}_{\textrm{rev}}$-mixing systems.
However, it is difficult to prove these trimming results by only assuming a certain speed of mixing.
These problems are elucidated in Section \ref{subsubsec: comments about method}.

Our method is based on proving an exponential inequality for the dynamical systems
using a spectral method similar to the Nagaev-Guivarc'h spectral method for the central limit theorem,
giving stronger results than exponential inequalities for mixing random variables. 
A sketch of the proof will be given in Section \ref{subsec: sketch proof} 
and some of the difficulties will be discussed in Section \ref{subsubsec: comments about method}.

The spectral gap property for dynamical systems, 
guaranteed  by our later stated Property $\mathfrak{C}$, see Definition \ref{def: prop C},
is a typical assumption under which limit theorems for dynamical systems can be proven. 
The first statements proven in this setting were central limit theorems for Markov chains, see \cite{nagaev_limit_1957},
and generalizations for other dynamical systems, see \cite{rousseau_theoreme_1982} and \cite{guivarch_thoeremes_1988}.
Also other limit theorems have been proven in this setting as for example
local central limit theorems, see \cite{hennion_limit_2001}, \cite{guivarch_thoeremes_1988}, and
\cite{herve_nagaev_2010} for higher dimensions,
 Berry-Esseen theorems, see \cite{gouezel_berry_2005}, \cite{herve_nagaev_2010},
and almost sure invariance principles, see \cite{melbourne_vector_2009} and \cite{gouezel_almost_2010}.

There has also been some recent interest 
in limit theorems for dynamical systems with heavy tail distributions. 
This includes in particular convergence to a stable law. 
Aaronson and Denker proved in \cite{aaronson_local_2001} 
some necessary conditions for a stable limit laws for dynamical systems and observables with heavy tails.
Gou\"ezel proved in \cite{gouezel_characterization_2010} the necessity of  those conditions
using the work of Sarig, see \cite{sarig_continuous_2006}, 
which is restricted to a  particular class of distribution functions.
These results were also generalized for dynamical systems on intermittent maps, see \cite{melbourne_weak_2015}.
Furthermore, Gou\"ezel also proved a stable limit law for observables of the doubling map, see \cite{Gouezel_stable_2008}.

Tyran-Kaminska studied in \cite{tyran_weak_2010} the functional convergence of normalized Birkhoff sums 
with heavy tailed observables to $\alpha$-stable processes.

Aaronson and Zweim\"uller proved in \cite{aaronson_limit_2014} some stable laws 
and additionally a one-sided law of the iterated logarithm for
mixing dynamical systems and observables with heavy tails.
All these results use intrinsically  transfer operator techniques.
For further references concerning transfer operator methods we refer the reader to the review papers \cite{gouezel_limit_2015} and \cite{fan_spectral_2003}.

In contrast to the before mentioned stable laws, Carney and Nicol 
 investigated growth rates of Birkhoff
sums of non-integrable observables on a finite measure space giving some almost sure results, see \cite{carney_dynamical_2017}.
Closely related to the behavior of a non-integrable observables on a finite measure space 
is the behavior of observables on an infinite measure space. 

 The pointwise convergence behaviour of Birkhoff sums with respect to iterates of the transfer operator acting on integrable observables with a finite number of  poles have been studied in the context of both finite and infinite measure-preserving dynamical systems in \cite{KKS}. 
 In \cite{Lenci_pointwise_2018} Lenci and Munday further investigate the question 
under which conditions a Birkhoff ergodic theorem holds for infinite measure-preserving dynamical systems.

In the following section we will give the precise setting and 
even though not every $\bm{\phi}_{\textrm{rev}}$-mixing system can be described by a dynamical system with spectral gap,
our setting for dynamical systems is rather general, as many dynamical systems fulfill the later stated Property $\mathfrak{C}$,
for example subshifts of finite type, see for example \cite{parry_zeta_1990} or \cite{baladi_positive_2000},
piecewise expanding interval maps, see \cite{hofbauer_ergodic_1982}
and a generalization for infinitely many partitions, see \cite{rychlik_bounded_1983},
and Anosov or Axiom A systems, see
\cite{baladi_anisotropic_2007} and \cite{gouezel_compact_2008}.

To prove an intermediately trimmed strong law we need additional assumptions on the observable $\chi$ given in Property $\mathfrak{D}$.
As a concrete example we state later in Section \ref{Examples} that under rather mild assumptions a system of piecewise expanding interval maps 
fulfills the additional assumptions of Property $\mathfrak{D}$.

As we show, our setting is rather general in the sense that there are dynamical systems
for which a lightly trimmed strong law does not hold despite the fact that it would hold for i.i.d.\ random variables
with the same distribution function. Keeping the system but using another observable with a different distribution function allows intermediately trimmed strong laws 
for the same trimming sequence as in the i.i.d.\ case, 
see Remark \ref{rem: doubling map ex}.

In Sections \ref{subsec: gen results} and \ref{subsec: reg var results} we will present intermediately trimmed strong laws for the above mentioned dynamical systems,
for general distribution functions in Section \ref{subsec: gen results} and for regularly varying tail distributions
in Section \ref{subsec: reg var results}. 
The results for general distribution functions are almost as strong as the i.i.d.\ trimming results in \cite{kessebohmer_strong_2016}
and for the regular variation case we show that an intermediately trimmed strong law holds 
for the same trimming sequence $\left(b_n\right)$ which in the i.i.d.\ case can be derived from \cite{haeusler_laws_1987}.

As side results  we obtain limit theorems for sums of truncated random variables, see Section \ref{subsec: intro truncated rv}.
Namely, for $\chi:\Omega\to\mathbb{R}_{\geq 0}$ and a real valued sequence $\left(f_{n}\right)_{n\in\mathbb{N}}$
we consider the truncated sum $\mathsf{T}_n^{f_n}\chi\coloneqq \sum_{k=1}^n\left(\chi \cdot\mathbbm{1}_{\{\chi\leq f_n\}}\right)\circ T^{k-1}$.
Limit results for these truncated sums have been of recent interest for i.i.d.\ random variables, see \cite{nakata_limit_2015} and \cite{gyorfi_rate_2011}.
\medskip

In Section \ref{sec: main ideas proof} we will give the main ideas of the proof, explain the differences to the i.i.d.\ case,
and give the structure of the rest of the paper.

\subsection{Basic setting}
In the following we will define our two main properties. 
The first, Property $\mathfrak{C}$, restricts to dynamical systems with a spectral gap property
and is part of our next property, Property $\mathfrak{D}$. 
Property $\mathfrak{D}$
is a property on the dynamical system and additionally on 
the admissible observables allowing us  to state and prove intermediately trimmed strong laws.

\begin{Def}[Property $\mathfrak{C}$]\label{def: prop C}
Let $\left(\Omega, \mathcal{A},  T, \mu\right)$ be a dynamical system with $ T$ a non-singular transformation and $\widehat{ T}:\mathcal{L}^1\to \mathcal{L}^1$ be the transfer operator of $ T$, 
i.e.\ the uniquely defined operator such that for all $f\in\mathcal{L}^1$ and $g\in\mathcal{L}^{\infty}$ we have
\begin{align}
\int \widehat{ T}f\cdot g\mathrm{d}\mu=\int f\cdot g\circ T\mathrm{d}\mu,\label{hat CYRI}
\end{align}
see e.g. \cite[Section 2.3]{MR3585883} for further details.
Furthermore, let $\mathcal{F}$ be a subset of the measurable functions forming  a Banach algebra with respect to the  norm $\left\|\cdot\right\|$.
We say that $\left(\Omega, \mathcal{A}, T , \mu,\mathcal{F},\left\|\cdot\right\|\right)$ has Property $\mathfrak{C}$ if the following conditions hold:
\begin{itemize}
\item
$\mu$ is a $T $-invariant, mixing probability measure.
\item
$\mathcal{F}$ contains the constant functions and for all $f\in\mathcal{F}$ we have
\begin{align}
\left\|f\right\|\geq \left|f\right|_{\infty}.\label{ineq <l}
\end{align}
\item 
$\widehat{T}$ is a bounded linear operator with respect to $\left\|\cdot\right\|$, i.e.\ there {exists} a constant $K_0>0$ such that for all $f\in\mathcal{F}$ we have
\begin{align}
\left\|\widehat{T}f\right\|\leq K_0\cdot\left\|f\right\|.\label{C 0 f} 
\end{align} 
\item 
$\widehat{T}$ has a spectral gap on $\mathcal{F}$ with respect to $\left\|\cdot\right\|$, see Definition \ref{def spec gap}.
\end{itemize}
\end{Def}
The above mentioned property is a widely used setting for dynamical systems. 
In particular it implies that the transfer operator has $1$ 
as a unique and simple eigenvalue on the unit circle which implies an exponential decay of correlation.
We will give proofs of these properties in Section \ref{subsec: spec gap corr decay}.

However, in order to state our main theorems we need additional assumptions on the observable $\chi$ acting on a system fulfilling Property $\mathfrak{C}$.
\begin{Def}[Property $\mathfrak{D}$]\label{def: Prop D}
For a Banach algebra $\mathcal{F}$ 
and for a fixed measurable function $\chi:\Omega\to\mathbb{R}_{\geq 0}$ we set, for all $\ell\in\mathbb{R}_{\geq 0}$,
\begin{align*}
\upperleft{\ell}{}{\chi}\coloneqq\chi\cdot\mathbbm{1}_{\left\{\chi\leq \ell\right\}}.
\end{align*}

We say that $\left(\Omega, \mathcal{A}, T , \mu,\mathcal{F},\left\|\cdot\right\|,\chi\right)$ has Property $\mathfrak{D}$ if the following conditions hold:
\begin{itemize}
\item $\left(\Omega, \mathcal{A}, T , \mu,\mathcal{F},\left\|\cdot\right\|\right)$ fulfills Property $\mathfrak{C}$.
\item
There exists
$K_1>0$ such that for all $\ell>0$, 
\begin{align}
\left\|\upperleft{\ell}{}{\chi}\right\|\leq K_1\cdot \ell.\label{C 1}
\end{align}
\item
There exists 
$K_2>0$ such that for all $\ell>0$, 
\begin{align}
\left\|\mathbbm{1}_{\left\{\chi>\ell\right\}}\right\|\leq K_2.\label{C 2} 
\end{align}
\end{itemize}
\end{Def}
Even though a lot of dynamical systems fulfill Property $\mathfrak{C}$ as mentioned in the introduction, 
it is not immediately clear for which observables $\chi$ they additionally fulfill Property $\mathfrak{D}$.
However, it turns out that for piecewise expanding interval maps the assumptions on $\chi$ are rather weak, see Section \ref{Examples}.

\subsection{Results for general distribution functions}\label{subsec: gen results}
The following theorem provides us with a method to find a trimming
sequence $\left(b_{n}\right)$ if the distribution function $F$ is given. Before stating this theorem we 
define $\left\lceil x\right\rceil \coloneqq\min\left\{ n\in\mathbb{N}\colon n\geq x\right\}$.
\begin{thm}
\label{find bn} 
Let $\left(\Omega, \mathcal{A}, T , \mu,\mathcal{F},\left\|\cdot\right\|,\chi\right)$ fulfill Property $\mathfrak{D}$
and let $\left(f_{n}\right)_{n\in\mathbb{N}}$ be a sequence
of positive real numbers tending to infinity.
Fix $0<\epsilon<1/4$ 
such that for 
\[
a_{n}\coloneqq n\cdot\left(1-F\left(f_{n}\right)\right),\;\; d_{n}\coloneqq n\int_{0}^{f_{n}}x\,\mathrm{d}F\left(x\right),\,\, n\in\mathbb{N},
\]
we have 
\begin{align}
\lim_{n\to\infty}f_{n}/d_{n}\cdot\max\left\{a_{n}^{1/2+\epsilon}\left(\log\log n\right)^{1/2-\epsilon},\log n\right\}& =0.\label{cond 1 find bn}
\end{align}
Then there exists $W>0$ independent of $\epsilon$ and $(f_n)$ such that for
\begin{align*}
b_{n}\coloneqq\left\lceil a_{n}+W\cdot\max\left\{a_n^{1/2+\epsilon}\cdot\left(\log\log n\right)^{1/2-\epsilon},\log\log n\right\}\right\rceil, 
\end{align*}
$n\in\mathbb{N}$, we have 
\[
\lim_{n\rightarrow\infty}\frac{S_{n}^{b_{n}}}{d_{n}}=1\text{ a.s.}
\]
\end{thm}
\begin{Rem}
This theorem is the equivalent to \cite[Theorem B]{kessebohmer_strong_2016} for the setting of dynamical systems.
In \cite{kessebohmer_strong_2016} we also give an example how to find a proper trimming function for a given distribution function $F$.
\end{Rem}

As a corollary we obtain that an intermediately trimmed strong law under these conditions always holds.
\begin{Cor}
\label{S b(n) immer} 
Let $\left(\Omega, \mathcal{A}, T , \mu,\mathcal{F},\left\|\cdot\right\|,\chi\right)$ fulfill Property $\mathfrak{D}$.
Then there exist a sequence of natural numbers $\left(b_{n}\right)_{n\in\mathbb{N}}$
with $b_{n}=o\left(n\right)$ and a sequence of positive reals $\left(d_{n}\right)_{n\in\mathbb{N}}$
such that 
\[
\lim_{n\rightarrow\infty}\frac{\mathsf{S}_{n}^{b_{n}}\chi}{d_{n}}=1\text{ a.s.}
\]
\end{Cor}
\begin{Rem}
We would like to point out  that the sequence $\left(d_{n}\right)_{n\in\mathbb{N}}$
is not necessarily asymptotic to the sequence of expectations $\left(\int\mathsf{S}_{n}^{b_{n}}\chi\mathrm{d}\mu\right)_{n\in\mathbb{N}}$. This has been shown in \cite[Remark 2]{kessebohmer_strong_2016} for a sequence of i.i.d.\   summands.
\end{Rem}

\subsection{Results for regularly varying tails}\label{subsec: reg var results}
For stating our main theorem we set 
\[
\Psi\coloneqq\left\{ u:\mathbb{N}\rightarrow\mathbb{R}^{+}\colon\sum_{n=1}^{\infty}\frac{1}{u\left(n\right)}<\infty\right\}.
\]
Then our main result reads as follows:
\begin{thm}\label{Sb(n)} 
Let $\left(\Omega, \mathcal{A}, T , \mu,\mathcal{F},\left\|\cdot\right\|,\chi\right)$ fulfill Property $\mathfrak{D}$
and let $\chi:\Omega\rightarrow\mathbb{R}_{\geq 0}$ be such that 
$\mu\left(\chi>x\right)=L\left(x\right)/x^{\alpha}$ with $L$ a slowly varying function and $0<\alpha<1$. 
Further, let $\left(b_{n}\right)_{n\in\mathbb{N}}$ be a sequence
of natural numbers tending to infinity with $b_{n}=o\left(n\right)$.
If there exists $\psi\in\Psi$ such that 
\begin{align}
\lim_{n\rightarrow\infty}\frac{b_{n}}{\log\psi\left(\left\lfloor \log n\right\rfloor \right)}=\infty\label{eq: t cond a2}
\end{align}
then there exists a positive valued sequence $\left(d_n\right)_{n\in\mathbb{N}}$ such that 
\begin{align}
\lim_{n\rightarrow\infty}\frac{\mathsf{S}_n^{b_n}\chi}{d_n}=1\text{ a.s.}\label{eq: lim Snbnchi}
\end{align}
and $\left(d_n\right)$ fulfills 
\begin{align}
d_{n}\sim\frac{\alpha}{1-\alpha}\cdot n^{1/\alpha}\cdot b_{n}^{1-1/\alpha}\cdot \left(L^{1/\alpha}\right)^{\#}\left(\left(\frac{n}{b_{n}}\right)^{1/\alpha}\right).\label{bn psi exp}
\end{align}
\end{thm}
\begin{Rem}
If $L\left(n\right)=1$, then the norming sequence simplifies to 
\begin{align*}
d_n\sim\frac{\alpha}{1-\alpha}\cdot n^{1/\alpha}\cdot b_{n}^{1-1/\alpha}.
\end{align*}
\end{Rem}

\begin{Rem}\label{Rem: log psi log n}
The condition on the trimming sequence \eqref{eq: t cond a2} seems very technical at the beginning.
If we choose $\psi(n)=n^2$, then $\psi\in\Psi$ and 
$\log\psi\left(\left\lfloor \log n\right\rfloor \right)=2\log\left\lfloor \log n\right\rfloor$ and thus \eqref{eq: t cond a2}
can also be written as $\lim_{n\to\infty}b_n/\log\log n=\infty$. 

However, this more general condition allows us to also consider $\left(b_n\right)$ which are more complicated and particularly not monotonic.
For example one can set
\begin{align*}
 \psi(n)=\begin{cases}
          \left(\log_{2} n\right)^2&\text{ if }n=2^k\text{, }k\in\N\\
          n^2&\text{ else,}
         \end{cases}
\end{align*}
and consider the sequence $(b_n)$ fulfilling 
$\lim_{k\to\infty}b_{2^{2^k}}/\log k=\infty$ for $k\in\N$ and, additionally,
$\lim^\star_{n\to\infty}b_{n}/\log\log n=\infty$, where the last limit $\lim^\star$ reaches  over all $n\in\N$ 
not taking the values $n=2^{2^k}$ with $k\in\N$.
Then $(b_n)$ fulfills condition \eqref{eq: t cond a2} as well.

For the case of i.i.d.\ summands it follows from \cite{haeusler_nonstandard_1993} that the convergence in 
\eqref{eq: lim Snbnchi} also implies $\lim_{n\to\infty}b_n/\log\log n=\infty$. 
This implies that our bound is nearly optimal. 
\end{Rem}

\subsection{Example: Piecewise expanding interval maps}\label{Examples}
For our main example, we first define the space of functions of bounded variation. 
For simplicity, we restrict ourself to the interval $\left[0,1\right]$ and let $\mathcal{B}$ denote the Borel sets of $\left[0,1\right]$. 
\begin{Def}
Let $\varphi:\left[0,1\right]\to\mathbb{R}$. The variation $\mathsf{var}\left(\varphi\right)$ of $\varphi$ is given by 
\begin{align*}
\mathsf{var}\left(\varphi\right)\coloneqq\sup\left\{\sum_{i=1}^n\left|\varphi\left(x_i\right)-\varphi\left(x_{i-1}\right)\right|\colon n\geq 1, x_i\in\left[0,1\right], x_0<x_1<\ldots<x_n\right\}
\end{align*}
and we define
\begin{align*}
\mathsf{V}\left(\varphi\right)\coloneqq  \inf\left\{\mathsf{var}\left(\varphi'\right)\colon \varphi'\text{ is a version of }\varphi\right\}.
\end{align*}
By $BV$ we denote the Banach space of functions of bounded variation, i.e.\ of functions $\varphi$ 
fulfilling $\mathsf{V}\left(\varphi\right)<\infty$. It is equipped with the norm
$\left\|\varphi\right\|_{BV}\coloneqq\left|\varphi\right|_{\infty}+\mathsf{V}\left(\varphi\right)$.
\end{Def}
For further properties of functions of bounded variation see e.g. \cite[Chapter 2]{boyarsky_laws_1997}.

\begin{Prop}\label{prop: Ex1}
Let $\Omega'\subset \left[0,1\right]$ be a dense and open set such that $\mu\left(\Omega'\right)=1$ and
let $\mathcal{I}\coloneqq\left(I_n\right)_{n\in\mathbb{N}}$ be finite or countable family
of closed intervals with disjoint interiors and for any $I_j$ such that the set $I_j\cap \left(\left[0,1\right]\backslash \Omega'\right)$ 
consists exactly of the endpoints of $I_j$.

Then there exists a probability measure $\mu$ absolutely continues with respect to the Lebesgue measure on $\left[0,1\right]$ such that 
$\left(\left[0,1\right], \mathcal{B},\mu, T, BV,\left\|\cdot \right\|_{BV}\right)$ fulfills Property $\mathfrak{C}$
if $T$ fulfills the following properties
\begin{itemize}
\item (Uniform expansion)
$ T_n\coloneqq T\lvert_{\mathring{I}_n}\in \mathcal{C}^1$, and $\left| T_n'\right|\geq m>1$ for any $n\in\mathbb{N}$.
\item
$ T$ is topologically mixing.
\item 
If we set $g\left(x\right)\coloneqq1/\left| T'\left(x\right)\right|$, then $g\lvert_{\mathring{I}_i}$ is a function of bounded variation for all $n\in\mathbb{N}$.
\end{itemize}
\end{Prop}
We note here that this example mainly relies on results in \cite{rychlik_bounded_1983} on piecewise expanding interval maps on countable partitions generalizing \cite{lasota_existence_1973} where  finite partitions are considered.

Due to \cite{zweimueller_ergodic_1998} the above properties are fulfilled for an infinite partition $\left(I_n\right)_{n\in\mathbb{N}}$ and the absolutely continuous measure $\mu$ is finite if
\begin{itemize}
 \item (Adler's condition) $ T_i\coloneqq T\lvert_{\mathring{I}_i}\in \mathcal{C}^2$ and
 $ T''/\left( T'\right)^2$ is bounded on $\Omega'$.
 \item (Finite image condition) $\#\left\{ T I_n\colon I_n\in\mathcal{I}\right\}<\infty$.
 \item (Uniform expansion)
 $ T_n\coloneqq T\lvert_{\mathring{I}_n}\in \mathcal{C}^1$, and $\left| T_n'\right|\geq m>1$ for any $n\in\mathbb{N}$.
 \item $ T$ is topologically mixing. 
\end{itemize}

\begin{Prop}\label{prop: Ex2}
The system $\left(\left[0,1\right], \mathcal{B},\mu, T, BV,\left\|\cdot \right\|_{BV},\chi\right)$
with $\mu$ and $T$ 
from Proposition \ref{prop: Ex1}
fulfills Property $\mathfrak{D}$
with $\chi:\left[0,1\right]\to\mathbb{R}_{\geq 0}$ if there exist $\widetilde{K}_1, \widetilde{K}_2>0$ such that for all $\ell\in\mathbb{R}_{\geq 0}$
\begin{align}
\mathsf{V}\left(\upperleft{\ell}{}{\chi}\right)\leq \widetilde{K}_1\cdot {\ell}\label{cond C 1}
\end{align}
and 
\begin{align}
\mathsf{V}\left(\mathbbm{1}_{\left\{\chi>\ell\right\}}\right)\leq \widetilde{K}_2.\label{cond C 2}
\end{align}
\end{Prop}
Setting for example $\chi$ monotonically increasing or decreasing implies that \eqref{cond C 1} and \eqref{cond C 2}
are fulfilled. 
Giving then a system $\left(\left[0,1\right], \mathcal{B},\mu, T, BV,\left\|\cdot \right\|_{BV}\right)$ 
which fulfills Property $\mathfrak{C}$ enables us to apply the theorems from the previous two sections.

\begin{Rem}\label{rem: doubling map ex}
We note here that Property $\mathfrak{D}$ 
includes examples with a mixing structure as in 
\cite{haynes_quantitative_2014} not fulfilling a lightly trimmed strong law. 
Haynes investigated the dynamical system $\left(\left[0,1\right), \mathcal{B}, \lambda_{[0,1)}, T\right)$ with
$Tx=2x\mod 1$ and the observable $\chi(x)= \left\lfloor 1/x\right\rfloor$ 
which does not fulfill a lightly trimmed strong law.

We might compare this system to i.i.d.\ random variables with the same distribution function as the above system
and the $\bm{\psi}$-mixing continued fraction system mentioned in the introduction 
which also has the same distribution function with respect to the Lebesgue measure.
Then we obtain that both systems 
fulfill a lightly trimmed strong law after only removing the maximal term.
The i.i.d.\ case follows after an easy calculation using the results in \cite{mori_stability_1977}.

However, if we take the same system $\left(\left[0,1\right), \mathcal{B}, \lambda_{[0,1)}, T\right)$ but alter $\chi$ 
now being defined as $\chi(x)=\lfloor 1/x^{1/\alpha}\rfloor$ with $0<\alpha<1$, then the system
$\left(\left[0,1\right],\mathcal{B},\lambda\lvert_{\left[0,1\right)},T,BV,\left\|\cdot\right\|_{BV},\chi\right)$
fulfills Property $\mathfrak{D}$, see the above propositions.
In this case we have $\left(y+1\right)^{-\alpha}\leq \mu\left(\chi>x\right)\leq y^{-\alpha}$, i.e.\
$\mu\left(\chi>x\right)$ is regularly varying with index $-\alpha$.
Applying Theorem \ref{Sb(n)} gives that for this system an intermediately trimmed strong law holds.
By the subsequent Remark \ref{Rem: log psi log n} the optimal trimming sequence $(b_n)$ 
also does not change by the fact that the random variables are not independent.

So it seems that intermediately trimmed strong laws are less susceptible of dependence structures than lightly trimmed strong laws. 
\end{Rem}

\section{Main ideas of proofs}\label{sec: main ideas proof}
All our trimming theorems, i.e.\ 
Theorem \ref{find bn} 
with its Corollary \ref{S b(n) immer} and Theorem \ref{Sb(n)} for regularly varying tails
follow the same idea of proof which will be given in the following. 
The remaining parts of the proof will then be given in Section \ref{proofs transfer prelim}.

\subsection{Sketch of proof for almost sure limit theorems}\label{subsec: sketch proof}
In order to prove an intermediately trimmed strong law
we use two main properties discussed in the following and denoted as Properties $\bm{A}$ and $\bm{B}$.
We will first state the properties and then in Lemma \ref{lem: Prop A B to trimming} state how 
together with an additional condition they imply an intermediately trimmed strong law.

In the subsequent sections, i.e.\ Section \ref{subsec: intro truncated rv}
and Section \ref{subsec: intro large dev}
we will give an idea how to prove Properties $\bm{A}$ and $\bm{B}$, 
 what distinguishes them from the case of i.i.d.\ random variables
 and why it is difficult to directly use results about mixing random variables 
 instead of random variables fulfilling Property $\mathfrak{D}$, 
 see Definition \ref{def: Prop D}, i.e.\ fulfilling a spectral gap property.

The first property considers the sum of truncated random variables.
Namely, for $\chi:\Omega\to\mathbb{R}_{\geq 0}$ we define 
\begin{align*}
\prescript{\ell}{}{\chi}\coloneqq\mathbbm{1}_{\left\{\chi\leq \ell\right\}}\cdot\chi,
\end{align*}
for all $\ell\in\mathbb{R}_{\geq 0}$ and for a real valued sequence 
$\left(f_{n}\right)_{n\in\mathbb{N}}$ we let 
\begin{align*}
\mathsf{T}_n^{f_n}\chi&\coloneqq\upperleft{f_n}{}{\chi}+\upperleft{f_n}{}{\chi}\circ  T+\ldots +\upperleft{f_n}{}{\chi}\circ  T^{n-1}
\end{align*}
denote the corresponding truncated sum process.
\begin{Def}
 Let $\left(\Omega, \mathcal{A}, T , \mu,\mathcal{F},\left\|\cdot\right\|,\chi\right)$ fulfill Property $\mathfrak{D}$.
 We say that $(f_n)$ fulfills Property $\bm{A}$ for the system $\left(\Omega, \mathcal{A}, T , \mu,\mathcal{F},\left\|\cdot\right\|,\chi\right)$ if 
\begin{align}
 \lim_{n\to\infty}\frac{\mathsf{T}_n^{f_n}\chi}{\int\mathsf{T}_n^{f_n}\chi\mathrm{d}\mu}=1\text{ a.s.}\label{eq: Tn/ETn intro}
\end{align}
\end{Def}

The second property deals with the average number of large entries and is defined as follows.
\begin{Def}
 Let $\left(\Omega, \mathcal{A}, T , \mu,\mathcal{F},\left\|\cdot\right\|,\chi\right)$ fulfill Property $\mathfrak{D}$.
 We say that a tuple $((f_n),(\gamma_n))$ fulfills Property $\bm{B}$ for the system $\left(\Omega, \mathcal{A}, T , \mu,\mathcal{F},\left\|\cdot\right\|,\chi\right)$ if 
\begin{align*}
 \mu\left(\left|\sum_{i=1}^{n}\left(\mathbbm{1}_{\left\{\chi> f_n\right\}}\circ T^{i-1}-\mu\left(\chi>f_n\right)\right)\right|\geq \gamma_n\text{ i.o.}\right)=0
\end{align*}
\end{Def}

In the following lemma we show how these two properties together with condition \eqref{eq: fn gamman} give an 
intermediately trimmed strong law.

\begin{lemma}\label{lem: Prop A B to trimming}
Let $\left(\Omega, \mathcal{A}, T , \mu,\mathcal{F},\left\|\cdot\right\|,\chi\right)$ which fulfill Property $\mathfrak{D}$.
Further, let $(f_n)$ fulfill Property $\bm{A}$ and let $((f_n),(\gamma_n))$ fulfill Property $\bm{B}$ 
 for the system $\left(\Omega, \mathcal{A}, T , \mu,\mathcal{F},\left\|\cdot\right\|,\chi\right)$.
 If additionally
 \begin{align}
  \lim_{n\to\infty}\frac{\gamma_n\cdot f_n}{\int\mathsf{T}_n^{f_n}\chi\mathrm{d}\mu}=0\label{eq: fn gamman}
 \end{align}
 holds, then we have for $b_n\coloneqq \left\lceil n\cdot \mu\left(\chi>f_n\right)+\gamma_n\right\rceil$ that
\begin{align*}
 \lim_{n\to\infty}\frac{\mathsf{S}_n^{b_n}\chi}{\int\mathsf{T}_n^{f_n}\chi\mathrm{d}\mu}=1\text{ a.s.}
\end{align*}
\end{lemma}
\begin{proof}
We can conclude from Property $\bm{B}$ that a.s. 
\begin{align}
\mathsf{S}_{n}^{b_n}\chi\leq \mathsf{T}_{n}^{f_{n}}\chi\text{ eventually.}\label{San Tn}
\end{align}

On the other hand since $\upperleft{f_n}{}{\chi}\leq f_{n}$ it follows by
Property $\bm{B}$ that a.s. 
\begin{align}
\mathsf{T}_{n}^{f_{n}}\chi-2\gamma_n f_{n}\leq \mathsf{S}_{n}^{b_n}\chi\text{ eventually.}\label{Sbn >}
\end{align}

Combining \eqref{San Tn} and \eqref{Sbn >} yields that a.s.\
\begin{align*}
 \lim_{n\to\infty}\frac{\mathsf{T}_{n}^{f_{n}}\chi-2\gamma_n f_{n}}{\int\mathsf{T}_{n}^{f_{n}}\chi\mathrm{d}\mu}
 \leq \lim_{n\to\infty}\frac{\mathsf{S}_{n}^{b_n}\chi}{\int\mathsf{T}_{n}^{f_{n}}\chi\mathrm{d}\mu}
 \leq \lim_{n\to\infty}\frac{\mathsf{T}_{n}^{f_{n}}\chi}{\int\mathsf{T}_{n}^{f_{n}}\chi\mathrm{d}\mu}\text{ eventually.}
\end{align*}
Combining this with Property $\bm{A}$ and \eqref{eq: fn gamman} gives the statement of the lemma.
\end{proof}

\subsection{Statements about truncated random variables (Property A)}\label{subsec: intro truncated rv} 
We will first state two results giving an answer for which sequences $(f_n)$ Property $\bm{A}$ holds. 
Recall that we denote by $F$ the distribution function of $\chi$ with respect to  $\mu$, i.e.\
$F(x)=\mu(\chi\leq x)$.
\begin{thm}
\label{Thm: Sn* allg} 
Let $\left(\Omega, \mathcal{A},  T, \mu,\mathcal{F},\left\|\cdot\right\|,\chi\right)$
fulfill Property $\mathfrak{D}$.
For a positive valued sequence $\left(f_{n}\right)_{n\in\mathbb{N}}$
assume $F\left(f_{n}\right)>0$ for all $n\in\mathbb{N}$ and there
exists $\psi\in\Psi$ such that 
\begin{align}
\frac{f_{n}}{\int\mathsf{T}_n^{f_n}\chi\mathrm{d}\mu}=o\left(\frac{n}{\log\psi\left(n\right)}\right)\label{cond a}
\end{align}
holds. Then 
\begin{align*}
\lim_{n\rightarrow\infty}\frac{\mathsf{T}_{n}^{f_{n}}\chi}{\int\mathsf{T}_n^{f_n}\chi\mathrm{d}\mu}=1\text{ a.s.}
\end{align*}
\end{thm}

\begin{thm}
\label{Thm: Sn* reg var} 
Let $\left(\Omega, \mathcal{A},  T, \mu,\mathcal{F},\left\|\cdot\right\|,\chi\right)$
fulfill Property $\mathfrak{D}$.
Assume that $F\left(x\right)=L\left(x\right)/x^{\alpha}$ where $0<\alpha<1$ and $L$ is slowly varying. 
Let $\left(f_{n}\right)_{n\in\mathbb{N}}$
be a positive valued sequence and assume that $F\left(f_{n}\right)>0$,
for all $n\in\mathbb{N}$. 
If there exists $\psi\in\Psi$ such that 
\begin{align}
\frac{f_{n}^{\alpha}}{L\left(f_{n}\right)}=o\left(\frac{n}{\log\psi\left(\left\lfloor \log n\right\rfloor \right)}\right),\label{cond}
\end{align}
then we have 
\begin{align}
\lim_{n\rightarrow\infty}\frac{\mathsf{T}_{n}^{f_{n}}\chi}{\int\mathsf{T}_n^{f_n}\chi\mathrm{d}\mu}=1\text{ a.s.}\label{eq: lim Tntn reg var}
\end{align}
\end{thm}
\begin{Rem}
If additionally $f_{n}$ tends to infinity, we have the more explicit
statement that 
\begin{align*}
\lim_{n\rightarrow\infty}\frac{\mathsf{T}_{n}^{f_{n}}\chi}{\frac{\alpha}{1-\alpha}\cdot n\cdot f_{n}^{1-\alpha}\cdot L\left(f_{n}\right)}=1\text{ a.s.}
\end{align*}
This can be easily obtained from Lemma \ref{E V X*}.
\end{Rem}

\begin{Rem}
The above result also improves a result by Nakata for i.i.d.\ random variables, see \cite{nakata_limit_2015}.
The results in \cite[Theorem 1.2 (ii)]{nakata_limit_2015}
are slightly more general in the sense that he considers two sequences $(l_n)$ and $(c_n)$ both depending 
on $n$, where $(c_n)$ denotes the truncation, in our case denoted by $(f_n)$ and $(l_n)$ denotes the summation index 
which we alway set equal to $n$.

However, his results are restricted to the case that $1-F\left(x\right)\asymp x^{-\alpha}$, where $0<\alpha<1$ and $f(x)\asymp g(x)$ denotes
the existence of $C>0$ such that $1/Cg(x)\leq f(x)\leq C g(x)$, for all $x$.

One example for this setting is $F\left(x\right)=1-1/x^{\alpha}$ with $0<\alpha<1$.
For this example his condition for \eqref{eq: lim Tntn reg var} to hold can in our notation be written as
\begin{align*}
 f_n^{\alpha}=o\left(\frac{n}{\log n}\right).
\end{align*}
Comparing this with the condition that there exists $\psi\in\Psi$ such that
\eqref{cond} holds with $L$ constant, shows that \eqref{cond} is indeed a weaker condition.
\end{Rem}

\subsection{A statement about a large deviation result (Property B)}\label{subsec: intro large dev}
The following lemma gives us a statement to determine under which conditions Property $\bm{B}$ holds.
The formulation is very general so that it enables us to later prove trimming statements 
for different settings.
\begin{lemma}\label{bernoulli}
Set
\begin{align}
c\left(k,n\right)\coloneqq c_{\epsilon,\psi}\left(k,n\right)\coloneqq\left(\max\left\{ k,\log\psi\left(\left\lfloor \log n\right\rfloor \right)\right\} \right)^{1/2+\epsilon}\cdot\left(\log\psi\left(\left\lfloor \log n\right\rfloor \right)\right)^{1/2-\epsilon},\label{c(n)}
\end{align}
for $k\in\mathbb{R}_{\geq1}$, $n\in\mathbb{N}_{\geq3}$,  $0<\epsilon<1/4$,
 and $\psi\in\Psi$.
There exists a constant  $V>0$ such that for all  $ \epsilon\in (0,1/4)$, $\psi\in\Psi$, and  all positive valued sequences $\left(u_n\right)_{n\in\mathbb{N}}$ and $p_n\coloneqq\mu\left(\chi> u_n\right)$, $n\in \mathbb{N}$, we have
\begin{align}
\mu\left(\left\{ \left|\mathsf{S}_n\mathbbm{1}_{\left\{\chi>u_n\right\}}-p_n\cdot n\right|\geq V\cdot c_{\epsilon,\psi}\left(p_{n}\cdot n,n\right)\text{ i.o.}\right\} \right)=0.\label{pn B}
\end{align}
\end{lemma}

We note here that $\mathsf{S}_n\mathbbm{1}_{\left\{\chi>u_n\right\}}=\sum_{k=1}^n\mathbbm{1}_{\left\{\chi\circ T^{k-1}>u_n\right\}}$.
The lemma is an analogous statement to \cite[Theorem 8]{kessebohmer_strong_2016} for the setting of dynamical systems instead of independent random variables.

\subsection{Comments about the method of proof}\label{subsubsec: comments about method}
The method to first prove a statement as in Property $\bm{A}$ in order to prove an intermediately trimmed strong law
coincides with the method used for many trimming statements in the i.i.d.\ setting, for instance in \cite[Last part of the proof of Lemma 2]{haeusler_laws_1987} and \cite[Lemma 6]{mori_stability_1977}. 

The method to prove trimming results with a large deviation statement as in Property $\bm{B}$ 
is the same as in \cite{mori_stability_1977} and \cite{kessebohmer_strong_2016}, 
but differs for example from \cite{haeusler_laws_1987}
who used a quantile transformation. 
However, the quantile transformation heavily relies on the independence of the random variables, 
so it does not seem reasonable to transfer this method to our case.

The main idea for proving that $\left(f_n\right)$ fulfills Property $\bm{A}$
is to use one version of an exponential inequality, in the i.i.d.\ case usually the Bernstein inequality, to estimate
\begin{align}
 \mu\left(\left|\mathsf{T}_n^{f_n}\chi-\int\mathsf{T}_n^{f_n}\chi\mathrm{d}\mu\right|> \epsilon \int\mathsf{T}_n^{f_n}\chi\mathrm{d}\mu\right).\label{eq: mu(Tn -ETn) intro}
\end{align}
Using in the following a Borel-Cantelli argument for the right choice of $(f_n)$ gives a statement as in \eqref{eq: Tn/ETn intro}.
Similarly, 
to prove that a tuple $((f_n),(\gamma_n))$ fulfills Property $\bm{B}$ one can 
use an exponential inequality estimating
\begin{align*}
 \mu\left(\left|\mathsf{S}_n\mathbbm{1}_{\left\{\chi>f_n\right\}}-\mu\left(\chi>f_n\right)\cdot n\right|\geq \gamma_n\right)
\end{align*}
in order to apply a Borel-Cantelli statement afterwards. 

In particular for proving a Property $\bm{A}$ statement it is difficult 
to use an exponential inequality for mixing random variables. 
To explain this we will in the following distinguish the notation between random variables from a dynamical system, 
writing $\mathsf{S}_{n}\chi$, $\mathsf{T}^r_{n}\chi$, $\mu$, $(\chi\circ T^{n-1})$ and $\int \chi\mathrm{d}\mu$
and generic random variables, for example i.i.d.\ writing 
$S_n$, $T_n^r$, $\mathbb{P}$, $(X_n)$, and $\mathbb{E}\left(X_1\right)$.

 The first difficulty is that exponential inequalities usually 
 depend on the maximal bound of the considered random variables. In our case we have that $\left|\prescript{f_n}{}{\chi}\right|_\infty=f_n$
 and $f_n$ is depending on $n$.
 Using a direct application of the Bernstein inequality as in \cite[Lemma 10]{kessebohmer_strong_2016} we obtain for the i.i.d.\ case that 
 \begin{align}
 \mathbb{P}\left(\left|T_n^{f_n}-\mathbb{E}\left(T_n^{f_n}\right)\right|> \epsilon \mathbb{E}\left(T_n^{f_n}\right)\right)
 &\leq \exp\left(-\frac{3\epsilon^2}{6+2\epsilon}\cdot \frac{\mathbb{E}\left(T_n^{f_n}\right)}{f_n}\right).\label{eq: Bernstein ineq truncated sum}
 \end{align}
 The easiest way to prove a statement as in \eqref{eq: Tn/ETn intro} is then to determine $\left(f_n\right)$ such that
 \begin{align}
  \sum_{n=1}^{\infty}\exp\left(-\frac{3\epsilon^2}{6+2\epsilon}\cdot \frac{\mathbb{E}\left(T_n^{f_n}\right)}{f_n}\right)<\infty.\label{eq: sum Bernstein}
 \end{align}
 In order to be precise we have to remark here that in our proof we actually use a different sum given in \eqref{eq: sum for BC}. 
 However, the problem explained in the following does not change by considering the sum in \eqref{eq: sum Bernstein} instead.

 In case that we have mixing random variables the situation becomes more difficult than for independent random variables. 
 Even though there is a vast literature about exponential inequalities for mixing random variables, 
 it is difficult to obtain results as strong as in the i.i.d.\ case using them. 
 We will explain the difficulties by explaining some of the more recent results.
 Many exponential inequalities are given with a constant depending on the distribution function of the considered random variable and, consequently, depend on the bounds $\left(f_n\right)$ and the dependence can not be easily determined 
 by the proof of these inequalities, this is for example the case in \cite{merlevede_bernstein_2011} and some statements in \cite{doukhan_mixing:_1994}.
 For instance, \cite[Theorem 1]{merlevede_bernstein_2011} states that under some regularity conditions on the random variables $(X_n)$
 there exist constants $\gamma$ and $C_1$ such that 
 \begin{align*}
  \mathbb{P}\left(\sup_{j\leq n} \left|S_j\right|>x\right)\leq n\cdot \exp\left(-\frac{x^{\gamma}}{C_1}\right)+\text{ further summands}.
 \end{align*}
 In their theorem, $\gamma$ can be calculated by the distribution function of $X_1$ and the mixing properties of $(X_n)$,
 but $C_1$ is dependent on $\gamma$ which changes if one considers $S_n^{f_n}$ with $f_n$ not constant.

  Other results, for example \cite{adamczak_tail_2008}, require a more restricted setting than considered here, 
  covering only Markov chains.

 In other publications as in \cite{merlevede_bernstein_2009} the bound depending on $\left(f_n\right)$ is given, but the obtained results 
 are weaker than in the i.i.d.\ case. Theorem 2 of this publication states that for exponentially $\bm{\alpha}$-mixing identically distributed random variables $(X_n)$ there exists a constant $C>0$ such that for all $n\in\N$ and $x\geq 0$
 \begin{align*}
  \mathbb{P}\left(\left|S_n-\mathbb{E}\left(S_n\right)\right|>x\right)\leq \exp\left(-\frac{C\cdot x^2}{n\cdot v^2+M^2+xM\left(\log n\right)^2}\right),
 \end{align*}
 where $\left|X_i\right|\leq M$ and $v^2=\sup_{i>0}\left(\mathbb{V}\left(X_i\right)+2\sum_{j>i}\left|\mathrm{Cov}\left(X_i,X_j\right)\right|\right)>0$ 
 and $\mathbb{V}(Y)$ denotes the variance of $Y$.
 Applying this to our case this result yields 
 \begin{align*}
  \mu\left(\left|\mathsf{T}_n^{f_n}\chi-\int \mathsf{T}_n^{f_n}\chi\mathrm{d}\mu\right|>\epsilon\cdot \int \mathsf{T}_n^{f_n}\chi\mathrm{d}\mu\right)
  &\leq \exp\left(-\frac{C\cdot \left(\epsilon\cdot \int \mathsf{T}_n^{f_n}\chi\mathrm{d}\mu\right)^2}{n\cdot v^2+f_n^2+\epsilon\cdot\int \mathsf{T}_n^{f_n}\chi\mathrm{d}\mu\cdot f_n\left(\log n\right)^2}\right).
 \end{align*}
 For the example $F\left(x\right)=1-1/x^{\alpha}$ with $0<\alpha<1$ we have by \eqref{E S*} that 
 \begin{align*}
  \exp\left(-\frac{C\cdot \left(\epsilon\cdot \int\mathsf{T}_n^{f_n}\chi\mathrm{d}\mu\right)^2}{n\cdot v^2+f_n^2+\epsilon\cdot \int \mathsf{T}_n^{f_n}\chi\mathrm{d}\mu\cdot f_n\left(\log n\right)^2}\right)
  &\geq \exp\left(-\frac{C\cdot \left(\epsilon\cdot \int\mathsf{T}_n^{f_n}\chi\mathrm{d}\mu\right)^2}{f_n^2}\right)\\
  &\geq \exp\left(-\frac{C\cdot \epsilon^2\cdot \alpha^2}{\left(1-\alpha\right)^2}\cdot\frac{n}{f_n^{2\alpha}}\right).
 \end{align*}
 As a comparison, in the i.i.d.\ case we obtain by the Bernstein inequality \eqref{eq: Bernstein ineq truncated sum}
 for the same distribution function
 \begin{align*}
  \mathbb{P}\left(\left|T_n^{f_n}-\mathbb{E}\left(T_n^{f_n}\right)\right|>\epsilon\cdot \mathbb{E}\left(T_n^{f_n}\right)\right)
  &\leq \exp\left(-\frac{3\epsilon^2}{6+2\epsilon}\cdot \frac{\alpha}{1-\alpha}\cdot\frac{n}{f_n^{\alpha}}\right).
 \end{align*}
 This difference in the denominator changes the set of sequences $(f_n)$ which fulfill \eqref{eq: sum Bernstein} a lot
 and thus also statements about the set of sequences $(f_n)$ fulfilling Property $\bm{A}$.
 Similar problems occur by applying other theorems in this publication.
 
 Other statements, see for example \cite[Theorem 4, Chapter 1.4.2]{doukhan_mixing:_1994}, give an exponential inequality with an additional summand depending on the speed of mixing. 
 For $\left(X_n\right)_{n\in\mathbb{N}}$ being a sequence of non-negative, $\bm{\beta}$-mixing (a slightly weaker form of mixing than $\bm{\phi}_{\textrm{rev}}$-mixing), identically distributed 
 random variables with $\lim_{n\rightarrow\infty}\bm{\beta}\left(n\right)\cdot n=0$ 
Doukhan states for any $0<\delta<1$ and $0\leq\eta\leq n/\left(1+\delta^2/4\right)$ that
\begin{align}
\mathbb{P}\left(\left|S_n-\mathbb{E}\left(S_n\right)\right|\geq x\right)
&\leq 4\exp\left(-\frac{\left(1-\delta\right)\cdot x^2}{2\left(n\cdot \sigma^2+\frac{1}{3}\cdot \eta\cdot M\cdot x\right)}\right)+2\cdot\frac{n\cdot{\bm \beta}\left(\left\lfloor \eta\cdot \delta^2/4\right\rfloor-1\right)}{\eta},\label{eq: m mix bernstein 0}
\end{align} 
where $\left|X_1\right|\leq M$ and $\sigma^2\geq \mathbb{V}\left(\sum_{k=i}^{i+j-1}Y_k\right)/j$ for all $i,j\in\mathbb{N}$.
 
In order to obtain a summable expression for an estimate of \eqref{eq: mu(Tn -ETn) intro}
one would have to guarantee that both summands in \eqref{eq: m mix bernstein 0} are summable.
For the second summand this can only be ensured if $\eta_n\coloneqq \eta$ is tending to infinity for $n$ tending to infinity.
However, this increases the first summand,
leading to worse results than in the i.i.d.\ case.

Further, we want to point out that Adamczak gives an example in \cite{adamczak_tail_2008} that for observables on an infinite state aperiodic and irreducible Markov chain 
the estimation can not be as good as in the independent case in that sense that the expression in the exponent differs by more than only a constant.
 Even though it is not clear if this example fulfills a $\bm{\phi}_{\textrm{rev}}$-mixing property
 and it does not fulfill a $\bm{\phi}$-mixing property, by \cite[Theorem 3.2]{bradley_basic_2005} this Markov chain is at least exponentially $\bm{\beta}$-mixing.
 From this consideration it might be reasonable to assume 
 that the lack of independence is a reason that exponential inequalities for $\bm{\phi}_{\textrm{rev}}$-mixing random variables
 can not be formulated with the same sharpness as in the i.i.d.\ case.
\medskip

 Instead of using results for mixing random variables
 we prove an exponential inequality for dynamical systems fulfilling Property $\mathfrak{C}$, stated as follows:
 \begin{lemma}\label{lemma: Tnfn chi deviation allg}
Let $\left(\Omega, \mathcal{A},  T, \mu,\mathcal{F},\left\|\cdot\right\|\right)$ fulfill Property $\mathfrak{C}$.
Then there exist positive constants $K$, $N$, $U$
such that for all 
$\varphi\in\mathcal{F}$ fulfilling $\int\varphi\mathrm{d}\mu=0$, all 
$u\in\mathbb{R}_{>0}$, 
and all $n\in\N_{>N}$ we have
\begin{align*}
\mu\left(\max_{i\leq n}\left|\mathsf{S}_i\varphi\right|\geq  u\right)
&\leq K\cdot \exp\left(-U\cdot \frac{ u}{\left\|\varphi\right\|}\cdot\min\left\{\frac{u}{n\cdot\left|\varphi\right|_1},1\right\}\right).
\end{align*}
\end{lemma}

 Additionally to their mixing properties the random variables $\left(\varphi\circ T^{n-1}\right)$ obtain additional structures which facilitate to obtain an exponential inequality.
 More precisely, we first use Markov's inequality to obtain the statement in Lemma \ref{lemma 2 2nd part}. 
 The subsequent estimation of $\int\exp\left(t\cdot \mathsf{S}_n\varphi\right)\mathrm{d}\mu$ seems to give better results using analytic perturbation theory 
 than using blocking techniques for mixing random variables.
 
 The main idea here is to use perturbation theory for the perturbed transfer operator 
 $\widehat{ T}_zh\coloneqq\upperleft{\varphi}{\widehat}{ T_z}h\coloneqq\widehat{ T}\left(\mathrm{e}^{z\cdot \varphi}\cdot h\right)$.
 This gives us $\int\mathrm{e}^{t\cdot \mathsf{S}_n\varphi}\mathrm{d}\mu=\int \widehat{ T}_t^n 1\mathrm{d}\mu$,
 see the proof of Lemma \ref{e Snphi}.
 We assume that $\widehat{T}$ has a spectral gap and prove in Lemma \ref{spect sep analyt} that this is also true for $\upperleft{\varphi}{\widehat}{ T_t}$
 if $\left\|\varphi\right\|\cdot t$ is small enough. 
 This will eventually lead us to the result that $\int \widehat{ T}_t^n 1\mathrm{d}\mu
 \leq K'\cdot \left|\upperleft{\varphi}{}{\lambda}_t\right|^n$
 with $\lambda_t\coloneqq \upperleft{\varphi}{}{\lambda}_t$ being the largest eigenvalue of the operator $\upperleft{\varphi}{\widehat}{ T}_t$ and $K'$ a constant,
 see Lemma \ref{e Snphi}.
 Finally, we estimate $\lambda_t$ using the Taylor expansion $\sum_{k=0}^{\infty}t^k\cdot\lambda_0^{\left(k\right)}/k!$, see Lemma \ref{lambda^n}. 
 
 The proof of Lemma \ref{lemma: Tnfn chi deviation allg} seems rather lengthy 
 compared to the proof of limit theorems using analytic perturbation theory, for example the central limit theorem. 
 One reason for that is that in Lemma \ref{spect sep analyt} we need a uniform bound of $\left\|\varphi\right\|\cdot t$
 for $\upperleft{\varphi}{\widehat}{ T_t}$ to fulfill a spectral gap property. Such results are usually only developed for 
 $\upperleft{\varphi}{\widehat}{ T_t}$ for a fixed $\varphi$. 
 A similar problem occurs in Lemma \ref{lemma: resolvent}. 
 
 Furthermore, in order to obtain optimal results we indeed have to estimate $\lambda_0^{(k)}$, for all $k\in\N$, see Lemma \ref{lambda^n},
 not stopping after a certain term of the Taylor expansion.

Additionally, we prove this exponential inequality as a maximal inequality. 
Even though there is literature generalizing exponential inequalities to maximal exponential inequalities, see \cite{kevei_more_2013},
we obtain better results using a direct approach, i.e.\
we transform the sum $\mathsf{S}_n\varphi$ into a martingale and apply a generalization of Doob's inequality, 
see Lemma \ref{doob ineq}.
\medskip

With respect to Lemma \ref{bernoulli} we remark that the considered random variables are all bounded by $1$ 
and it might therefore be possible to directly use exponential inequalities stated for mixing random variables.
However, for brevity we used Lemma \ref{lemma: Tnfn chi deviation allg} again.
The obtained results only differ from the results for i.i.d.\ random variables 
in the respect that in the i.i.d.\ case the constant $V$ in Lemma \ref{bernoulli} can be given explicitly, 
whereas in our case the constant depends on some inherent properties of the dynamical system.
\medskip

 The second difficulty is that in order to prove optimal results we 
 use some inclusions in the proof of Theorem \ref{Thm: Sn* reg var} as well as in the proof of Lemma \ref{bernoulli}, i.e.\
 we consider in fact a more complicated term than \eqref{eq: mu(Tn -ETn) intro}.
 The method is mentioned as \ref{en: trunc A} and \ref{en: trunc B} on p.\ \pageref{en: trunc A}
 for the proof of Theorem \ref{Thm: Sn* reg var} and as \ref{en: large dev a} and \ref{en: large dev b} on p.\ \pageref{en: large dev a} 
 for the proof of Lemma \ref{bernoulli}.
 This enables us to use the Borel-Cantelli lemma for an exponential subsequence leading to stronger results.

 The method used in \cite{haeusler_laws_1987} is slightly shorter, however it requires $\left(f_n\right)$ to be additionally monotonic.
 With this result it is only possible to prove trimming statements with $\left(b_n\right)$ being asymptotic to an increasing sequence.
 Sequences $\left(b_n\right)$ as considered in Remark \ref{Rem: log psi log n} are not covered in their results,
 since following the proof of Theorem \ref{Sb(n)} gives $b_n\sim n\cdot \mu\left(\chi>f_n\right)$. Thus,
 $f_n$ being monotonic does not allow for large oscillations of $b_n$.

\subsection{Structure of the paper}
We will give some preliminary statements in Section \ref{sec: Intro statements}, 
i.e.\ we will give the definition of a spectral gap and show the resulting decay of correlation in Section \ref{subsec: spec gap corr decay}.
In Section \ref{subsec: analytic perturbation} we give an introduction 
into analytic perturbation theory for the operator $\widehat{ T}_z$ defined as
$\widehat{ T}_zh\coloneqq\upperleft{\varphi}{\widehat}{ T_z}h\coloneqq\widehat{ T}\left(\mathrm{e}^{z\cdot \varphi}\cdot h\right)$.
We will improve two classical results by showing that they also hold uniformly. 
In Section \ref{subsec: reg var intro} we will recall some classical results in the theory of regular variation 
which enables us later to prove Theorem \ref{Sb(n)}. 

With this background we are able to prove the main steps explicated in the previous section. 
In Section \ref{subsec: exp ineq} we prove the exponential inequality Lemma \ref{lemma: Tnfn chi deviation allg}.
Using this inequality we can prove the statements concerning Properties $\bm{A}$ and $\bm{B}$, i.e.\ 
in Section \ref{proofs transfer A} we prove Theorems \ref{Thm: Sn* allg} and \ref{Thm: Sn* reg var}
and in Section \ref{proofs transfer C} we prove Lemma \ref{bernoulli}.
Finally, in Section \ref{proofs transfer prelim} we prove our main theorems, i.e.\ we prove 
the necessary properties to apply Lemma \ref{lem: Prop A B to trimming}
and show then how the trimming statements in Sections \ref{subsec: gen results} and \ref{subsec: reg var results}
follow from this lemma.
In the last Section \ref{subsec: proof ex} we prove that our example from Section \ref{Examples}
actually fulfills all required properties.

\section{Preliminary statements}\label{sec: Intro statements}
\subsection{Spectral gap and decay of correlation}\label{subsec: spec gap corr decay}
In this section we will first give the precise definition of a spectral gap.
Then we state some properties for dynamical systems fulfilling Property $\mathfrak{C}$, for instance a decay of correlations.

The results of this section are known, but scattered around the literature and often only proven for one particular system
why we decided to reformulate them in the general setting.

In the following we denote the spectral radius of an operator $U$ by $\rho\left(U\right)$.
The following definition is a key part of Property $\mathfrak{C}$ for $\widehat{T}$, see Definition \ref{def: prop C}.
\begin{Def}[Spectral gap]\label{def spec gap}
Suppose $\mathcal{F}$ is a Banach space and $U:\mathcal{F}\to\mathcal{F}$ a bounded linear operator. We say that $U$ has a spectral gap if there exists a decomposition $U=\lambda P+N$ with $\lambda\in\mathbb{C}$ and $P,N$ bounded linear operators such that 
\begin{itemize}
\item $P$ is a projection, i.e.\ $P^2=P$ and $\dim\left(\Image\left(P\right)\right)=1$,
\item $N$ is such that $\rho\left(N\right)<\left|\lambda\right|$,
\item $P$ and $N$ are orthogonal, i.e.\ $PN=NP=0$.
\end{itemize}
\end{Def}

Dealing with operators $U:\mathcal{F}\to \mathcal{F}$ we use the operator norm
$\left\|U\right\|\coloneqq\sup_{\left\|\varphi\right\|\leq 1} \left\|U\varphi\right\|$. 

\begin{lemma}\label{spec gap}
Let $\left(\Omega, \mathcal{A},  T, \mu,\mathcal{F},\left\|\cdot\right\|\right)$ fulfill Property $\mathfrak{C}$.
Then $\widehat{ T}$ has a simple eigenvalue $\lambda=\rho\left(\widehat{ T}\right)=1$. This eigenvalue is unique on the unit circle and has maximal modulus.
\end{lemma}
\begin{proof}
Using the defining relation of $\widehat{T}$ in \eqref{hat CYRI} gives that its operator norm  with respect to the $\mathcal{L}_{\mu}^{1}$-norm is equal to $1$. This immediately implies that the modulus of any eigenvalue cannot exceed $1$. 
Now, if $\lambda$ is an eigenvalue of the eigenfunction $f$, then we have for every $A\in\mathcal{B}$ and $n\in\mathbb{N}$ that
\begin{align}
\int_{ T^{-n}\left(A\right) }f\mathrm{d}\mu&=\int f\cdot \mathbbm{1}_{ T^{-n}\left(A\right)}\mathrm{d}\mu
=\int f\cdot \mathbbm{1}_A\circ  T^n\mathrm{d}\mu
=\int \widehat{ T}^nf\cdot  \mathbbm{1}_A\mathrm{d}\mu=\lambda^n\int_{A}f\mathrm{d}\mu.\label{eq: int T^-n A}
\end{align}
If we assume that $\lambda=1$, then for all $a\in \mathbb{R}$ and   $n\in\mathbb{N}$, we have
\begin{align*}
\int_{ T^{-n}\left(\left\{f>a\right\}\right)}f\mathrm{d}\mu=\int_{\left\{f>a\right\}}f\mathrm{d}\mu.
\end{align*}
Since $T$ is ergodic and  preserves the probability measure $\mu$ it follows that $\mu (\left\{f>a\right\})\in \{0,1\}$, for all $a\in \mathbb{R}$, and consequently  $f$ has to be constant almost everywhere. Since $1$ is the eigenvalue only for the constant functions, 
this eigenvalue has to be simple. 
Furthermore, since $\mu$ is mixing we have for all $f\in\mathcal{L}^2$ and thus $f\in\mathcal{F}$ and all $A\in\mathcal{B}$ that
\begin{align*}
 \lim_{n\to\infty}\int f\cdot \mathbbm{1}_A\circ  T^n\mathrm{d}\mu
 =\int f\mathrm{d}\mu\cdot\mu\left(A\right).
\end{align*}
If we assume that $f$ is an eigenfunction to the eigenvalue $\lambda$,
\eqref{eq: int T^-n A} implies
\begin{align*}
 \lim_{n\to\infty}\lambda^n\int_{A}f\mathrm{d}\mu=\int f\mathrm{d}\mu\cdot\mu\left(A\right).
\end{align*}
In case that $\lambda$ lies on the unit circle the only possibility that this equality holds is that $\lambda=1$.
\end{proof}

\begin{lemma}\label{Pf=intf}
Let $\left(\Omega, \mathcal{A},  T, \mu,\mathcal{F},\left\|\cdot\right\|\right)$ fulfill Property $\mathfrak{C}$ and write $\widehat{ T}=\lambda P+N$ with $\lambda, P, N$ as in Definition \ref{def spec gap}.
Then $Pf=\int f\mathrm{d}\mu$ holds for all $f\in\mathcal{F}$.
\end{lemma}
\begin{proof}
From the proof of Lemma \ref{spec gap} we know that the constant functions are eigenfunctions of $\widehat{ T}$ to the maximal eigenvalue $\lambda=1$. 
Thus, $Pf$ is almost everywhere constant for all $f\in\mathcal{F}$. By Riesz' representation theorem we know that there exists $g\in \mathcal{L}^{\infty}$ such that for all $f\in\mathcal{F}$ we have that $Pf=\int f\cdot g\mathrm{d}\mu$. Furthermore,
\begin{align*}
\int f\cdot g\mathrm{d}\mu=Pf=P\left(\widehat{ T}f\right)=\int\widehat{ T}f\cdot g\mathrm{d}\mu=\int f\cdot g\circ  T\mathrm{d}\mu.
\end{align*}
Since $\mu$ is ergodic, $g=g\circ  T$ and thus, $g$ is constant. 
\end{proof}

With the above lemmas we can prove that a decay of correlation holds under Property $\mathfrak{C}$.
\begin{lemma}[Decay of correlations]\label{decay corr}
Let $\left(\Omega, \mathcal{A},  T, \mu,\mathcal{F},\left\|\cdot\right\|\right)$ fulfill Property $\mathfrak{C}$.
Further, let $\varphi_1\in \mathcal{L}^1$ and $\varphi_2\in\mathcal{F}$ and define 
\begin{align*}
\overline{\tau}\coloneqq\sup\left\{\left|z\right|\colon \left|z\right|<1, z\in\spec\left(\widehat{ T}\right)\right\}. 
\end{align*}
Then there exists for every $\tau>\overline{\tau}$ a constant $R>0$ such that 
\begin{align*}
\cor_{\varphi_1,\varphi_2}\left(n\right)&\coloneqq\left|\int \left(\varphi_1\circ  T^n\right)\cdot \varphi_2\mathrm{d}\mu-\int\varphi_1\mathrm{d}\mu\cdot\int\varphi_2\mathrm{d}\mu\right|
\leq  R \cdot\left|\varphi_1\right|_1\cdot \left\|\varphi_2\right\|\cdot \tau^n,
\end{align*}
for all $n\in\mathbb{N}$. 
\end{lemma}
The proof of this lemma follows the proof of \cite[Theorem 1.6]{baladi_positive_2000} 
which is given for piecewise expanding interval maps.
\begin{proof}
By the properties of the transfer operator we have that 
\begin{align*}
\cor_{\varphi_1,\varphi_2}\left(n\right)&\coloneqq\left|\int \left(\varphi_1\circ  T^n\right)\cdot \varphi_2\mathrm{d}\mu-\int\varphi_1\mathrm{d}\mu\cdot\int\varphi_2\mathrm{d}\mu\right|\\
&= \left|\int \varphi_1\cdot \widehat{ T}^n\varphi_2\mathrm{d}\mu-\int\varphi_1\mathrm{d}\mu\cdot\int\varphi_2\mathrm{d}\mu\right|\\
&= \left|\int \varphi_1\cdot \left(\widehat{ T}^n\varphi_2-\int\varphi_2\mathrm{d}\mu\right)\mathrm{d}\mu\right|\\
&\leq \int \left|\varphi_1\right|\mathrm{d}\mu\cdot \left|\widehat{ T}^n\left(\varphi_2-\int\varphi_2\mathrm{d}\mu\right)\right|_{\infty}.
\end{align*}
It follows from Lemmas \ref{spec gap} and \ref{Pf=intf} that $\spec\big(\widehat{ T}\big)$ consists of the simple eigenvalue $1$ with the corresponding spectral projection $P f=\int f\mathrm{d}\mu$ and a subset of a disc with radius $\overline{\tau}<1$.
Hence, 
\begin{align*}
\rho\left(\widehat{ T}\left(\id-P\right)\right)=\overline{\tau}
\end{align*}
and it follows from  Gelfand's formula, see for example \cite[p.\ 195]{lax_functional_2002}, and $\left|\cdot\right|_{\infty}\leq \left\|\cdot\right\|$ that for all $\tau>\overline{\tau}$ there exists $ R>0$ such that for all $n\in\mathbb{N}$
\begin{align*}
\left|\widehat{ T}^n\varphi_2-\int\varphi_2\mathrm{d}\mu\right|_{\infty}&\leq \left\|\widehat{ T}^n\varphi_2-\int\varphi_2\mathrm{d}\mu\right\|
\leq  R\cdot \tau^n\cdot \left\|\varphi_2\right\| 
\end{align*}
and thus, the statement of the lemma follows.
\end{proof}

\subsection{Analytic perturbation of the transfer operator}\label{subsec: analytic perturbation}
We consider in the following the perturbed transfer operator given by
\begin{align}
\widehat{ T}_zh\coloneqq\upperleft{\varphi}{\widehat}{ T_z}h\coloneqq\widehat{ T}\left(\mathrm{e}^{z\cdot \varphi}\cdot h\right).\label{eq: def twisted transfer}
\end{align}
Obviously, $\widehat{ T}=\widehat{ T}_0$.
We remember from the previous section that under the restriction that condition $\mathfrak{C}$ holds 
we have that $\widehat{ T}= P+N\coloneqq P_0+N_0$ 
with $\rho(N_0)<1$.

Our main goal in this section is to prove the following two main lemmas, where the first reads as follows.
\begin{lemma}\label{spect sep analyt}
Let $\left(\Omega, \mathcal{A},  T, \mu,\mathcal{F},\left\|\cdot\right\|\right)$ fulfill Property $\mathfrak{C}$.
For all $\theta\in\left(0,(1-\rho(N_0))/2\right)$ there exists $\varkappa>0$ such that for all
$\varphi\in\mathcal{F}$ with $\int \varphi\mathrm{d}\mu=0$ and  $\left|z\right|\cdot \left\|\varphi\right\| < \varkappa$ we have 
\begin{align}
\upperleft{\varphi}{\widehat}{ T_z} &=\upperleft{\varphi}{}{\lambda_z} \upperleft{\varphi}{}{P_z} + \upperleft{\varphi}{}{N_z}
\eqqcolon\lambda_z P_z + N_z,\label{spect sep form}
\end{align}
where 
\begin{align}
P_z^2 = P_z,\text{ }\dim \Image\left(P_z\right) = 1,\text{ } N_z P_z = P_z N_z = 0,\text{ }
\left|\lambda_z\right|\geq  1-\theta,\text{ and }\rho(N_z)\leq \rho(N_0)+\theta,\label{spec gap prop z}
\end{align}
and
$T_z$, $\lambda_z$, $P_z$, and $N_z$ are analytic with the definition of analyticity for operators given in Definition \ref{def: weak strong analyt}.
\end{lemma}

The fact that for a given function $\varphi$ the operator $\upperleft{\varphi}{\widehat}{ T_z}$ is analytic and has a
decomposition $\upperleft{\varphi}{}{\lambda_z}\cdot \upperleft{\varphi}{}{P_z}+\upperleft{\varphi}{}{N_z}$, where 
$\upperleft{\varphi}{\widehat}{T_z}$, $\upperleft{\varphi}{}{\lambda_z}$, $\upperleft{\varphi}{}{P_z}$, and $\upperleft{\varphi}{}{N_z}$ 
are analytic in a neighborhood of zero is widely known, see for example \cite[Proposition 2.3]{gouezel_limit_2015},
\cite[Theorem 8.2]{fan_spectral_2003},
\cite[Chapter 3]{sarig_introduction_2012}.

However, those results say nothing about a uniform bound of the analytic convergence radius 
if one considers the perturbed operator $\upperleft{\varphi}{\widehat}{ T_z}$ for different functions $\varphi$. 
Furthermore, to prove Lemma \ref{lemma: resolvent} we will see that the additional uniform bounds 
$\left|\lambda_z\right|\geq  1-\theta$ and $\rho(N_z)\leq \rho(N_0)+\theta$ instead of just assuming $\rho(N_z)<\lambda_z$ are necessary.

When we speak here about analyticity of operators we usually refer to strong analyticity given in the following definition.
\begin{Def}[Weak and strong analyticity]\label{def: weak strong analyt}
Let $\mathcal{G}$ be a Banach space and $B\coloneqq B\left(\mathcal{G}\right)$ be the space of all bounded linear operators.
Let $U\subset\mathbb{C}$ be open and $\left(L_z\right)_{z\in U}$ a one parameter family of operators in $B$.

We call $\left(L_z\right)$ weakly analytic if $\varphi\left(L_z\right)$ is homomorphic on $U$ for every bounded linear $\varphi: B\to\mathbb{C}$.

We call $\left(L_z\right)$ strongly analytic if for every $z\in U$ there exists $L'_z\in U$ such that 
\begin{align*}
\lim_{h\to 0}\left\|\frac{L_{z+h}-L_z}{h}-L'_z\right\|=0. 
\end{align*}
\end{Def}
However, from  \cite[Theorem III-3.12]{kato_perturbation_1995} we know that in fact
weak analyticity and strong analyticity are equivalent.

Our second main lemma in this section will give a bound on $n$ such that 
$\left|\upperleft{\varphi}{}{\lambda}_z\right|^{-n}\cdot\left\|\upperleft{\varphi}{}{N}_z^n\right\|\leq 1$.
\begin{lemma}\label{lemma: resolvent}
Let $\left(\Omega, \mathcal{A},  T, \mu,\mathcal{F},\left\|\cdot\right\|\right)$ fulfill Property $\mathfrak{C}$ and
let $\upperleft{\varphi}{}{\lambda}_z$, $\upperleft{\varphi}{}{N}_z$, and $\varkappa$ be given as in Lemma \ref{spect sep analyt}.
Then there exists $N\in\N$ and $\varkappa'\in\left(0, \varkappa\right)$ such that
for all $\varphi\in\mathcal{F}$ fulfilling $\int\varphi\mathrm{d}\mu=0$,
all $z\in\mathbb{C}$ with $\left|z\right|\cdot \left\|\varphi\right\| < \varkappa'$, 
and all $n\in\N_{>N}$
\begin{align}
\left|\upperleft{\varphi}{}{\lambda}_z\right|^{-n}\cdot\left\|\upperleft{\varphi}{}{N}_z^n\right\|
&\leq 1.\label{eq: lambdaz Nz}
\end{align}
\end{lemma}
From Gelfand's formula, see for example \cite[p.\ 195]{lax_functional_2002}, 
and Lemma \ref{spect sep analyt}
it follows easily that
\begin{align*}
\sup_{\left|z\right|<\varkappa/\left\|\varphi\right\|}\rho\left(\upperleft{\varphi}{}{N}_z\right)=\adjustlimits\sup_{\left|z\right|<\varkappa/\left\|\varphi\right\|}\lim_{n\to\infty}\sqrt[n]{\left\|\upperleft{\varphi}{}{N}_z^n\right\|}<\inf_{\left|z\right|<\varkappa/\left\|\varphi\right\|}\left|\upperleft{\varphi}{}{\lambda}_z\right|
\end{align*}
and for fixed $\varphi$ an $N$ fulfilling \eqref{eq: lambdaz Nz} for all $n\geq N$ is easily determined.
However, as we will see in the proof of this lemma at the end of the section, some more attention is needed to obtain a uniform bound.

In what follows we will recall some of the known results and prove with general operator theory the above lemmas.
In the following let $\varphi_z^{\left(k\right)}$ with $z\in U$ and $U$ an open domain in $\mathbb{C}$, denote the $k$-th derivative of operators or $\mathbb{C}$ valued functions $\left(z\mapsto\varphi_{z}\right)_{z\in U}$. 
\begin{lemma}\label{lemma Tz}
Let $\left(\Omega, \mathcal{A},  T, \mu,\mathcal{F},\left\|\cdot\right\|\right)$ fulfill Property $\mathfrak{C}$ and let $\varphi\in\mathcal{F}$. Then the operator $\widehat{ T}_z$ given by \eqref{eq: def twisted transfer}
is analytic on $\mathbb{C}$ and we have 
\begin{align}
\widehat{ T}_z^{\left(k\right)}=\widehat{ T}_zM_{\varphi}^k\label{Tzk}
\end{align}
where $M_{\varphi}:\mathcal{F}\to\mathcal{F}$ and $M_{\varphi}h\coloneqq\varphi h$.
\end{lemma}

\begin{proof}
To show the analyticity of $\widehat{ T}_z$, observe that 
by definition
\begin{align}\label{series exp}
 \widehat{ T}_zh
 &=\widehat{T}\left(\sum_{n=0}^{\infty}\frac{\left(z\cdot\varphi\right)^n}{n!}\cdot h\right)
 =\sum_{n=0}^{\infty}\frac{z^n}{n!}\cdot \widehat{T}\left(\varphi^n\cdot h\right)
 =\sum_{n=0}^{\infty} \frac{z^n}{n!}\cdot\widehat{T} M^n_{\varphi}h
\end{align}
in a neighborhood of zero.
$M_{\varphi}$ is bounded, because $\left\|M_{\varphi} h \right\|\leq \left\|\varphi\right\| \left\|h\right\|$.
Therefore, 
\begin{align*}
\left\|\widehat{ T}_z\right\|\leq \left\|\widehat{ T}\right\|\cdot \sum_{n=0}^{\infty}\frac{\left|z\right|^n\cdot\left\|\varphi\right\|^n}{n!}
\end{align*}
converges, i.e.\ $\widehat{ T}_z$ converges absolutely with respect to $\left\|\cdot\right\|$ and the representation in \eqref{series exp} holds for all $z\in\mathbb{C}$. 
Hence, for each $\varphi$ element of the dual of the space of bounded linear operators, $\varphi\,\big(\widehat{ T}_z\big)$ can be expanded into a power series.
Since strong and weak analyticity are equivalent  it follows that $\widehat{ T}_z'$ and by the analyticity also the derivatives $\widehat{ T}_z^{\left(n\right)}$ exist. 
\eqref{Tzk} follows from the series expansion \eqref{series exp} by induction.
\end{proof}

In the following we are considering for which $\varphi$ and $z\in\mathbb{C}$ 
the operator $\upperleft{\varphi}{\widehat}{T}_z$ is close to $\widehat{T}$.
In order to specify closeness we start with some definitions concerning the norms.
\begin{Def}\label{dist measure}
For $M\subset\mathcal{F}$ set $S_{M}\coloneqq\left\{u\in M\colon \left\|u\right\|=1\right\}$.
We define for two sets $M,N\subset \mathcal{F}$ and $u\in\mathcal{F}$ the distance measures
\begin{align*}
\dist\left(u,M\right)&\coloneqq\inf_{v\in M}\left\|u-v\right\|,\\
d \left(M,N\right)&\coloneqq\sup_{u\in S_M}\dist\left(u,N\right),\\
\widehat{d }\left(M,N\right)&\coloneqq\max\left\{d \left(M,N\right),d \left(N,M\right)\right\}.  
\end{align*}
Furthermore, we define for two operators $U,V:\mathcal{F}\to\mathcal{F}$ 
\begin{align*}
\widehat{d }\left(U,V\right)\coloneqq\widehat{d }\left(G\left(U\right),G\left(V\right)\right),
\end{align*}
where $G\left(U\right)$ denotes the graph of $U$. 
As the graph norm for a graph of an operator $U:\mathcal{F}\to \mathcal{F}$ we use for $\left(x, y\right)\in \mathcal{F}\times \mathcal{F}$ the norm
\begin{align*}
\left\|\left(x,y\right)\right\|\coloneqq\sqrt{\left\|x\right\|^2+\left\|y\right\|^2}.
\end{align*}
\end{Def}
\begin{Rem}
Note that in \cite{kato_perturbation_1995} $d$ and $\widehat{d}$ are denoted by $\delta$ and $\widehat{\delta}$, whereas $d$ are $\widehat{d}$ are used in there to denote some other distance measures.

 Indeed, the convergence with respect to $\widehat{d}$ is a generalization to convergence in norm as the following calculation for $U, V\colon\mathcal{F}\to\mathcal{F}$ shows:
 By Definition \ref{dist measure} it follows that
\begin{align}
d \left(G(U),G\left(V\right)\right)&=\adjustlimits\sup_{u\in S_{G\left(U\right)}}\inf_{v\in G\left(V\right)}\left\|u-v\right\|\notag\\
&\leq \sup_{\left\|\left(f,Vf\right)\right\|\leq 1}\sqrt{\left\|f-f\right\|^2+\left\|U f-V f\right\|^2}\notag\\
&\leq \sup_{\left\|f\right\|\leq 1}\left\|U f-V f\right\|.\label{delta 1} 
\end{align}
Similarly, we obtain
\begin{align}
d \left(G\left(V\right),G\left(U\right)\right)
&\leq \sup_{\left\|f\right\|\leq 1}\left\|V f-U f\right\|.\label{delta 2} 
\end{align}
Combining \eqref{delta 1} and \eqref{delta 2} yields 
\begin{align}
\widehat{d }\left(U,V\right)
&\leq \sup_{\left\|f\right\|\leq 1}\left\|V f-U f\right\|.\label{eq: hat d UV < norm} 
\end{align}
\end{Rem}

\begin{lemma}\label{lemma: distance of pertub}
Let $\left(\Omega, \mathcal{A},  T, \mu,\mathcal{F},\left\|\cdot\right\|\right)$ fulfill Property $\mathfrak{C}$. 
Let $\upperleft{\varphi}{\widehat}{T_z}$ be defined as in \eqref{eq: def twisted transfer}.
Then for all $\epsilon>0$ there exists $\delta>0$ such that $\left|z\right|\cdot \left\|\varphi\right\|<\delta$
implies $\widehat{d}\left(\upperleft{\varphi}{\widehat}{T_z}, T\right)<\epsilon$.
\end{lemma}
\begin{proof}
Let $\epsilon>0$ be given such that $\widehat{d }\,\big(\widehat{ T},\upperleft{\varphi}{\widehat}{ T_z}\big)<\epsilon$.
By the definition of $\upperleft{\varphi}{\widehat}{ T_z}$ and inequalities \eqref{C 0 f} and \eqref{eq: hat d UV < norm} we obtain
\begin{align}
\widehat{d }\left(\widehat{ T},\upperleft{\varphi}{\widehat}{ T_z}\right)
&\leq \sup_{\left\|f\right\|\leq 1}\left\|\widehat{ T}\left(\exp\left(\varphi\cdot z\right)\cdot f\right)-\widehat{ T} f\right\| =\sup_{\left\|f\right\|\leq 1}\left\|\widehat{ T}\left(\left(\exp\left(\varphi\cdot z\right)-1\right)\cdot f\right)\right\|\notag\\
&\leq \sup_{\left\|f\right\|\leq 1}K_0\cdot \left\|\left(\exp\left(\varphi\cdot z\right)-1\right)\cdot f\right\|\leq K_0\cdot \left\|\left(\exp\left(\varphi\cdot z\right)-1\right)\right\|,\label{delta 3}
\end{align}
where the last inequality follows from the fact that $\mathcal{F}$ is a Banach algebra.
Furthermore, we have that 
\begin{align}
\left\|\exp\left(z\cdot \varphi\right)-1\right\|=\left\|\sum_{n=1}^{\infty}\frac{\left(z\cdot \varphi\right)^n}{n!}\right\|\leq \sum_{n=1}^{\infty}\frac{\left|z\right|^n\cdot\left\|\varphi\right\|^n}{n!}\leq\exp\left(\left\|\varphi\right\|\cdot\left|z\right|\right)-1\label{exp z kappa}
\end{align}
and we can thus conclude from \eqref{delta 3} that 
\begin{align*}
\widehat{d }\left(\widehat{ T},\upperleft{\varphi}{\widehat}{ T_z}\right)
&\leq K_0\cdot \left(\exp\left(\left\|\varphi\right\|\cdot\left|z\right|\right)-1\right).
\end{align*}
Furthermore, there exists $\delta>0$ such that $\left\|\varphi\right\|\cdot\left|z\right|<\delta$ implies
$K_0\cdot \left(\exp\left(\left\|\varphi\right\|\cdot\left|z\right|\right)-1\right)<\epsilon$, which gives the statement of the lemma.
\end{proof}

In the next steps we prove that the operator $\upperleft{\varphi}{\widehat}{T_z}$, if close enough to $\widehat{T}$, can indeed be spectrally decomposed similarly as $\widehat{T}$ without saying anything about the analyticity of the spectral components.
More precisely we formulate this in the following lemma.
\begin{lemma}\label{spect sep}
Let $\left(\Omega, \mathcal{A},  T, \mu,\mathcal{F},\left\|\cdot\right\|\right)$ fulfill Property $\mathfrak{C}$. 
Let $\upperleft{\varphi}{\widehat}{T_z}$ be defined as in \eqref{eq: def twisted transfer}.
Then there exists $\kappa>0$ such that for all $\varphi\in\mathcal{F}$ with $\int \varphi\mathrm{d}\mu=0$ and  $\left|z\right|\cdot \left\|\varphi\right\| < \kappa$ 
a spectral decomposition as in \eqref{spect sep form} exists fulfilling \eqref{spec gap prop z}.
\end{lemma}

For the proof of Lemma \ref{spect sep}
we will basically use the following parts of \cite[Theorem IV-3.16]{kato_perturbation_1995}:
\begin{lemma}\label{Kato IV 3.16}
Let $\mathcal{G}$ be a Banach space, let $U:\mathcal{G}\to\mathcal{G}$ be a bounded linear operator, 
and let $\Sigma\left(U\right)\coloneqq\spec\left(U\right)$ be separated into two parts $\Sigma'\left(U\right)$ and $\Sigma''\left(U\right)$ by a simple closed curve $\Gamma$.
Let $\mathcal{G}=M'\left(U\right)\oplus M''\left(U\right)$ be the associated decomposition of $\mathcal{G}$. 
Then there exists $\epsilon>0$ depending on $U$ and $\Gamma$ such that for all $V:\mathcal{G}\to\mathcal{G}$ with $\widehat{d }\left(U,V\right)<\epsilon$ 
the spectrum $\Sigma(V)\coloneqq \spec(V)$ is separated by $\Gamma$ into $\Sigma'(V)$ and $\Sigma''(V)$
and the associated decomposition 
\begin{align*}
\mathcal{G}=M'\left(V\right)\oplus M''\left(V\right)
\end{align*}
fulfills
\begin{align*}
\dim M'\left(V\right)=\dim M'\left(U\right)\text{ and }\dim M''\left(V\right)=\dim M''\left(U\right).
\end{align*}
The decomposition $\mathcal{G} = M' (V) \oplus M'' (V)$ is continuous in $V$ in the sense that
for all $\epsilon>0$ there exists $\delta>0$ such that $\widehat{d }\left(U, V\right)<\delta$
implies $\left\|P[U]- P[V]\right\|<\epsilon$, 
where $P[U]$ denotes the projection of $\mathcal{G}$ onto $M' (V)$ along $M'' (V)$ and $P[V]$ analogously.

The lemma holds analogously for separating the spectrum in more than one component.
\end{lemma}

\begin{proof}[Proof of Lemma \ref{spect sep}]
Since $\widehat{ T}$ has a spectral gap, we can write
\begin{align*}
\widehat{ T}=\widehat{ T}_0=\lambda_0 P_0 + N_0
\end{align*}
with $\lambda_0$, $P_0$, and $N_0$ fulfilling \eqref{spec gap prop z}.
For proving that $\widehat{ T}_z$ can be represented in the same manner we aim to apply Lemma \ref{Kato IV 3.16} splitting the spectrum formally in three components. 
First we choose for $\Gamma=\Gamma_1$ a circle around $0$ with radius $r=\rho(N_0)+\theta$
and as a second step we choose for $\Gamma=\Gamma_2$ a circle around $0$ with radius $r=1-\theta$.
By the fact that $\widehat{T}$ has a spectral gap and Lemma \ref{spec gap} this separates the spectrum of $\widehat{T}$ into the single eigenvalue $1$ 
(outside the circle $\Gamma_1$), an empty part of the spectrum (between the circles $\Gamma_1$ and $\Gamma_2$) 
and the remainder of the spectrum $\spec\big(\widehat{ T}\big)\backslash \left\{1\right\}$ (inside the circle $\Gamma_2$). 

These circles determine $\epsilon$ such that the spectrum of all $U:\mathcal{F}\to\mathcal{F}$ with $\widehat{d }\,\big(\widehat{ T},U\big)<\epsilon$ is separated in the above described way. 
Hence, $\widehat{d }\,\big(\widehat{ T},\upperleft{\varphi}{\widehat}{ T_z}\big)<\epsilon$ implies that a spectral decomposition as in \eqref{spect sep form} exists fulfilling \eqref{spec gap prop z}.
Using Lemma \ref{lemma: distance of pertub} and choosing $\kappa$ appropriately gives the statement of the lemma.
\end{proof}

Finally, to prove Lemma \ref{spect sep analyt} we will use \cite[Theorems VII-1.7 and VII-1.8]{kato_perturbation_1995} and their proofs. Together they state the following.
\begin{lemma}\label{Kato VII 1.7}
If a family of operators $\left(U_z\right)_{z\in\mathbb{C}}$ with $U_z:\mathcal{F}\to\mathcal{F}$ depends on $z$ holomorphically 
and has a spectrum consisting of two separated parts and each $U_z$ can be represented as $U_z=\lambda_zP_z+N_z$ fulfilling the requirements as in \eqref{spec gap prop z}, 
then there exists $\epsilon>0$ such that $\lambda_z$, $P_z$, and $N_z$ consist of branches of one or several analytic functions with at most algebraic singularities in $B\left(\epsilon,0\right)$, i.e.\ the ball with radius $\epsilon$ around zero. 
Let
\begin{align}
\left\|\left(U_z-U_0\right)f\right\|\leq a\cdot \left\|f\right\|+b\cdot \left\|U_z f\right\|\label{U1}
\end{align}
and
\begin{align}
a\cdot\left\|U_0^{-1}\right\|+b<1,\label{U2} 
\end{align}
then $\epsilon$ can be determined as the supremum of $r\in\mathbb{R}_{>0}$ such that \eqref{U1} and \eqref{U2} hold for all $\left|z\right|\leq r$ and all $f\in\mathcal{F}$. 
\end{lemma}

\begin{proof}[Proof of Lemma \ref{spect sep analyt}]
By Lemma \ref{spect sep} there exists $\kappa>0$ such that the representation $\widehat{ T}_z=\lambda_zP_z+N_z$ fulfilling \eqref{spec gap prop z} exists if $\left|z\right|\cdot \left\|\varphi\right\|<\kappa$. 
We aim to apply Lemma \ref{Kato VII 1.7} on $\widehat{ T}_z$. The analyticity of $\widehat{ T}_z$ is given by Lemma \ref{lemma Tz}. 
By \eqref{C 0 f} we have that 
\begin{align*}
\left\|\left(\widehat{ T}_z-\widehat{ T}_0\right)f\right\|&=\left\|\widehat{ T}_0\left(\left(\exp\left(z\cdot \varphi\right)-1\right)f\right)\right\|
\leq K_0\cdot \left\|\exp\left(z\cdot \varphi\right)-1\right\|\cdot\left\|f\right\|
\end{align*}
and by \eqref{exp z kappa} it follows that
\begin{align*}
\left\|\left(\widehat{ T}_z-\widehat{ T}_0\right)f\right\|&\leq K_0\cdot \left(\exp\left(\left|z\right|\cdot \left\|\varphi\right\|\right)-1\right)\cdot\left\|f\right\|.
\end{align*}
Hence, $U_z\coloneqq\widehat{ T}_z$ fulfills \eqref{U1} if we set $a\coloneqq K_0\cdot \left(\exp\left(\left|z\right|\cdot \left\|\varphi\right\|\right)-1\right)$ and $b\coloneqq0$. 
Since $\exp\left(\left|z\right|\cdot \left\|\varphi\right\|\right)-1$ can be chosen arbitrarily close to zero for sufficiently small $\left|z\right|$, there exists $\varkappa\in\left(0,\kappa\right]$ such that 
$a\leq 1/\|\widehat{T}_0^{-1}\|$
and thus \eqref{U2} holds for all $\left|z\right|\cdot \left\|\varphi\right\|<\varkappa$.

Since \eqref{spect sep form} holds and there is in particular only one branch in $z=0$, it follows that there are no algebraic singularities in $B\left(\epsilon,0\right)$ and $\lambda_z$ is analytic on $B\left(\epsilon,0\right)$.
\end{proof}

Finally, we want to prove Lemma \ref{lemma: resolvent}. First we define for a bounded operator $U$ the resolvent $R\left(\xi,U\right)\coloneqq (\xi\mathrm{id}-U)^{-1}$
and the resolvent set of $U$ as the set of all scalars $\xi$ such that $R\left(\xi,U\right)<\infty$.
Before we start with the proof of this lemma, we first state the following lemma about the perturbation of the resolvent 
which is a simplified version of \cite[Theorem IV-3.15]{kato_perturbation_1995}.

\begin{lemma}\label{lemma: resolvent cont}
Let $\mathcal{G}$ be a Banach space and $U\colon\mathcal{G}\to\mathcal{G}$ be a bounded linear operator.
For all $\epsilon>0$ there exists $\delta>0$ such that for all bounded operators $V\colon\mathcal{G}\to\mathcal{G}$ 
we have that $\widehat{d }(U, V) < \delta$ implies $\left\|R (\xi, V) - R (\xi, U) \right\| <\epsilon$.
\end{lemma}

\begin{proof}[Proof of Lemma \ref{lemma: resolvent}]
Following the proof of Gelfand's formula, see for example \cite[p.\ 196]{lax_functional_2002},
yields for all $\delta>0$ and $n\in\mathbb{N}$
that 
\begin{align}
\left\|\upperleft{\varphi}{}{N}_z^n\right\|\leq c\cdot \left(\rho\left(\upperleft{\varphi}{}{N}_z\right)+\delta\right)^{n+1}\label{eq: ||Nz^n||}
\end{align}
with
\begin{align*}
c\coloneqq \upperleft{\varphi}{}{c}_z\coloneqq\max_{\left|\xi\right|=\rho\left(\upperleft{\varphi}{}{N}_z\right)+\delta}\left\|R\left(\xi, \upperleft{\varphi}{}{N}_z\right)\right\|.
\end{align*}
In the next steps we will show that there exists $\varkappa'\in\left(0,\varkappa\right)$ such that $\upperleft{\varphi}{}{c}_z$ is uniformly bounded for $\left|z\right|\cdot \left\|\varphi\right\|<\varkappa'$.
Applying Lemma \ref{spect sep analyt} with $\theta\coloneqq \left(1-\rho\left(N_0\right)\right)/4$
yields $\sup_{\left|z\right|<\varkappa/\left\|\varphi\right\|}\rho\left(N_z\right)\leq \rho\left(N_0\right)+\theta$ if $\left|z\right|\cdot\left\|\varphi\right\|<\varkappa$. 
Thus, the maximum principle and choosing $\delta\coloneqq \theta$ yield for $\left|z\right|\cdot\left\|\varphi\right\|<\varkappa$ that
\begin{align*}
 \upperleft{\varphi}{}{c}_z
 &\leq\max_{\left|\xi\right|=\rho\left(\upperleft{\varphi}{}{N}_0\right)+2\theta}\left\|R\left(\xi, \upperleft{\varphi}{}{N}_z\right)\right\|\\
 &\leq \max_{\left|\xi\right|=\rho\left(\upperleft{\varphi}{}{N}_0\right)+2\theta}\left\|R\left(\xi, N_0\right)\right\|+\max_{\left|\xi\right|=\rho\left(\upperleft{\varphi}{}{N}_0\right)+2\theta}\left\|R\left(\xi, \upperleft{\varphi}{}{N}_z\right)-R\left(\xi, N_0\right)\right\|. 
\end{align*}
Our choice of $\theta$ ensures that $\zeta$ belongs to the resolvent set of $\upperleft{\varphi}{}{N}_z$
if $\left|\xi\right|=\rho\left(N_0\right)+2\theta$.

We note here that \eqref{eq: hat d UV < norm} implies $\left\|N_z-N_0\right\|\geq \widehat{d}\left(\upperleft{\varphi}{}{N}_z,N_0\right)$.
In the following we combine Lemma \ref{lemma: distance of pertub}, 
the continuity statement of Lemma \ref{Kato IV 3.16} 
applied for the projection to $N$,
and Lemma \ref{lemma: resolvent cont}.
This implies that for all $\xi$ with $\left|\xi\right|=\rho\left(N_0\right)+2\theta$ there exists $\varkappa_\xi'>0$ such that $\left|z\right|\cdot\left\|\varphi\right\|<\varkappa_{\xi}'$ implies
$\left\|R\left(\xi, \upperleft{\varphi}{}{N}_z\right)-R\left(\xi, N_0\right)\right\|\leq 1$.

Since the set $\left\{\xi\colon \left|\xi\right|=\rho\left(N_0\right)+2\theta\right\}$ is compact,
we can set $\varkappa'\coloneqq \min_{\left|\xi\right|=\rho\left(N_0\right)+2\theta}\varkappa_{\xi}'>0$.
Hence, 
\begin{align*}
 \sup_{\left|z\right|\cdot\left\|\varphi\right\|<\varkappa'}\upperleft{\varphi}{}{c}_z
 &\leq \max_{\left|\xi\right|=\rho\left(\upperleft{\varphi}{}{N}_0\right)+2\theta}\left\|R\left(\xi, N_0\right)\right\|+1\eqqcolon \widetilde{c}. 
\end{align*}

Before we estimate $\left|\upperleft{\varphi}{}{\lambda}_z\right|^{-n}\cdot\left\|\upperleft{\varphi}{}{N}_z^n\right\|$
we remember that Lemma \ref{spect sep analyt} gives 
$\inf_{\left|z\right|<\varkappa/\left\|\varphi\right\|}\left|\upperleft{\varphi}{}{\lambda}_z\right|\geq 1-\theta$
and 
$\sup_{\left|z\right|<\varkappa/\left\|\varphi\right\|}\rho\left(\upperleft{\varphi}{}{N}_z^n\right)\leq \rho\left(N_0\right)+\theta$.
Applying this and the choice of $\delta=\theta$ on \eqref{eq: ||Nz^n||} yields for all $\left|z\right|\cdot\left\|\varphi\right\|<\varkappa'$
that
\begin{align*}
\left|\upperleft{\varphi}{}{\lambda}_z\right|^{-n}\cdot\left\|\upperleft{\varphi}{}{N}_z^n\right\|
&\leq \left|\lambda_z\right|^{-n}\cdot \widetilde{c}\cdot \left(\rho\left(N_z\right)+\delta\right)^{n+1}
\leq \left(1-\theta\right)^{-n}\cdot \widetilde{c}\cdot \left(\rho\left(N_0\right)+2\theta\right)^{n+1}\\
&=\left(\frac{\rho\left(N_0\right)+2\theta}{1-\theta}\right)^n\cdot \widetilde{c}\cdot\left(\rho\left(N_0\right)+2\theta\right).
\end{align*}
Since by construction $\rho\left(N_0\right)+2\theta<1-\theta$, there exists $N\in\N$ such that for all $n> N$ 
\begin{align*}
 \left(\frac{\rho\left(N_0\right)+2\theta}{1-\theta}\right)^n<\frac{1}{\widetilde{c}\cdot\left(\rho\left(N_0\right)+2\theta\right)}
\end{align*}
and thus
$\left|\lambda_z\right|^{-n}\cdot\left\|N_z^n\right\|<1$, i.e.\ the statement of the lemma. 
\end{proof}

\subsection{Regular variation}\label{subsec: reg var intro}
In this section we give the definition of the de Bruijn conjugate, needed
in Theorem \ref{Sb(n)} and some technical lemmas used for the regular variation setting.
We start with the de Bruijn conjugate.
\begin{Def}
Let $L$ be a slowly varying function at infinity.
If the function $L^{\#}$ is slowly varying at infinity and fulfills
the following convergences 
\begin{align*}
\lim_{x\rightarrow\infty}L\left(x\right)\cdot L^{\#}\left(xL\left(x\right)\right)  =1=
\lim_{x\rightarrow\infty}L^{\#}\left(x\right)\cdot L\left(xL^{\#}\left(x\right)\right),
\end{align*}
it is called the {\em de Bruijn conjugate} of $L$. 
\end{Def}
The de Bruijn conjugate always exists and is unique up to asymptotic
equi\-va\-len\-ce. For further information see \cite[Section 1.5.7 and Appendix 5]{bingham_regular_1987}.
With the notion of the de Bruijn conjugate we are able to give asymptotic inverses for regularly varying functions
given in the next two lemmas.
\begin{lemma}
\label{bingham} Let $\gamma,\delta>0$. If $f:\mathbb{R}^{+}\rightarrow\mathbb{R}^{+}$
is such that $f\left(x\right)=x^{\gamma\cdot\delta}L^{\gamma}\left(x^{\delta}\right)$,
where $L$ denotes a slowly varying function in infinity, then any
function $g:\mathbb{R}^{+}\rightarrow\mathbb{R}^{+}$ with 
\begin{align}
g\left(x\right)\sim x^{1/\left(\gamma\cdot\delta\right)}L^{\#1/\delta}\left(x^{1/\gamma}\right)\label{asymp inv}
\end{align}
is the asymptotic inverse of $f$, i.e.\ $f\left(g\left(x\right)\right)\sim g\left(f\left(x\right)\right)\sim x$
for $x$ tending to infinity. One version of $g$ is $f^{\leftarrow}$. 
\end{lemma}
\begin{proof}  \cite[Proposition 1.5.15]{bingham_regular_1987}
states that \eqref{asymp inv} is the asymptotic inverse of $f$.
Proposition 1.5.12 of the same book states that $h$ with $h\left(y\right)\coloneqq\inf\left\{ x\in\left[0,\infty\right)\colon f\left(x\right)>y\right\} $
is one version of the asymptotic inverse and the asymptotic inverse
is unique up to asymptotic equivalence. So we are left to show that $h\left(y\right)\sim f^{\leftarrow}\left(y\right)$.
We assume the contrary of the statement. Obviously, $f^{\leftarrow}\leq h$.
Then there exists $\epsilon>0$ such that $f^{\leftarrow}\left(y\right)\left(1+\epsilon\right)\leq\inf\left\{ x\in\left[0,\infty\right)\colon f\left(x\right)>y\right\} $
for arbitrarily large $y$. This implies 
\begin{align*}
\inf\left\{ x\in\left[0,\infty\right)\colon f\left(x\right)=y\right\} \left(1+\epsilon\right)\leq\inf\left\{ x\in\left[0,\infty\right)\colon f\left(x\right)>y\right\} 
\end{align*}
for arbitrarily large $y$. That means there exist arbitrarily large $y\in\left[0,\infty\right)$
and $x\in\left[0,\infty\right)$ fulfilling
\begin{align}
f\left(x\right)\geq y\geq f\left(\left(1+\epsilon\right)x\right).\label{eq: f(1+esp) f}
\end{align}

On the other hand we have that $f\left(x\right)=x^{\gamma\delta}L^{\gamma}\left(x^{\delta}\right)$
and 
\begin{align*}
f\left(\left(1+\epsilon\right)x\right) & =\left(1+\epsilon\right)^{\gamma\delta}x^{\gamma\delta}L^{\gamma}\left(\left(1+\epsilon\right)^{\delta}x^{\delta}\right)
 \sim\left(1+\epsilon\right)^{\gamma\delta}x^{\gamma\delta}L^{\gamma}\left(x^{\delta}\right)
 =\left(1+\epsilon\right)^{\gamma\delta}f\left(x\right).
\end{align*}
This implies $f\left(\left(1+\epsilon\right)x\right)\geq\left(1+\epsilon/2\right)^{\gamma\delta}f\left(x\right)$,
for $x$ sufficiently large. Since $f$ tends to infinity, this contradicts \eqref{eq: f(1+esp) f} and hence 
the assumption. 
\end{proof}

We conclude this section with calculating the expectation of the truncated observable $\upperleft{u_{n}}{}{\chi}$.
\begin{lemma}
\label{E V X*} Let $\chi$ be such that $\mu(\chi>x)=L(x)/x^{\alpha}$ with $L$ a slowly varying function and $0<\alpha<1$.
Further, let $\left(u_{n}\right)$ be a non-negative sequence
with $\lim_{n\rightarrow\infty}u_{n}=\infty$. Then 
\begin{align}
\int\upperleft{u_{n}}{}{\chi}\mathrm{d}\mu & \sim\frac{\alpha}{1-\alpha}\cdot u_{n}^{1-\alpha}\cdot L\left(u_{n}\right),\label{E X*}\\
\int\mathsf{T}_{n}^{u_{n}}\chi\mathrm{d}\mu & \sim n\cdot \frac{\alpha}{1-\alpha}\cdot u_{n}^{1-\alpha}\cdot L\left(u_{n}\right),\label{E S*}
\end{align}
\end{lemma}
\begin{proof} First note that \eqref{E S*} follows immediately from \eqref{E X*}. To see \eqref{E X*} note 
\begin{align}
\int\upperleft{u_{n}}{}{\chi}\mathrm{d}\mu & =\int_{0}^{u_{n}}x\mathrm{d}F\left(x\right)=\left[xF\left(x\right)\right]_{0}^{u_{n}}-\int_{0}^{u_{n}}F\left(x\right)\mathrm{d}x\notag\\
 & =\left[x\left(1-\frac{L\left(x\right)}{x^{\alpha}}\right)\right]_{0}^{u_{n}}-\int_{0}^{u_{n}}\left(1-\frac{L\left(x\right)}{x^{\alpha}}\right)\mathrm{d}x\notag\\
 & =u_{n}-\left[x^{1-\alpha}L\left(x\right)\right]_{0}^{u_{n}}-u_{n}+\int_{0}^{u_{n}}\frac{L\left(x\right)}{x^{\alpha}}\mathrm{d}x\notag\\
 & =\int_{0}^{u_{n}}\frac{L\left(x\right)}{x^{\alpha}}\mathrm{d}x-\left[x^{1-\alpha}L\left(x\right)\right]_{0}^{u_{n}}.\label{E X* est}
\end{align}
Also, $\left[x^{1-\alpha}L\left(x\right)\right]_{0}^{u_{n}}=u_{n}^{1-\alpha}L\left(u_{n}\right)$.
To estimate the first summand of \eqref{E X* est} we apply Karamata's
theorem, see for example \cite[Theorem 1.5.11]{bingham_regular_1987},
and obtain 
\begin{align*}
\int_{0}^{u_{n}}\frac{L\left(x\right)}{x^{\alpha}}\mathrm{d}x\sim\frac{1}{1-\alpha}u_{n}^{1-\alpha}L\left(u_{n}\right).
\end{align*}
Hence, 
\begin{align*}
\int\upperleft{u_{n}}{}{\chi}\mathrm{d}\mu & \sim\left(\frac{1}{1-\alpha}-1\right)\cdot u_{n}^{1-\alpha}\cdot L\left(u_{n}\right)
 =\frac{\alpha}{1-\alpha}\cdot u_{n}^{1-\alpha}\cdot L\left(u_{n}\right).
\end{align*} 
\eqref{E S*} follows immediately from \eqref{E X*} by noting that 
$\int\mathsf{T}_{n}^{u_{n}}\chi\mathrm{d}\mu=n\cdot \int\upperleft{u_{n}}{}{\chi}\mathrm{d}\mu$.
\end{proof}

\section{Proofs of main theorems}\label{proofs transfer prelim}
\subsection{Proof of the exponential inequality}\label{subsec: exp ineq}
We will give here the proof of Lemma \ref{lemma: Tnfn chi deviation allg} in a series of lemmas. The proof applies the analytic perturbation theory 
discussed in Section \ref{subsec: analytic perturbation}, particularly Lemma \ref{spect sep analyt} and Lemma \ref{lemma: resolvent}.
We start with a generalization of Doob's inequality in order 
to estimate $\mu\left(\max_{i\leq n}\left|\mathsf{S}_i\varphi_n\right|\geq  u_n\right)$ in Lemma \ref{lemma 2 2nd part}.
\begin{lemma}[Generalized inequality of Doob]\label{doob ineq}
Let $\left(M_n\right)_{n\in\mathbb{N}}$ be a martingale, then we have for every convex function $g:\mathbb{R}\to\mathbb{R}$ and every $x>0$ that 
\begin{align}
\mathbb{P}\left(\max_{1\leq i\leq n}M_i\geq x\right)\leq\frac{\mathbb{E}\left(g\left(\left|M_n\right|\right)\right)}{g\left(x\right)}.\label{eq: doob}
\end{align}
\end{lemma}

\begin{proof}
The proof of this theorem is the same as of the usual inequality of Doob (see \cite[Theorem 11.2]{klenke_probability_2007} for example) using a general convex function instead of $g\left(x\right)=x^p$,
i.e.\ by \cite[Theorem 9.35]{klenke_probability_2007} we have that $\left(g\left(M_n\right)\right)_{n\in\mathbb{N}}$ is a submartingale 
if $\left(M_n\right)_{n\in\mathbb{N}}$ is a martingale, $g\colon\mathbb{R}\to\mathbb{R}$ is convex, and $\mathbb{E}\,\big(g\left(M_n\right)^+\big)<\infty$, for all $n\in\mathbb{N}$. 
We can assume the last inequality here. Otherwise the inequality in \eqref{eq: doob} holds trivially.
Furthermore, \cite[Lemma 11.1]{klenke_probability_2007} states that for a submartingale $\left(X_n\right)$ and $y>0$ we have by setting $X_n^*\coloneqq\sup_{k\leq n}\left|X_k\right|$ that
$y\cdot \mathbb{P}\left(\left|X_n^*\right|\geq y\right)\leq\mathbb{E}\left(\left|X_n\right|\cdot \mathbbm{1}_{\left\{X_n^*\geq y\right\}}\right)$.
Setting $y\coloneqq g\left(x\right)$ and $X_n\coloneqq g\left(\left|M_n\right|\right)$ yields the statement of the lemma.
\end{proof}

\begin{lemma}\label{lemma 2 2nd part}
Let $\left(\Omega, \mathcal{A},  T, \mu,\mathcal{F},\left\|\cdot\right\|\right)$ fulfill Property $\mathfrak{C}$. 
Then there exists $C>0$ such that
for all $\varphi\in\mathcal{F}$ fulfilling $\int\varphi\mathrm{d}\mu=0$,
all $n\in\N$, and all $t>0$
\begin{align}
\MoveEqLeft\mu\left(\left|\max_{1\leq j\leq n}\mathsf{S}_j\varphi\right|>u\right)
\leq 2\exp\left(t\cdot \left(-u+C\cdot\left\|\varphi\right\|\right)\right)\cdot \int\exp\left(t\cdot \mathsf{S}_{n}\varphi\right)\mathrm{d}\mu.\label{eq: first estim after Markov}
\end{align}
\end{lemma}

\begin{proof}
First define $\mathsf{R}\coloneqq\sum_{n=1}^\infty\widehat{ T}^n$.
We will now estimate $\mathsf{R}\varphi$ for $\int\varphi\mathrm{d}\mu=0$.
We notice that for all $A\in\mathcal{A}$ we have
\begin{align*}
\int\widehat{ T}^n\varphi\cdot \mathbbm{1}_A \mathrm{d}\mu&=\int\varphi\cdot \left(\mathbbm{1}_A\circ  T^n\right) \mathrm{d}\mu=\int\varphi\cdot \left(\mathbbm{1}_A\circ  T^n\right) \mathrm{d}\mu-\int\varphi\mathrm{d}\mu\cdot \int\mathbbm{1}_A \mathrm{d}\mu
\end{align*}
and from Lemma \ref{decay corr} it follows that there exist $ R>0$ and $0<\tau<1$ such that
$\int\widehat{ T}^n\varphi\cdot \mathbbm{1}_A \mathrm{d}\mu\leq  R\cdot \tau^n \cdot\left\|\varphi\right\|\cdot \left|\mathbbm{1}_A\right|_1$.
We set $ C\coloneqq 4 R\cdot\sum_{n=1}^{\infty} \tau^n$ and have
$\int \mathsf{R}\varphi\mathbbm{1}_A\mathrm{d}\mu\leq C/4\cdot\left\|\varphi\right\|\cdot\left|\mathbbm{1}_A\right|_1<\infty$,
for all $A\in\mathcal{A}$. Thus,
\begin{align}
\left|\mathsf{R}\varphi\right|_{\infty}\leq C/2\cdot\left\|\varphi\right\|<\infty.\label{R varphi infty}
\end{align}

Let $h\coloneqq \varphi + \mathsf{R}\varphi - \left(\mathsf{R}\varphi\right)\circ  T$. 
From \eqref{R varphi infty} and $\left|\cdot\right|_2\leq \left\|\cdot\right\|$ we have that $h\in \mathcal{L}^2$. 
By construction we have that $\mathsf{S}_nh = \sum_{j=0}^{n-1} h\circ  T^j$ is a martingale. 
Further, $\varphi = h + \left(\mathsf{R}\varphi\right)\circ  T - \mathsf{R}\varphi$ and thus,
\begin{align*}
\mathsf{S}_n\varphi &= \mathsf{S}_n h + \sum_{j=0}^{n-1} \left(\left(\mathsf{R}\varphi\right)\circ  T^{j+1} - \left(\mathsf{R}\varphi\right)\circ  T^j\right) 
= \mathsf{S}_nh+\left(\mathsf{R}\varphi\right)\circ  T^{n}-\left(\mathsf{R}\varphi\right).
\end{align*}
We have that
\begin{align*}
\mu\left(\max_{1\leq j\leq n}\mathsf{S}_j\varphi\geq u\right)
&=\mu\left(\max_{1\leq j\leq n}\left(\mathsf{S}_jh+\left(\mathsf{R}\varphi\right)\circ  T^{j}-\mathsf{R}\varphi\right)\geq u\right)\\
&\leq \mu\left(\max_{1\leq j\leq n}\mathsf{S}_jh>u-2\left|\mathsf{R}\varphi\right|_{\infty}\right).
\end{align*}
Applying
 the generalized inequality of Doob, see Lemma \ref{doob ineq}, the fact that $\mathsf{S}_{n} h$ is a martingale,
and \eqref{R varphi infty} yields  
 for all $t>0$
\begin{align}
\mu\left(\max_{1\leq j\leq n}\mathsf{S}_j\varphi\geq u\right)
&\leq \exp\left(t\cdot \left(-u+C/2\cdot \left\|\varphi\right\|\right)\right)\cdot \int\exp\left(t\cdot \left|\mathsf{S}_{n}h\right|\right)\mathrm{d}\mu.\label{eq: mu max Sj}
\end{align}
Furthermore, we obtain by the definition of $h$ and \eqref{R varphi infty} for the second factor that 
\begin{align*}
 \int\exp\left(t\cdot \left|\mathsf{S}_{n}h\right|\right)\mathrm{d}\mu
 &=\int\exp\left(t\cdot \left(\mathsf{S}_{n}\varphi-\left(\mathsf{R}\varphi\right)\circ  T^{n}+\mathsf{R}\varphi\right)\right)\mathrm{d}\mu\\
 &\leq \int\exp\left(t\cdot \left(\mathsf{S}_{n}\varphi+ 2\left|\mathsf{R}\varphi\right|_{\infty}\right)\right)\mathrm{d}\mu\\
 &\leq \int\exp\left(t\cdot \mathsf{S}_{n}\varphi\right)\mathrm{d}\mu\cdot\exp\left(t\cdot C/2\left\|\varphi\right\|\right).
\end{align*}
Combining this with \eqref{eq: mu max Sj} yields for all $t>0$
\begin{align*}
\mu\left(\max_{1\leq j\leq n}\mathsf{S}_j\varphi\geq u\right)\leq  \exp\left(t\cdot \left(-u+C\cdot \left\|\varphi\right\|\right)\right)\cdot \int\exp\left(t\cdot \left|\mathsf{S}_{n}h\right|\right)\mathrm{d}\mu.
\end{align*}
Analogously, we obtain
\begin{align*}
\mu\left(\max_{1\leq j\leq n}-\mathsf{S}_j\varphi\geq u\right)&\leq  \exp\left(t\cdot \left(-u+C\cdot \left\|\varphi\right\|\right)\right)\cdot \int\exp\left(t\cdot \left|\mathsf{S}_{n}h\right|\right)\mathrm{d}\mu.
\end{align*}
Thus, the statement of the lemma follows. 
\end{proof}

Our next lemma estimates the second factor of \eqref{eq: first estim after Markov}.
\begin{lemma}\label{e Snphi}
Let $\left(\Omega, \mathcal{A},  T, \mu,\mathcal{F},\left\|\cdot\right\|\right)$ fulfill Property $\mathfrak{C}$.
Let $\lambda_z$ and $\varkappa'$ be given as in Lemma \ref{lemma: resolvent}.
Then there exist $K', N>0$ such that for all $\varphi\in\mathcal{F}$ fulfilling $\int\varphi\mathrm{d}\mu=0$, all $n\in\N_{>N}$, and all $t\in\left(0, \varkappa'/\left\|\varphi\right\|\right)$
\begin{align*}
\int\exp\left(t\cdot \mathsf{S}_n\varphi\right)\mathrm{d}\mu
&\leq K'\cdot  \left|\upperleft{\varphi}{}{\lambda}_t\right|^n.
\end{align*}
\end{lemma}
\begin{proof}
We will first show  
\begin{align}
\widehat{ T}_t^n h = \widehat{ T}^n\left(\mathrm{e}^{t\cdot \mathsf{S}_n\varphi}\cdot h\right)\label{hat CYRI z n}
\end{align}
by induction.
The base case is obvious. Hence, assume that \eqref{hat CYRI z n} holds for $n-1$. This in combination with \eqref{hat CYRI} yields for all $g\in\mathcal{L}^1$ that 
\begin{align*}
\int \widehat{ T}_t^n h\cdot g\mathrm{d}\mu
&= \int\widehat{ T}_t\left(\widehat{ T}^{n-1}\left(\mathrm{e}^{t\cdot \mathsf{S}_{n-1}\varphi}\cdot h\right)\right)\cdot g\mathrm{d}\mu\\
&= \int\widehat{ T}\left(\mathrm{e}^{t\cdot \varphi}\cdot\widehat{ T}^{n-1}\left(\mathrm{e}^{t\cdot \mathsf{S}_{n-1}\varphi}\cdot h\right)\right)\cdot g\mathrm{d}\mu\\
&= \int\mathrm{e}^{t\cdot \varphi}\cdot\widehat{ T}^{n-1}\left(\mathrm{e}^{t\cdot \mathsf{S}_{n-1}\varphi}\cdot h\right)\cdot \left(g\circ  T\right)\mathrm{d}\mu\\
&= \int\mathrm{e}^{t\cdot \mathsf{S}_{n-1}\varphi}\cdot h\cdot\left(\left(\mathrm{e}^{t\cdot \varphi}\cdot \left(g\circ  T\right)\right)\circ  T^{n-1}\right)\mathrm{d}\mu\\
&= \int\mathrm{e}^{t\cdot \mathsf{S}_{n-1}\varphi}\cdot h\cdot\mathrm{e}^{t\cdot \left(\varphi\circ  T^{n-1}\right)}\cdot \left(g\circ  T^n\right)\mathrm{d}\mu\\
&= \int\mathrm{e}^{t\cdot \mathsf{S}_{n}\varphi}\cdot h\cdot \left(g\circ  T^n\right)\mathrm{d}\mu= \int\widehat{ T}^n\left(\mathrm{e}^{t\cdot \mathsf{S}_{n}\varphi}\cdot h\right)\cdot g\mathrm{d}\mu.
\end{align*}

Since we have by Lemma \ref{spect sep analyt} that
$\widehat{ T}_t^n=\lambda_t^n P_t+N_t^n$,  
it follows that
\begin{align*}
\int\mathrm{e}^{t\cdot \mathsf{S}_n\varphi}\mathrm{d}\mu&=\int \widehat{ T}_t^n 1\mathrm{d}\mu=\int \left(\lambda_t^n P_t 1+N_t^n 1\right)\mathrm{d}\mu
\end{align*}
and thus using \eqref{ineq <l} yields
\begin{align}
\int\mathrm{e}^{t\cdot \mathsf{S}_n\varphi}\mathrm{d}\mu
&\leq \left|\lambda_t\right|^n \cdot \left|P_t 1\right|_{\infty}+\left\|N_t\right\|^n. \label{lambda z N z 0}
\end{align}
We observe by the separation of the spectrum given in Lemma \ref{spect sep analyt} that 
$\left|\lambda_t\right|>1-\theta\geq 1-\left(1-\rho(N_0)\right)/2\geq 1/2$.
Combining this with \eqref{ineq <l} and \eqref{C 0 f} and 
the assumption $\left|t\right|\cdot\left\|\varphi\right\| < \varkappa'$ yields
\begin{align}
\left|P_t1\right|_{\infty}&\leq \frac{\left|\widehat{ T}_t1\right|_{\infty}}{\left|\lambda_t\right|}
\leq 2\cdot\left\|\widehat{ T}\exp\left(\left\|\varphi\right\|\cdot \left|t\right|\right)\right\|
\leq 2\cdot K_0\cdot\exp\left(\left\|\varphi\right\|\cdot \left|t\right|\right)
\leq 2\cdot K_0\cdot\exp\left(\varkappa'\right).\label{Pz estim}
\end{align}
Combining this calculation with \eqref{lambda z N z 0}
yields 
\begin{align*}
 \int\exp\left(t\cdot \mathsf{S}_n\varphi\right)\mathrm{d}\mu
&\leq \left|\lambda_t\right|^n
\cdot \left(2\cdot K_0\cdot\exp\left(\varkappa'\right)+\left|\lambda_t\right|^{-n}\cdot \left\|N_t^n\right\|\right).
\end{align*}
Applying Lemma \ref{lemma: resolvent}
gives the statement of the lemma with $K'=2\cdot K_0\cdot\exp\left(\varkappa'\right)+1$.
\end{proof}

The next lemma gives an estimate of the $n$th derivative of $\lambda_t$ at $t=0$ which will help us later to estimate $\lambda_t$.
\begin{lemma}\label{lambda^n}
Let $\left(\Omega, \mathcal{A},  T, \mu,\mathcal{F},\left\|\cdot\right\|\right)$ fulfill Property $\mathfrak{C}$.
Let $\lambda_z$ be the eigenvalue of the perturbed transfer operator $\widehat{T}_z$ given in Lemma \ref{spect sep analyt}.
Then we have $\lambda_0^{\left(0\right)}=1,\lambda_0^{\left(1\right)}=0$, and for $n\geq 2$ there exists $\eta<\infty$ such that for all $\varphi\in\mathcal{F}$ with $\int \varphi\mathrm{d}\mu=0$
we have that
\begin{align}
\upperleft{\varphi}{}{\lambda}^{\left(n\right)}_0\leq n!\cdot \eta^n\cdot\left|\varphi\right|_1\cdot\left\|\varphi\right\|^{n-1}.\label{pi lambda}
\end{align}
\end{lemma}

\begin{proof}
We assume in the sequel that $\left|z\right|\cdot\left\|\varphi\right\|<\varkappa$ with $\varkappa$ as in Lemma \ref{spect sep analyt}. 
By Lemma \ref{spect sep analyt} we have that 
$\widehat{ T}_z P_z=\left(\lambda_z P_z+N_z\right)P_z=\lambda_z P_z$ with all parts being analytic.  
Dif\-fe\-ren\-tia\-ting both sides  
$n$ times and using \eqref{Tzk} we get by induction that
\begin{align}
\sum_{i=0}^{n}\binom{n}{i}\widehat{ T}_z M_{\varphi}^iP_z^{\left(n-i\right)}
&=\sum_{j=0}^{n}\binom{n}{j}\lambda^{\left(j\right)}_z P^{\left(n-j\right)}_z.\label{TMP}
\end{align}

Applying \eqref{TMP} to the constant one function, multiplying by a test function $h$, and integrating yields
\begin{align*}
\sum_{i=0}^{n}\binom{n}{i}\int \left(\widehat{ T}M_{\varphi}^iP_z^{\left(n-i\right)}1\right)\cdot h\mathrm{d}\mu
&=\sum_{j=0}^{n}\binom{n}{j}\lambda^{\left(j\right)}_z \int P^{\left(n-j\right)}_z1\cdot h\mathrm{d}\mu.
\end{align*}

Recalling the connection between $\widehat{ T}$ and $ T$ from \eqref{hat CYRI} yields
\begin{align*}
\sum_{i=0}^{n}\binom{n}{i}\int \varphi^i\cdot P_z^{\left(n-i\right)}1\cdot\left(h\circ  T\right)\mathrm{d}\mu
&=\sum_{j=0}^{n}\binom{n}{j}\lambda^{\left(j\right)}_z \int P^{\left(n-j\right)}_z1\cdot h\mathrm{d}\mu.
\end{align*}

Setting $h\equiv 1$ we have that
\begin{align}
\sum_{i=0}^{n}\binom{n}{i}\int \varphi^i\cdot P_z^{\left(n-i\right)}1\mathrm{d}\mu
&=\sum_{j=0}^{n}\binom{n}{j}\lambda^{\left(j\right)}_z \int P^{\left(n-j\right)}_z1\mathrm{d}\mu.\label{eq: Aufspaltung Bew}
\end{align}
If we set $n=1$, then
\begin{align*}
\int \varphi\cdot P_z1\mathrm{d}\mu+\int P_z^{\left(1\right)}1\mathrm{d}\mu
&=\lambda^{\left(1\right)}_z \int P_z1\mathrm{d}\mu+\lambda_z \int P^{\left(1\right)}_z1\mathrm{d}\mu
\end{align*}
and thus, with $P_0 h=\int h\mathrm{d}\mu$, see Lemma \ref{Pf=intf}, i.e.\ $P_0 1=1$, $\int\varphi\mathrm{d}\mu=0$, and $\lambda_0=1$
it follows that $\lambda_0^{\left(1\right)}=0$.

From \eqref{eq: Aufspaltung Bew} it follows then
for $n\geq 2$
\begin{align}
\lambda^{\left(n\right)}_0&=\lambda^{\left(n\right)}_0\int P_0 1\mathrm{d}\mu\notag\\
&=\sum_{i=0}^{n}\binom{n}{i}\int\varphi^i\cdot P_0^{\left(n-i\right)}1\mathrm{d}\mu-\sum_{j=0}^{n-1}\binom{n}{j}\lambda^{\left(j\right)}_0\int P^{\left(n-j\right)}_0 1\mathrm{d}\mu\notag\\
&=\sum_{i=1}^{n}\binom{n}{i}\int\varphi^i\cdot P_0^{\left(n-i\right)}1\mathrm{d}\mu-\sum_{j=0}^{n-1}\binom{n}{j}\lambda^{\left(j\right)}_0\int P^{\left(n-j\right)}_0 1\mathrm{d}\mu.\label{lambda00}
\end{align}

In the next steps we will estimate $P_0^{\left(k\right)}1\left(\omega\right)$. 
Remember that by \eqref{Pz estim} we have 
$\left|P_z1\right|_{\infty}
\leq K'$ if $\left|z\right|\cdot\left\|\varphi\right\|<\varkappa$.
Since $P_z$ is analytic on a circle around zero with radius $r\leq \varkappa/\left\|\varphi\right\|$ it follows
by Cauchy's integral formula and the maximum principle that 
\begin{align*}
\left|\upperleft{\varphi}{}{P_0}^{\left(n\right)}1\right|_{\infty}
&\leq n!\cdot 2\cdot K_0\cdot\exp\left(\varkappa\right) \cdot \left(\frac{\left\|\varphi\right\|}{\varkappa}\right)^{n}.
\end{align*}
We notice that
\begin{align*}
 n!\cdot2\cdot K_0\cdot \exp\left(\varkappa\right)\cdot \varkappa^{-n}
 \leq n!\cdot \left(2\cdot K_0\cdot \exp\left(\varkappa\right)\cdot \max\left\{1,\varkappa^{-1}\right\}\right)^n
 \eqqcolon n!\cdot a^n.
\end{align*}
Hence, we can estimate the first sum of \eqref{lambda00} by
\begin{align}
\sum_{i=1}^{n}\binom{n}{i}\int\varphi^i\cdot P_0^{\left(n-i\right)}1\mathrm{d}\mu&\leq \sum_{i=1}^{n}\binom{n}{i}\left|\varphi^i\right|_1\cdot \left|P_0^{\left(n-i\right)}1\right|_{\infty}\notag\\
&\leq \sum_{i=1}^{n}\binom{n}{i}\left|\varphi\right|_1^i\cdot (n-i)!\cdot a^{n-i}\cdot \left\|\varphi\right\|^{n-i}\notag\\
&\leq \sum_{i=1}^{n}\binom{n}{i} (n-i)!\cdot a^{n-i}\cdot \left|\varphi\right|_1 \cdot\left\|\varphi\right\|^{n-1}.\label{eq: eta first summand0}
\end{align}
A further estimate yields 
\begin{align}
\sum_{i=1}^{n}\binom{n}{i}(n-i)!\cdot a^{n-i}
\leq \sum_{i=1}^{n}n!\cdot a^{n-i}
\leq n!\cdot a^{n-1}\sum_{i=0}^{\infty}a^{-i}
=\frac{n!\cdot a^{n-1}}{1-a^{-1}}
\leq n!\cdot a^n,\label{eq: eta first summand}
\end{align}
since $a\geq 2$. 

In the sequel we will prove that \eqref{pi lambda} holds for $\eta=3a$ using an induction argument. 
The base case for $n=2$ is obvious since the minuend in \eqref{lambda00} equals zero in that case
and thus $\lambda^{(2)}\leq 2!\cdot a^2\cdot \left|\varphi\right|_1\cdot \left\|\varphi\right\|
\leq 2!\cdot b^2\cdot \left|\varphi\right|_1\cdot \left\|\varphi\right\|$.
By assuming that \eqref{pi lambda} holds for all $k\leq n-1$ we have for the second sum in \eqref{lambda00} that 
\begin{align}
\sum_{j=2}^{n-1}\binom{n}{j}\lambda^{\left(j\right)}_0\int P^{\left(n-j\right)}_0 1\mathrm{d}\mu
&\leq  \sum_{j=2}^{n-1}\binom{n}{j}j!\cdot \eta^j\cdot \left|\varphi\right|_1\cdot\left\|\varphi\right\|^{j-1}\cdot \left|P^{\left(n-j\right)}_0 1\right|_{\infty}\notag\\
&\leq \sum_{j=2}^{n-1}\binom{n}{j}j!\cdot \eta^j\cdot \left|\varphi\right|_1\cdot\left\|\varphi\right\|^{j-1}\cdot (n-j)!\cdot a^{n-j}\cdot\left\|\varphi\right\|^{n-j}\notag\\
&=n!\cdot \left|\varphi\right|_1\cdot\left\|\varphi\right\|^{n-1}\cdot \sum_{j=2}^{n-1} \eta^j\cdot a^{n-j}\label{eq: eta second summand}
\end{align}
Furthermore,
\begin{align*}
\sum_{j=0}^{n-1} \eta^j\cdot a^{n-j}
&\leq \eta^n\cdot \sum_{j=0}^{n-1} \left(\frac{a}{\eta}\right)^{n-j}
\leq  \eta^n\cdot \sum_{j=1}^{\infty} \left(\frac{1}{3}\right)^{j}
=\frac{1}{2}\cdot \eta^n.
\end{align*}
Combining this with \eqref{lambda00}, \eqref{eq: eta first summand0}, \eqref{eq: eta first summand}, and \eqref{eq: eta second summand} yields
\begin{align*}
 \lambda^{\left(n\right)}_0
 &\leq n!\cdot a^n\cdot \left|\varphi\right|_1\cdot\left\|\varphi\right\|
 +n!\cdot \frac{1}{2}\cdot b^n\cdot \left|\varphi\right|_1\cdot\left\|\varphi\right\|
 \leq n!\cdot b^n\cdot \left|\varphi\right|_1\cdot\left\|\varphi\right\|,
\end{align*}
which gives the statement of the lemma. 
\end{proof}

\begin{proof}[Proof of Lemma \ref{lemma: Tnfn chi deviation allg}] 
We assume that $t\in\left(0, \varkappa'/\left\|\varphi\right\|\right)$.
Using a combination of 
Lemma \ref{lemma 2 2nd part} and Lemma \ref{e Snphi}
implies that there exists $N>0$ as in Lemma \ref{e Snphi} such that for all $n\in\N_{>N}$
\begin{align*}
 \mu\left(\left|\mathsf{S}_n\varphi\right|> u\right)
 &\leq 2K'\cdot\exp\left(t \cdot \left(- u+C\cdot\left\|\varphi\right\|\right)\right)\cdot \left|\lambda_{t}\right|^n,
\end{align*}
with $C$ as in Lemma \ref{lemma 2 2nd part} and $K'$ as in Lemma \ref{e Snphi}.
Furthermore, by our choice of $t$ we have that 
$C\cdot \left\|\varphi\right\|\cdot t$ is bounded by $C\cdot \varkappa'$.
This implies
\begin{align}
 \mu\left(\left|\mathsf{S}_n\varphi\right|> u\right)
 &\leq K\cdot\exp\left(-t\cdot u\right)\cdot \left|\lambda_{t}\right|^n,\label{eq: ineq for proof of 4.1}
\end{align}
with $K\coloneqq 2K'\cdot \exp\left(C\cdot \varkappa'\right)$.

From Lemma \ref{spect sep analyt} and Lemma \ref{lambda^n}
it follows for $t\in\left(0,\varkappa'/\left\|\varphi\right\|\right)$
that 
\begin{align*}
\left|\lambda_{t}\right|&\leq \left|\sum_{k=0}^{\infty}\frac{t^k\cdot\lambda_0^{\left(k\right)}}{k!}\right|=1+\sum_{k=2}^{\infty}\frac{t^k\cdot \left|\lambda_0^{\left(k\right)}\right|}{k!}
\leq 1+\sum_{k=2}^{\infty}t^k\cdot\eta^k\cdot \left|\varphi\right|_1\cdot \left\|\varphi\right\|^{k-1}\\
&=1+t^2\cdot \eta^2\cdot \left|\varphi\right|_1\cdot \left\|\varphi\right\|\cdot \sum_{k=2}^{\infty}t^{k-2}\cdot \eta^{k-2}\cdot \left\|\varphi\right\|^{k-2}.
\end{align*}
For the following we assume that $t\cdot \left\|\varphi\right\|\leq \min\left\{\varkappa', \left(2\eta\right)^{-1}\right\}$. 
This implies
$\eta\cdot \left\|\varphi\right\|\cdot t<1/2$ and we can apply the geometric series formula obtaining
\begin{align*}
\left|\lambda_{t}\right|\leq 1+\frac{t^2\cdot\eta^2\cdot \left|\varphi\right|_1\cdot \left\|\varphi\right\|}{1-t\cdot\eta\cdot \left\|\varphi\right\|}
\leq 1+t^2\cdot 2\eta^2\cdot \left|\varphi\right|_1\cdot \left\|\varphi\right\|.
\end{align*}
Thus, 
\begin{align}
\left|\lambda_{t}\right|^n
<\exp\left(t^2\cdot n\cdot 2\eta^2\cdot \left|\varphi\right|_1\cdot \left\|\varphi\right\|\right).\label{Markov 1b}
\end{align}
Combining \eqref{eq: ineq for proof of 4.1} and \eqref{Markov 1b} yields 
\begin{align*}
\mu\left(\max_{i\leq n}\left|\mathsf{S}_i\varphi\right|\geq  u\right)
&\leq K\cdot \exp\left(-t\cdot u+t^2\cdot n\cdot 2\eta^2\cdot \left|\varphi\right|_1\cdot \left\|\varphi\right\|\right).
\end{align*}
We set for the following 
\begin{align*}
 U\coloneqq\frac{1}{2}\cdot \min\left\{\left(4\cdot\eta^2\right)^{-1}, \varkappa', \left(2\eta\right)^{-1}\right\}
\end{align*}
and
\begin{align}
 t=2\cdot U\cdot \min\left\{\frac{u}{n\cdot\left|\varphi\right|_1}, 1\right\}\cdot \left\|\varphi\right\|^{-1}.\label{eq: def t}
\end{align}
This definition of $t$ ensures that $\eta\cdot \left\|\varphi\right\|\cdot t<1/2$
and $t\cdot \left\|\varphi\right\|\leq \varkappa'$.
Assume the minimum is attained at the first entry in \eqref{eq: def t}. Then 
\begin{align*}
\mu\left(\max_{i\leq n}\left|\mathsf{S}_i\varphi\right|\geq  u\right)
&\leq K\cdot \exp\left(-U\cdot \frac{u^2}{n\cdot\left|\varphi\right|_1\cdot \left\|\varphi\right\|}\right).
\end{align*}
If the minimum in \eqref{eq: def t} is attained at the second entry, then
$t\cdot u>1/2\cdot t^2\cdot n\cdot 2\eta^2\cdot \left|\varphi\right|_1\cdot \left\|\varphi\right\|$
and thus
\begin{align*}
\mu\left(\max_{i\leq n}\left|\mathsf{S}_i\varphi\right|\geq  u\right)
&\leq K\cdot \exp\left(-U \cdot \frac{u}{\left\|\varphi\right\|}\right).
\end{align*}
These two considerations give the statement of the lemma.
\end{proof}

\subsection{Proof of theorems concerning the truncated sum \texorpdfstring{$\mathsf{T}_n^{f_n}\chi$}{}}\label{proofs transfer A}
In this section we will prove the theorems given in Section \ref{subsec: intro truncated rv}.
We will start with a lemma to be used in their proofs.
\begin{lemma}\label{lemma: Tnfn chi deviation}
Let $\left(\Omega, \mathcal{A},  T, \mu,\mathcal{F},\left\|\cdot\right\|,\chi\right)$ fulfill Property $\mathfrak{D}$.
Then there exist constants $N'\in\mathbb{N}$ and $E, K>0$ such that for all $\epsilon\in (0,E)$, $r>0$ with $F(r)>0$, and $n\in\mathbb{N}_{>N'}$
\begin{align*}
\MoveEqLeft\mu\left(\max_{i\leq n}\left|\mathsf{T}^{r}_{i}\chi-\int\mathsf{T}_{i}^{r}\chi\mathrm{d}\mu\right|\geq \epsilon\cdot \int\mathsf{T}^{r}_{n}\chi\mathrm{d}\mu\right)
\leq K\cdot\exp\left(-\epsilon \cdot \frac{\int\mathsf{T}^{r}_{n}\chi\mathrm{d}\mu}{r}\right).
\end{align*}
\end{lemma}
 
\begin{proof}
We denote for the following $\upperleft{r}{}{\overline{\chi}}\coloneqq \upperleft{r}{}{\chi}-\int \upperleft{r}{}{\chi}\mathrm{d}\mu$.   
Applying Lemma \ref{lemma: Tnfn chi deviation allg} yields
\begin{align*}
\MoveEqLeft\left(\max_{i\leq n}\left|\mathsf{T}^{r}_{i}\chi-\int\mathsf{T}_{i}^{r}\chi\mathrm{d}\mu\right|\geq \epsilon\cdot \int\mathsf{T}^{r}_{n}\chi\mathrm{d}\mu\right)\\
&\leq K\exp\left(- U\cdot \frac{\epsilon\cdot \int\mathsf{T}^{r}_{n}\chi\mathrm{d}\mu}{\left\|\upperleft{r}{}{\overline{\chi}}\right\|}\cdot \min\left\{\frac{\epsilon\cdot \int\mathsf{T}^{r}_{n}\chi\mathrm{d}\mu}{n\cdot \left|\upperleft{r}{}{\overline{\chi}}\right|_1},1\right\}\right),
\end{align*}
for $n\geq N$ given in Lemma \ref{lemma: Tnfn chi deviation allg}.
We have that $n\cdot \left|\upperleft{r}{}{\overline{\chi}}\right|_1\geq 1/2\cdot\int\mathsf{T}^{r}_{n}\chi\mathrm{d}\mu\geq \epsilon\int\mathsf{T}^{r}_{n}\chi\mathrm{d}\mu$
if $\epsilon\leq 1/2$.
Furthermore, from \eqref{C 1} we obtain
\begin{align*}
 \left\|\upperleft{r}{}{\overline{\chi}}\right\|=\left\|\upperleft{r}{}{\chi}-\int \upperleft{r}{}{\chi}\mathrm{d}\mu\right\|
 \leq \left\|\upperleft{r}{}{\chi}\right\|+\left\|\int \upperleft{r}{}{\chi}\mathrm{d}\mu\right\|
 \leq K_1\cdot r+\left\|r\right\|
 = \left(K_1+\left\|1\right\|\right)\cdot r.
\end{align*}
Thus,
\begin{align*}
\mu\left(\max_{i\leq n}\left|\mathsf{T}^{r}_{i}\chi-\int\mathsf{T}_{i}^{r}\chi\mathrm{d}\mu\right|\geq \epsilon\cdot \int\mathsf{T}^{r}_{n}\chi\mathrm{d}\mu\right)
&\leq K\exp\left(-\frac{U\cdot \epsilon^2}{K_1+\left\|1\right\|}\cdot \frac{\int\mathsf{T}^{r}_{n}\chi\mathrm{d}\mu}{r}\right),
\end{align*}
for $n\geq N$ given in Lemma \ref{lemma: Tnfn chi deviation allg}.

If $\epsilon$ is sufficiently small, then 
$\epsilon^2\cdot U_1\cdot \left(2\cdot \left(K_1+\left\|1\right\|\right)\cdot r\right)^{-1}$
and we obtain the statement of the lemma.
\end{proof}

\begin{proof}[Proof of Theorem \ref{Thm: Sn* allg}]
We assume that \eqref{cond a} is fulfilled for a sequence $\left(f_n\right)$ with $F\left(f_n\right)>0$, for all $n\in\mathbb{N}$. Further, we note that 
this condition
implies that for all $\epsilon>0$ there exists $N\in\N$ such that
\begin{align}
\epsilon\cdot\frac{\int\mathsf{T}_{n}^{f_{n}}\chi\mathrm{d}\mu}{f_{n}}\geq\log\psi\left(n\right),\label{xi E tn}
\end{align}
for some $\psi\in\Psi$ and all $n\geq N$.
Fix $\epsilon>0$ and apply Lemma \ref{lemma: Tnfn chi deviation} to obtain
\begin{align*}
\mu\left(\left|\mathsf{T}_{n}^{f_{n}}\chi-\int\mathsf{T}_{n}^{f_{n}}\chi\mathrm{d}\mu\right|\geq\epsilon\int\mathsf{T}_{n}^{f_{n}}\chi\mathrm{d}\mu\right) 
& \leq K_{\epsilon}\cdot \exp\left(-\epsilon\cdot\frac{\int\mathsf{T}_{n}^{f_{n}}\chi\mathrm{d}\mu}{f_{n}}\right),
\end{align*}
for $n$ sufficiently large.
From \eqref{xi E tn} we can conclude that the right hand side is summable
and obtain by the Borel-Cantelli lemma that 
$\mu\left(\left|\mathsf{T}_{n}^{f_{n}}\chi-\int\mathsf{T}_{n}^{f_{n}}\chi\mathrm{d}\mu\right|\geq\epsilon\int\mathsf{T}_{n}^{f_{n}}\chi\mathrm{d}\mu\text{ i.o.}\right)=0.$ 
Since $\epsilon>0$ is arbitrary, it follows that $\left|\mathsf{T}_{n}^{f_{n}}\chi-\int\mathsf{T}_{n}^{f_{n}}\chi\mathrm{d}\mu\right|=o\left(\int\mathsf{T}_{n}^{f_{n}}\chi\mathrm{d}\mu\right)$
almost surely and hence the assertion of the theorem.
\end{proof}

\begin{proof} [Proof of Theorem \ref{Thm: Sn* reg var}]
We define the sequences $\left(g_{n}\right)_{n\in\mathbb{N}}$ and $\left(\bar{g}_{n}\right)_{n\in\mathbb{N}}$
with $g_{n}\coloneqq\max\left\{ f_{n},n\right\} $
and $\bar{g}_{n}\coloneqq\min\left\{ f_{n},n\right\}$.
Further, set $I_{j}\coloneqq\left[2^{j},2^{j+1}-1\right]$ for $j\in\mathbb{N}$
and for given $\epsilon>0$ we set $\rho_{k}\coloneqq\rho_{k}\left(\epsilon\right)\coloneqq\left(1+\epsilon\right)^{k}$
and 
the sequences $\left(q_{j}\right)_{j\in\mathbb{N}}$ and
$\left(r_{j}\right)_{j\in\mathbb{N}}$ as
\begin{align*}
q_{j}\coloneqq\left\lfloor \frac{j\cdot\log2}{\log\left(1+\epsilon\right)}\right\rfloor 
\text{ and }
r_{j}\coloneqq{\left\lceil \frac{\log\left(\max_{n\in I_{j}}g_{n}\right)}{\log\left(1+\epsilon\right)}\right\rceil }.
\end{align*}
These numbers are chosen such that 
\begin{align}
\left[\rho_{q_j},\rho_{r_j}\right]\supset\left[\min_{n\in I_{j}}g_n,\max_{n\in I_{j}}g_n\right].\label{eq: superset}
\end{align}

We will split the proof of this theorem as follows:
\begin{enumerate}[label=(\Alph*)]
 \item\label{en: trunc A}
There exist $E',N>0$ such that for all $j\geq N$ and $\epsilon<E'$ 
\begin{align*}
 \MoveEqLeft\bigcup_{n\in I_j}\left\{\left|\mathsf{T}_n^{g_n}\chi-\int\mathsf{T}_n^{g_n}\chi\mathrm{d}\mu\right|>4\epsilon\cdot \int\mathsf{T}_n^{g_n}\chi\mathrm{d}\mu\right\}\\
 &\subset \bigcup_{k=q_j}^{r_j}\left\{\max_{n\in I_j}\left|\mathsf{T}_n^{\rho_k}\chi-\int\mathsf{T}_n^{\rho_k}\chi\mathrm{d}\mu\right|>\epsilon\cdot \int\mathsf{T}_{2^{j+1}-1}^{\rho_k}\chi\mathrm{d}\mu\right\}.
\end{align*}
\item\label{en: trunc B}  
We have that 
\begin{align}
 \sum_{j=1}^{\infty}\sum_{k=q_j}^{r_j}\mu\left(\max_{n\in I_j}\left|\mathsf{T}_n^{\rho_k}\chi-\int\mathsf{T}_n^{\rho_k}\chi\mathrm{d}\mu\right|>\epsilon\cdot \int\mathsf{T}_{2^{j+1}-1}^{\rho_k}\chi\mathrm{d}\mu\right)<\infty.\label{eq: sum for BC}
\end{align}

\item\label{en: trunc C} 
For all $\epsilon>0$ we have
\begin{align*}
\mu\left(\left|\mathsf{T}_{n}^{\bar{g}_{n}}\chi-\int\mathsf{T}_{n}^{\bar{g}_{n}}\chi\mathrm{d}\mu\right|\geq\epsilon\int\mathsf{T}_{n}^{\bar{g}_{n}}\chi\mathrm{d}\mu\text{ i.o.}\right)=0.
\end{align*}
\end{enumerate}
Applying the Borel-Cantelli lemma on \ref{en: trunc B} and combining this with \ref{en: trunc A} using \eqref{eq: superset} yields 
\begin{align*}
\mu\left(\left|\mathsf{T}_{n}^{g_{n}}\chi-\int\mathsf{T}_{n}^{g_{n}}\chi\mathrm{d}\mu\right|\geq\epsilon\int\mathsf{T}_{n}^{g_{n}}\chi\mathrm{d}\mu\text{ i.o.}\right)=0.
\end{align*}
Combining this with \ref{en: trunc C} gives the statement of the theorem.

\emph{Proof of \ref{en: trunc A}}.
We have that \eqref{E X*} implies 
\begin{align*}
\int\mathsf{T}_{n}^{\rho_{k+1}}\chi\mathrm{d}\mu 
& \sim n\cdot\frac{\alpha}{1-\alpha}\cdot\rho_{k+1}^{1-\alpha}L\left(\rho_{k+1}\right)
  \sim n\cdot\frac{\alpha}{1-\alpha}\cdot \left(1+\epsilon\right)^{1-\alpha}\rho_k^{1-\alpha}L\left(\rho_k\right)\\
&  \sim \left(1+\epsilon\right)^{1-\alpha}\int\mathsf{T}_{n}^{\rho_k}\chi\mathrm{d}\mu,
\end{align*}
for $k$ tending to infinity.
Choosing $k$ such that $\rho_k\leq g_n\leq \rho_{k+1}$ implies
\begin{align*}
\left(1+\epsilon\right)\int\mathsf{T}_{n}^{\rho_k}\chi\mathrm{d}\mu
\geq\int\mathsf{T}_{n}^{g_{n}}\chi\mathrm{d}\mu
\geq\int\mathsf{T}_{n}^{\rho_{k+1}}\chi\mathrm{d}\mu/\left(1+\epsilon\right),
\end{align*}
for $n$ sufficiently large since $g_n$ tends to infinity. 
We assume for the following that $\epsilon<1/4$.
Then we obtain the following implications
\begin{align}
 &  & \int\mathsf{T}_{n}^{g_n}\chi\mathrm{d}\mu-\mathsf{T}_{n}^{g_n}\chi & >4\epsilon\cdot\int\mathsf{T}_{n}^{g_n}\chi\mathrm{d}\mu\notag\\
 & \Longrightarrow & \int\mathsf{T}_{n}^{g_n}\chi\mathrm{d}\mu-\mathsf{T}_{n}^{\rho_k}\chi & >4\epsilon\cdot\int\mathsf{T}_{n}^{\rho_k}\chi\mathrm{d}\mu\notag\\
 & \Longrightarrow & \int\mathsf{T}_{n}^{\rho_k}\chi\mathrm{d}\mu-\mathsf{T}_{n}^{\rho_k}\chi & >4\epsilon\cdot\int\mathsf{T}_{n}^{\rho_k}\chi\mathrm{d}\mu-\left(\int\mathsf{T}_{n}^{g_{n}}\chi\mathrm{d}\mu-\int\mathsf{T}_{n}^{\rho_k}\chi\mathrm{d}\mu\right)\notag\\
 &  &  & \geq 4\epsilon\cdot\int\mathsf{T}_{n}^{\rho_k}\chi\mathrm{d}\mu-\epsilon\cdot \int\mathsf{T}_{n}^{\rho_k}\chi\mathrm{d}\mu
=3\epsilon\cdot\int\mathsf{T}_{n}^{\rho_k}\chi\mathrm{d}\mu\notag
\end{align}
and thus 
\begin{align}
\left\{\int\mathsf{T}_{n}^{g_n}\chi\mathrm{d}\mu-\mathsf{T}_{n}^{g_n}\chi >4\epsilon\cdot\int\mathsf{T}_{n}^{g_n}\chi\mathrm{d}\mu\right\}
\subset \left\{\int\mathsf{T}_{n}^{\rho_k}\chi\mathrm{d}\mu-\mathsf{T}_{n}^{\rho_k}\chi\leq 2\epsilon\cdot\int\mathsf{T}_{n}^{\rho_k}\chi\mathrm{d}\mu\right\}.
\label{ETn-Tn}
\end{align}

Analogously to the situation above we obtain 
\begin{align}
 &  & \mathsf{T}_{n}^{g_n}\chi-\int\mathsf{T}_{n}^{g_n}\chi\mathrm{d}\mu & >4\epsilon\cdot\int\mathsf{T}_{n}^{g_n}\chi\mathrm{d}\mu\notag\\
 & \Longrightarrow & \mathsf{T}_{n}^{\rho_{k+1}}\chi-\int\mathsf{T}_{n}^{g_n}\chi\mathrm{d}\mu & >4\epsilon\cdot\int\mathsf{T}_{n}^{\rho_{k+1}}\chi\mathrm{d}\mu
 -4\epsilon\cdot \left(\int\mathsf{T}_{n}^{\rho_{k+1}}\chi\mathrm{d}\mu-\int\mathsf{T}_{n}^{g_n}\chi\mathrm{d}\mu\right)\notag\\
 & \Longrightarrow & \mathsf{T}_{n}^{\rho_{k+1}}\chi-\int\mathsf{T}_{n}^{\rho_{k+1}}\chi\mathrm{d}\mu & >4\epsilon\cdot\int\mathsf{T}_{n}^{\rho_{k+1}}\chi\mathrm{d}\mu
 -\left(1+4\epsilon\right)\cdot \left(\int\mathsf{T}_{n}^{\rho_{k+1}}\chi\mathrm{d}\mu-\int\mathsf{T}_{n}^{g_n}\chi\mathrm{d}\mu\right)
 \notag\\
 &\Longrightarrow & \mathsf{T}_{n}^{\rho_{k+1}}\chi-\int\mathsf{T}_{n}^{\rho_{k+1}}\chi\mathrm{d}\mu & >\left(4\epsilon -\epsilon\cdot\left(1+4\epsilon\right)\right)\cdot
\int\mathsf{T}_{n}^{\rho_{k+1}}\chi\mathrm{d}\mu
\geq 2\epsilon\cdot\int\mathsf{T}_{n}^{\rho_{k+1}}\chi\mathrm{d}\mu\notag
\end{align}
and thus 
\begin{align}
\left\{\mathsf{T}_{n}^{g_n}\chi-\int\mathsf{T}_{n}^{g_n}\chi\mathrm{d}\mu  >4\epsilon\cdot\int\mathsf{T}_{n}^{g_n}\chi\mathrm{d}\mu\right\}
\subset \left\{\mathsf{T}_{n}^{\rho_{k+1}}\chi-\int\mathsf{T}_{n}^{\rho_{k+1}}\chi\mathrm{d}\mu>2\epsilon\cdot\int\mathsf{T}_{n}^{\rho_{k+1}}\chi\mathrm{d}\mu\right\}
.\label{Tn-ETn}
\end{align}

Furthermore, we have for every $m\in I_{j}$ and $k\in\N$ that 
\begin{align}
\left\{ \left|\mathsf{T}_{m}^{\rho_k}\chi-\int\mathsf{T}_{m}^{\rho_k}\chi\mathrm{d}\mu\right|\geq 2\epsilon\int\mathsf{T}_{m}^{\rho_k}\chi\mathrm{d}\mu\right\}
 & \subset\left\{ \left|\mathsf{T}_{m}^{\rho_k}\chi-\int\mathsf{T}_{m}^{\rho_k}\chi\mathrm{d}\mu\right|\geq 2\epsilon\int\mathsf{T}_{2^{j}}^{\rho_k}\chi\mathrm{d}\mu\right\} \notag\\
 & \subset\left\{ \max_{n\in I_{j}}\left|\mathsf{T}_{n}^{\rho_k}\chi-\int\mathsf{T}_{n}^{\rho_k}\chi\mathrm{d}\mu\right|\geq2\epsilon\int\mathsf{T}_{2^{j}}^{\rho_k}\chi\mathrm{d}\mu\right\} \notag\\
 & \subset\left\{ \max_{n\in I_{j}}\left|\mathsf{T}_{n}^{\rho_k}\chi-\int\mathsf{T}_{n}^{\rho_k}\chi\mathrm{d}\mu\right|\geq\epsilon\int\mathsf{T}_{2^{j+1}-1}^{\rho_k}\chi\mathrm{d}\mu\right\}.\label{max, nicht max}
\end{align}
Combining \eqref{ETn-Tn}, \eqref{Tn-ETn}, and \eqref{max, nicht max} yields \ref{en: trunc A}.

\emph{Proof of \ref{en: trunc B}}.
We have
that $\upperleft{\rho_{k}}{}{\chi}\leq\rho_{k}$. Hence, we can apply Lemma \ref{lemma: Tnfn chi deviation}
assuming that $\epsilon<E$
to the sum $\mathsf{T}_{n}^{\rho_{k}}\chi$ and obtain 
for $j$ sufficiently large 
\begin{align}
\mu\left( \max_{n\in I_{j}}\left|\mathsf{T}_{n}^{\rho_k}\chi-\int\mathsf{T}_{n}^{\rho_k}\chi\mathrm{d}\mu\right|\geq\epsilon\int\mathsf{T}_{2^{j+1}-1}^{\rho_k}\chi\mathrm{d}\mu\right)
\leq K_{\epsilon}\cdot\exp\left(-\epsilon\cdot\frac{\int\mathsf{T}_{2^{j+1}-1}^{\rho_{k}}\chi\mathrm{d}\mu}{\rho_{k}}\right).\label{eq: Tnrhok} 
\end{align}
Using \eqref{E S*} we have for $k$ tending to infinity that 
\begin{align*}
\int\mathsf{T}_{2^{j+1}-1}^{\rho_{k}}\chi\mathrm{d}\mu\sim\left(2^{j+1}-1\right)\cdot\frac{\alpha}{1-\alpha}\cdot\rho_{k}^{1-\alpha}\cdot L\left(\rho_{k}\right).
\end{align*}
Hence, we obtain by \eqref{eq: Tnrhok} that 
\begin{align*}
\mu\left( \max_{n\in I_{j}}\left|\mathsf{T}_{n}^{\rho_k}\chi-\int\mathsf{T}_{n}^{\rho_k}\chi\mathrm{d}\mu\right|\geq\epsilon\int\mathsf{T}_{2^{j+1}-1}^{\rho_k}\chi\mathrm{d}\mu\right)
& \leq K_{\epsilon}\cdot \exp\left(-\epsilon\cdot\frac{\alpha}{1-\alpha}\cdot\frac{2^j\cdot L\left(\rho_{k}\right)}{\rho_{k}^{\alpha}}\right).
\end{align*}
We estimate 
\begin{align}
\MoveEqLeft\sum_{k=q_{j}}^{r_{j}}\mu\left( \max_{n\in I_{j}}\left|\mathsf{T}_{n}^{\rho_k}\chi-\int\mathsf{T}_{n}^{\rho_k}\chi\mathrm{d}\mu\right|\geq\epsilon\int\mathsf{T}_{2^{j+1}-1}^{\rho_k}\chi\mathrm{d}\mu\right)\notag\\
 & \leq K_{\epsilon}\cdot\sum_{k=q_{j}}^{r_{j}}\exp\left(-\epsilon\cdot \frac{\alpha}{1-\alpha}\cdot\frac{2^j\cdot L\left(\left(1+\epsilon\right)^{k}\right)}{\left(1+\epsilon\right)^{\alpha\cdot k}}\right)\notag\\
 & =K_{\epsilon}\cdot\exp\left(-\epsilon\cdot \frac{\alpha}{1-\alpha}\cdot\frac{2^j\cdot L\left(\left(1+\epsilon\right)^{r_{j}}\right)}{\left(1+\epsilon\right)^{\alpha\cdot r_{j}}}\right)\notag\\
 &\qquad\cdot\sum_{k=q_{j}}^{r_{j}}\exp\left(-\epsilon\cdot\frac{\alpha}{1-\alpha}\cdot2^j\left(\frac{L\left(\left(1+\epsilon\right)^{k}\right)}{\left(1+\epsilon\right)^{\alpha\cdot k}}-\frac{L\left(\left(1+\epsilon\right)^{r_{j}}\right)}{\left(1+\epsilon\right)^{\alpha\cdot r_{j}}}\right)\right),\label{reg var 1}
\end{align}
for $j$ sufficiently large.
We obtain by factoring out that 
\begin{align}
\frac{L\left(\left(1+\epsilon\right)^{k}\right)}{\left(1+\epsilon\right)^{\alpha\cdot k}}-\frac{L\left(\left(1+\epsilon\right)^{r_{j}}\right)}{\left(1+\epsilon\right)^{\alpha\cdot r_{j}}}
 =\frac{L\left(\left(1+\epsilon\right)^{r_{j}}\right)}{\left(1+\epsilon\right)^{\alpha\cdot r_{j}}}\cdot\left(\left(1+\epsilon\right)^{\alpha\cdot\left(r_{j}-k\right)}\frac{L\left(\left(1+\epsilon\right)^{k}\right)}{L\left(\left(1+\epsilon\right)^{r_{j}}\right)}-1\right).\label{summand}
\end{align}

For $k$ sufficiently large we have 
${L((1+\epsilon)^{k})}/{L((1+\epsilon )^{k+1} )}>{ (1+\epsilon )^{-\alpha/2}}$
 and hence we have for $k>q_{j}$ and $j$ sufficiently large 
${L ( (1+\epsilon )^{k} )}/{L((1+\epsilon)^{r_{j}})}> { (1+\epsilon )^{-\alpha\cdot (r_{j}-k )/2}}.$
Applying this to \eqref{summand} and defining $\epsilon_1\coloneqq(1+\epsilon)^{\alpha/2}-1$
yields 
\begin{align*}
\frac{L((1+\epsilon)^{k})}{\left(1+\epsilon\right)^{\alpha\cdot k}}-\frac{L\left(\left(1+\epsilon\right)^{r_{j}}\right)}{\left(1+\epsilon\right)^{\alpha\cdot r_{j}}}> & \frac{L\left(\left(1+\epsilon\right)^{r_{j}}\right)}{\left(1+\epsilon\right)^{\alpha\cdot r_{j}}}\left(\left(1+\epsilon_1\right)^{r_{j}-k}-1\right)
 <\frac{L\left(\left(1+\epsilon\right)^{r_{j}}\right)}{\left(1+\epsilon\right)^{\alpha\cdot r_{j}}}\cdot \epsilon_1\cdot\left(r_{j}-k\right),
\end{align*}
for $k>q_{j}$ and $j$ sufficiently large. Hence, we can understand
the second factor of \eqref{reg var 1} as a geometric series and
obtain 
\begin{align}
\MoveEqLeft\sum_{k=q_{j}}^{r_{j}}\exp\left(-\epsilon\cdot \frac{\alpha}{1-\alpha}\cdot2^j\left(\frac{L\left(\left(1+\epsilon\right)^{k}\right)}{\left(1+\epsilon\right)^{\alpha\cdot k}}-\frac{L\left(\left(1+\epsilon\right)^{r_{j}}\right)}{\left(1+\epsilon\right)^{\alpha\cdot r_{j}}}\right)\right)\notag\\
 & <\sum_{k=q_{j}}^{r_{j}}\exp\left(-\epsilon\cdot\frac{\alpha}{1-\alpha}\cdot2^j\cdot\frac{L\left(\left(1+\epsilon\right)^{r_{j}}\right)}{\left(1+\epsilon\right)^{\alpha\cdot r_{j}}}\cdot\epsilon_1\cdot\left(r_{j}-k\right)\right)\notag\\
 & <\frac{1}{1-\exp\left(-\epsilon\cdot\frac{\alpha}{1-\alpha}\cdot2^j\cdot\frac{L\left(\left(1+\epsilon\right)^{r_{j}}\right)}{\left(1+\epsilon\right)^{\alpha\cdot r_{j}}}\cdot\epsilon_1\right)}.\label{qj rj}
\end{align}
Furthermore, we have that $\left(1+\epsilon\right)^{r_{j}-1}<\max_{n\in I_{j}}g_{n}\leq\left(1+\epsilon\right)^{r_{j}}$.
Since $F$ with $F\left(x\right)=1-L\left(x\right)/x^{\alpha}$ is
a distribution function, $L\left(x\right)/x^{\alpha}$ is decreasing.
Thus, the slow variation of $L$ implies
\begin{align}
\frac{\left(1+\epsilon\right)^{\alpha\cdot r_j}}{L\left(\left(1+\epsilon\right)^{r_j}\right)}
\leq 2\cdot \left(1+\epsilon\right)^{\alpha}\cdot \frac{\left(1+\epsilon\right)^{\alpha\cdot\left(r_j-1\right)}}{L\left(\left(1+\epsilon\right)^{r_j-1}\right)}
&\leq2\cdot \left(1+\epsilon\right)^{\alpha}\cdot \frac{\left(\max_{n\in I_j}g_n\right)^{\alpha}}{L\left(\max_{n\in I_j}g_n\right)},\label{eq: 1+eps/ L}
\end{align}
for $j$ sufficiently large. 
To continue we state the following lemma, which is \cite[Lemma 5]{kessebohmer_strong_2016}.
\begin{lemma}\label{log gamma log tilde gamma}
Let $a,b>1$ and $\psi\in\Psi$. Then there exists $\omega\in\Psi$ such that 
\begin{align*}
\omega\left(\left\lfloor \log_b n\right\rfloor \right)\leq \psi\left(\left\lfloor \log_a n\right\rfloor\right).
\end{align*}
\end{lemma}
From the above observations and Lemma \ref{log gamma log tilde gamma} it follows that there exists $\tilde{\kappa}>0$ and $\tilde{\psi}\in\Psi$ such that 
\begin{align*}
\frac{m}{\log\psi\left(\left\lfloor\log m\right\rfloor\right)}\leq \tilde{\kappa}\cdot\frac{2^j}{\log\tilde{\psi}\left(j\right)},
\end{align*}
for all $m\in I_j$. 
Thus, the existence of $\psi\in\Psi$ such that \eqref{cond} holds implies together with \eqref{eq: 1+eps/ L}
\begin{align*}
\frac{\left(1+\epsilon\right)^{\alpha\cdot r_{j}}}{L\left(\left(1+\epsilon\right)^{r_{j}}\right)}=o\left(\frac{2^j}{\log\widetilde{\psi}\left(j\right)}\right).
\end{align*}
In particular we have 
\begin{align*}
\lim_{j\rightarrow\infty}\epsilon\cdot\frac{\alpha}{1-\alpha}\cdot 2^{j}\cdot\frac{L\left(\left(1+\epsilon\right)^{r_{j}}\right)}{\left(1+\epsilon\right)^{\alpha\cdot r_{j}}}\cdot\epsilon_1=\infty
\end{align*}
and thus we have from \eqref{qj rj} that 
\begin{align}
\MoveEqLeft\sum_{k=q_{j}}^{r_{j}}\exp\left(-\epsilon\cdot\frac{\alpha}{1-\alpha}\cdot2^{j}\left(\frac{L\left(\left(1+\epsilon\right)^{k}\right)}{\left(1+\epsilon\right)^{\alpha\cdot k}}-\frac{L\left(\left(1+\epsilon\right)^{r_{j}}\right)}{\left(1+\epsilon\right)^{\alpha\cdot r_{j}}}\right)\right)\notag\\
& <\frac{1}{1-\exp\left(-\epsilon\cdot\frac{\alpha}{1-\alpha}\cdot2^{j}\cdot\frac{L\left(\left(1+\epsilon\right)^{r_{j}}\right)}{\left(1+\epsilon\right)^{\alpha\cdot r_{j}}}\cdot\epsilon_1\right)}<2,\label{reg var 2}
\end{align}
for $j$ sufficiently large.

On the other hand we have that 
\begin{align*}
\frac{L\left(\left(1+\epsilon\right)^{r_{j}}\right)}{\left(1+\epsilon\right)^{\alpha\cdot r_{j}}} & >\frac{L\left(\left(1+\epsilon\right)\max_{n\in I_{j}}g_{n}\right)}{\left(\left(1+\epsilon\right)\max_{n\in I_{j}}g_{n}\right)^{\alpha}}>\frac{1}{2}\cdot\frac{L\left(\max_{n\in I_{j}}g_{n}\right)}{\left(\max_{n\in I_{j}}g_{n}\right)^{\alpha}},
\end{align*}
for $j$ sufficiently large. Hence, we have for the first factor of \eqref{reg var 1} that 
\begin{align}
\exp\left(-\epsilon\cdot\frac{\alpha}{1-\alpha}\cdot\frac{2^{j}L\left(\left(1+\epsilon\right)^{r_{j}}\right)}{\left(1+\epsilon\right)^{\alpha\cdot r_{j}}}\right) & <\exp\left(-\frac{\epsilon\cdot\alpha}{2\left(1-\alpha\right)}\cdot\frac{2^{j}\cdot L\left(\max_{n\in I_{j}}g_{n}\right)}{\max_{n\in I_{j}}g_{n}^{\alpha}}\right),\label{reg var 3}
\end{align}
for $r_{j}$ sufficiently large.
Inserting \eqref{reg var 2} and \eqref{reg var 3} in \eqref{reg var 1}
yields 
\begin{align*}
\MoveEqLeft\sum_{k=q_{j}}^{r_{j}}\mu\left( \max_{n\in I_{j}}\left|\mathsf{T}_{n}^{\rho_k}\chi-\int\mathsf{T}_{n}^{\rho_k}\chi\mathrm{d}\mu\right|\geq\epsilon\int\mathsf{T}_{2^{j+1}-1}^{\rho_k}\chi\mathrm{d}\mu\right)\\
 & <2K\exp\left(-\frac{\epsilon\cdot\alpha}{2\left(1-\alpha\right)}\cdot\frac{2^{j}\cdot L\left(\max_{n\in I_{j}}g_{n}\right)}{\max_{n\in I_{j}}g_{n}^{\alpha}}\right),
\end{align*}
for $j$ and thus $q_{j}$ sufficiently large.
Next we show that the above expression is summable in $j$.
\eqref{cond} implies that there exists $\psi\in\Psi$ such that
\begin{align*}
\frac{\epsilon\cdot \alpha}{2\left(1-\alpha\right)}\cdot n\cdot \frac{ L\left(g_n\right)}{g_n^{\alpha}}\geq\log\psi\left(\left\lfloor \log n\right\rfloor \right),
\end{align*}
for all $n\in\N$.
Applying Lemma \ref{log gamma log tilde gamma} with $a=\mathrm{e}$ and $b=2$ 
yields that this implies the existence of $\omega\in\Psi$ such that
\begin{align*}
\frac{\epsilon\cdot\alpha}{2\left(1-\alpha\right)}\cdot n\cdot \frac{ L\left(g_n\right)}{g_n^{\alpha}}
\geq\log\omega\left(\left\lfloor \log_{2} n\right\rfloor \right).
\end{align*}
This implies
\begin{align*}
 \exp\left(-\frac{\epsilon\cdot\alpha}{2\left(1-\alpha\right)}\cdot n\cdot \frac{ L\left(g_n\right)}{g_n^{\alpha}}\right)
 &\leq \frac{1}{\omega\left(\left\lfloor \log_{2} n\right\rfloor \right)}
\end{align*}
and thus,
\begin{align*}
 \exp\left(-\frac{\epsilon\cdot\alpha}{2\left(1-\alpha\right)}\cdot\frac{2^{j}\cdot L\left(\max_{n\in I_{j}}g_{n}\right)}{\max_{n\in I_{j}}g_{n}^{\alpha}}\right)
 &\leq \frac{1}{\omega\left(j\right)}.
\end{align*}
Since $\omega\in\Psi$,
this is equivalent to
\begin{align*}
\sum_{j=1}^{\infty}\exp\left(-\frac{\epsilon\cdot \alpha}{2\left(1-\alpha\right)}\cdot\frac{2^{j}\cdot L\left(\max_{n\in I_{j}}g_{n}\right)}{\max_{n\in I_{j}}g_{n}^{\alpha}}\right)<\infty.
\end{align*}
Thus,
\begin{align*}
\sum_{j=1}^{\infty}\sum_{k=q_{j}}^{r_{j}}\mu\left(\max_{n\in I_{j}}\left|\mathsf{T}_{n}^{\rho_{k}}\chi-\int\mathsf{T}_{n}^{\rho_{k}}\chi\mathrm{d}\mu\right|\geq\frac{\epsilon}{2}\int\mathsf{T}_{2^{j+1}-1}^{\rho_{k}}\chi\mathrm{d}\mu\right)<\infty,
\end{align*}
i.e.\ the statement of \ref{en: trunc B}.

\emph{Proof of \ref{en: trunc C}}.
We note that \eqref{E X*} implies that there exists $R\in\mathbb{\mathbb{\mathbb{\mathbb{R}}}}_{>0}$
such that 
\begin{align*}
\frac{r}{\int \upperleft{r}{}{\chi}\mathrm{d}\mu}\leq2\cdot\frac{1-\alpha}{\alpha}\cdot\frac{r^{\alpha}}{L\left(r\right)},
\end{align*}
for all $r\geq R$. With 
\begin{align*}
D\coloneqq\begin{cases}
           \sup_{r\in\left[\inf_{n\in\mathbb{N}}\overline{g}_{n},R\right]}\frac{r}{\int \upperleft{r}{}{\chi}\mathrm{d}\mu}&\text{if }\inf_{n\in\mathbb{N}}\overline{g}_{n}<R\\
           0&\text{else}
          \end{cases}
\end{align*}
we have that 
\begin{align*}
\frac{\bar{g}_{n}}{\int\upperleft{\bar{g}_{n}}{}{\chi}\mathrm{d}\mu}\leq\max\left\{ D,2\cdot\frac{1-\alpha}{\alpha}\cdot\frac{\bar{g}_{n}^{\alpha}}{L\left(\bar{g}_{n}\right)}\right\} 
\end{align*}
and since we have for $\psi\left(n\right)=n^{2}$ that 
\begin{align*}
\frac{\bar{g}_{n}^{\alpha}}{L\left(\bar{g}_{n}\right)}=o\left(\frac{n}{\log\psi\left(n\right)}\right),
\end{align*}
by Theorem \ref{Thm: Sn* allg} we find 
that the statement of \ref{en: trunc C} holds.
\end{proof}

\subsection{Proof of the large deviation result}\label{proofs transfer C}
In this section we will prove Lemma \ref{bernoulli}.
The proof is similar to the proof of \cite[Lemma 11]{kessebohmer_strong_2016}.
However, due to dependence some additional argumentations are required. 
In order to increase the readability, we will formulate this proof to be self contained.
\begin{proof}[Proof of Lemma \ref{bernoulli}]
Let
$I_n\coloneqq\left[2^n,2^{n+1}-1\right]$ and 
\begin{align*}
\kappa_n\coloneqq\left\lfloor\min_{k\in I_n}\log\psi\left(\left\lfloor \log k\right\rfloor\right)\right\rfloor,
\end{align*}
for $\psi\in\Psi$
and let $\nu:\mathbb{N}\to\mathbb{N}$ be defined as 
$\nu\left(n\right)\coloneqq\left\lfloor\log_2 n\right\rfloor$.
Further let 
\begin{align*}
B^{\ell}_i\coloneqq \begin{cases}
                    \mathbbm{1}_{\left\{\chi>\ell\right\}} &\text{ if }\mu\left(\chi>\ell\right)\geq\kappa_{\nu\left(i\right)}/2^{\nu\left(i\right)}\\
                    0&\text{ otherwise}
                   \end{cases}
\,\,\text{ and }\,\,
C^{\ell}_i\coloneqq \begin{cases}
                    \mathbbm{1}_{\left\{\chi>\ell\right\}} &\text{ if }\mu\left(\chi>\ell\right)\leq\kappa_{\nu\left(i\right)}/2^{\nu\left(i\right)}\\
                    0&\text{ otherwise}
                   \end{cases}
\end{align*}
and
\begin{align*}
\Delta_{n,l}&\coloneqq\left\{x\colon\mu\left(\chi\geq x\right)\in\left[\frac{l}{2^{n+1}},\frac{l+1}{2^{n+1}}\right]\right\}.
\end{align*}
Furthermore, define for $l,n\in\mathbb{N}$ fulfilling $\Delta_{n,l}\neq\emptyset$
the truncated observables 
\begin{align*}
f_{n,l}&\coloneqq\mathbbm{1}_{\left\{\chi\geq F^{\leftarrow}\left(1-l/2^{n+1}\right)\right\}}\,\,\,\text{ and }\,\,\,
g_{n,l}\coloneqq\mathbbm{1}_{\left\{\chi> F^{\leftarrow}\left(1-(l+1)/2^{n+1}\right)\right\}}.
\end{align*}

The definition of those observables ensures
\begin{align*}
 \frac{l}{2^{n+1}}\leq \int f_{n,l}\mathrm{d}\mu\leq \int g_{n,l}\mathrm{d}\mu\leq \frac{l+1}{2^{n+1}}.
\end{align*}

We will show \eqref{pn B} by separately showing the following:
\begin{enumerate}[label=(\Alph*)]
\item\label{en: large dev a} 
For all $i\in I_n$ and $V>0$ we have that
\begin{align*}
 \left\{p_i\cdot i-\mathsf{S}_i\mathbbm{1}_{\left\{\chi>u_i\right\}}> V\cdot c\left(p_i\cdot i, i\right)\right\}
 &\subset \left\{\max_{j\in I_{n}}\left|\mathsf{S}_j f_{n,l}-\int f_{n,l}\mathrm{d}\mu\cdot j\right|>\frac{V}{4}\cdot \min_{r\in I_{n}} c\left(l, r\right)\right\}\notag
\end{align*}
and 
\begin{align*}
\left\{\mathsf{S}_i\mathbbm{1}_{\left\{\chi>u_i\right\}}-p_i\cdot i> V\cdot c\left(p_i\cdot i, i\right)\right\}
 &\subset \left\{\max_{j\in I_{n}}\left|\mathsf{S}_{j}g_{n,l}-\int g_{n,l}\mathrm{d}\mu\cdot j\right|>\frac{V}{4}\cdot  \min_{r\in I_{n}}c\left(l, r\right)\right\}.
\end{align*}

\item\label{en: large dev b} 
There exists $V>0$ such that
\begin{align*}
 \sum_{n=k}^{\infty}\sum_{l=\kappa_n}^{2^{n+1}}\mu\left(\left|\max_{j\in I_n}\mathsf{S}_j f_{n,l}-\int f_{n,l}\mathrm{d}\mu\cdot j\right|>\frac{V}{4}\cdot\min_{r\in I_n} c\left(l, r\right)\right)<\infty
\end{align*}
and 
\begin{align}
 \sum_{n=k}^{\infty}\sum_{l=\kappa_n}^{2^{n+1}}\mu\left(\left|\max_{j\in I_n}\mathsf{S}_j g_{n,l}-\int g_{n,l}\mathrm{d}\mu\cdot j\right|>\frac{V}{4}\cdot\min_{r\in I_n} c\left(l, r\right)\right)<\infty.\label{l c tilde}
\end{align}

\item\label{en: large dev c}
\begin{align*}
\mu\left( \left|\mu\left(\chi>u_i\right)\cdot i-\mathsf{S}_iC^{u_i}_{i}\right|\geq V\cdot c\left(p_{i}\cdot i,i\right)\text{ i.o.} \right)=0.
\end{align*}
\end{enumerate}
Applying the Borel-Cantelli lemma on \ref{en: large dev b} and combining this with \ref{en: large dev a} yields 
\begin{align*}
\mu\left( \left|\mu\left(\chi>u_i\right)\cdot i-\mathsf{S}_iB^{u_i}_{i}\right|\geq V\cdot c\left(p_{i}\cdot i,i\right)\text{ i.o.} \right)=0.
\end{align*}
This together with \ref{en: large dev c} yields the statement of the lemma. 

\emph{Proof of \ref{en: large dev a}}.
Let us assume $i\in I_n$ and
\begin{align*}
\frac{l}{2^{n+1}}\leq p_i\leq\frac{l+1}{2^{n+1}}.
\end{align*}
This implies by the construction of $f_{n,l}$ and $g_{n,l}$ that
\begin{align}
\frac{l}{2^{n+1}}\leq \int f_{n,l}\mathrm{d}\mu \leq p_i\leq \int g_{n,l}\mathrm{d}\mu\leq \frac{l+1}{2^{n+1}}.\label{l pj}
\end{align}
We can conclude from \eqref{l pj} that 
$ \left(p_{i}-\int f_{n,l}\mathrm{d}\mu\right)\cdot i\leq{i}/{2^{n+1}}\leq 1$
and furthermore $l\leq p_i\cdot 2^{n+1}\leq 2\cdot p_i\cdot i$.
Thus,
\begin{align}
 \left\{p_i\cdot i-\mathsf{S}_i\mathbbm{1}_{\left\{\chi>u_i\right\}}> V\cdot c\left(p_i\cdot i, i\right)\right\}
 &\subset \left\{p_i\cdot i-\mathsf{S}_i f_{n,l}> V\cdot c\left(p_i\cdot i, i\right)\right\}\notag\\
 &\subset \left\{\int f_{n,l}\mathrm{d}\mu\cdot i-\mathsf{S}_i f_{n,l}>\frac{V}{2}\cdot  c\left(l, i\right)-1\right\}\notag\\
 &\subset \left\{\max_{j\in I_{n}}\left|\int f_{n,l}\mathrm{d}\mu\cdot i-\mathsf{S}_i f_{n,l}\right|>\frac{V}{4}\cdot \min_{r\in I_{n}} c\left(l, r\right)\right\},\label{eq: subset lar dev 1}
\end{align}
if $V\geq 2$.
On the other hand by \eqref{l pj} we can conclude that 
$\left(\int g_{n,l}\mathrm{d}\mu-p_{i}\right)\cdot i\leq i/2^{n+1}\leq 1$
and $l\leq p_i\cdot 2^{n+1}\leq 2\cdot p_i\cdot i$.
Hence, 
\begin{align}
 \left\{\mathsf{S}_i\mathbbm{1}_{\left\{\chi>u_i\right\}}-p_i\cdot i> V\cdot c\left(p_i\cdot i, i\right)\right\}
 &\subset \left\{\mathsf{S}_ig_{n,l}-p_i\cdot i> V\cdot c\left(p_i\cdot i, i\right)\right\}\notag\\
 &\subset \left\{\mathsf{S}_ig_{n,l}-\int g_{n,l}\mathrm{d}\mu\cdot i>\frac{V}{2}\cdot  c\left(l, i\right)-1\right\}\notag\\
 &\subset \left\{\max_{j\in I_{n}}\left|\mathsf{S}_ig_{n,l}-\int g_{n,l}\mathrm{d}\mu\cdot j\right|>\frac{V}{4}\cdot  \min_{r\in I_{n}}c\left(l, r\right)\right\},\label{eq: subset lar dev 2}
\end{align}
if $V\geq 2$.
Combining 
\eqref{eq: subset lar dev 1}
and \eqref{eq: subset lar dev 2} gives the statement of \ref{en: large dev a}.

\emph{Proof of \ref{en: large dev b}}.
To ease notation we define
\begin{align*}
\overline{f}_{n,l}&\coloneqq f_{n,l}-\int f_{n,l}\mathrm{d}\mu
\,\,\,\text{ and }\,\,\,
\overline{g}_{n,l}\coloneqq g_{n,l}-\int g_{n,l}\mathrm{d}\mu.
\end{align*}
We aim to apply Lemma \ref{lemma: Tnfn chi deviation allg} and note that by \eqref{l pj}
\begin{align}
 \frac{V}{4}\cdot \frac{\min_{r\in I_n} c\left(l, r\right)}{\left(2^{n+1}-1\right)\cdot \left|\overline{f}_{n,l}\right|_1}
 \geq \frac{V}{4}\cdot \frac{\left(l+1\right)^{1/2+\epsilon}\cdot \kappa_n^{1/2-\epsilon}}{l}
 \geq \min\left\{\frac{V}{8}, 1\right\}\cdot \left(\frac{\kappa_n}{l}\right)^{1/2-\epsilon}
 \leq  1,\label{eq: V/2 min}
\end{align}
if $V>1$.
Using Lemma \ref{lemma: Tnfn chi deviation allg} this implies 
\begin{align}
\MoveEqLeft\mu\left(\left|\max_{j\in I_n}\mathsf{S}_j\overline{f}_{n,l}\right|>\frac{V}{4}\cdot \min_{r\in I_n} c\left(l, r\right)\right)\notag\\
&\leq K\exp\left(-U\cdot \frac{V \min_{r\in I_n} c\left(l, r\right)}{2\cdot l\cdot \left\|\overline{f}_{n,l}\right\|}\cdot \min\left\{\frac{V}{8}, 1\right\}\cdot \left(\frac{\kappa_n}{l}\right)^{1/2+\epsilon}\right)\notag\\
&\leq K\exp\left(-\frac{U\cdot V}{2}\cdot\min\left\{\frac{V}{8}, 1\right\}\cdot  \frac{l^{2\epsilon}\cdot \kappa_n^{1-2\epsilon}}{\left\|\overline{f}_{n,l}\right\|}\right),\label{max Enk}
\end{align}
for $n$ sufficiently large.
Furthermore, we have by \eqref{C 2}
\begin{align*}
 \left\|\overline{f}_{n,l}\right\|\leq \left\|f_{n,l}\right\|+\left\|\int f_{n,l}\mathrm{d}\mu\right\|=\left\|f_{n,l}\right\|+\int f_{n,l}\mathrm{d}\mu\left\|\mathbbm{1}\right\|
 \leq K_2+ \left\|\mathbbm{1}\right\|.
\end{align*}
Hence,
\begin{align}
\mu\left(\left|\max_{j\in I_n}\mathsf{S}_j\overline{f}_{n,l}\right|>\frac{V}{4}\cdot\min_{r\in I_n} c\left(l, r\right)\right)
&\leq K\exp\left(-\frac{U\cdot V}{2\cdot \left(K_2+ \left\|\mathbbm{1}\right\|\right)}\cdot\min\left\{\frac{V}{8}, 1\right\}\cdot l^{2\epsilon}\cdot \kappa_n^{1-2\epsilon}\right).\label{eq: sum Enlk}
\end{align}
If we set 
\begin{align*}
 V=\max\left\{\frac{2\cdot \left(K_2+ \left\|\mathbbm{1}\right\|\right)}{U}, 8\right\},
\end{align*}
then 
\begin{align*}
\frac{U\cdot V}{2\cdot \left(K_2+ \left\|\mathbbm{1}\right\|\right)}\cdot\min\left\{\frac{V}{8}, 1\right\}\geq 1.
\end{align*}
We estimate
\begin{align}
 \sum_{n=k}^{\infty}\sum_{l=\kappa_n}^{2^{n+1}}\mu\left(\left|\max_{j\in I_n}\mathsf{S}_j\overline{f}_{n,l}\right|>\frac{V}{4}\cdot\min_{r\in I_n} c\left(l, r\right)\right)
 &\leq \sum_{n=k}^{\infty}\sum_{l=\kappa_n}^{2^{n+1}} 2\exp\left(-\max\left\{l,\kappa_n\right\}^{2\epsilon}\cdot\kappa_n^{1-2\epsilon}\right).\label{eq: sum in Phink}
\end{align}
This estimation holds for $k$ sufficiently large.
Furthermore,
\begin{align}
\sum_{l=\kappa_n}^{2^{n+1}}\exp\left(- \kappa_n^{1-2\epsilon}\cdot l^{2\epsilon}\right)
&= \exp\left(-\kappa_n\right)\sum_{l=\kappa_n}^{2^{n+1}}\exp\left(-\kappa_n^{1-2\epsilon}\left(l^{2\epsilon}-\kappa_n^{2\epsilon}\right)\right)\notag\\
&< \exp\left(-\kappa_n\right)\sum_{l=0}^{\infty}\exp\left(-\kappa_n^{1-2\epsilon}l^{2\epsilon}\right)\notag\\
&\leq \exp\left(-\kappa_n\right)\sum_{l=0}^{\infty}\exp\left(-l^{2\epsilon}\right)\eqqcolon \exp\left(-\kappa_n\right)\cdot W,\label{m qn 2^n}
\end{align}
for $n$ and thus $\kappa_n$ sufficiently large.
Further, by construction
\begin{align}
\exp\left(-\kappa_n\right)
&=\exp\left(-\left\lfloor\min_{j\in I_n}\log\psi\left(\left\lfloor\log j\right\rfloor\right)\right\rfloor\right)
\leq \exp\left(-\min_{j\in I_n}\log\psi\left(\left\lfloor\log j\right\rfloor\right)+1\right)\notag\\
&=\frac{\mathrm{e}}{\min_{j\in I_n}\psi\left(\left\lfloor\log j\right\rfloor\right)}.\label{eq: exp(-kappan)}
\end{align}
Furthermore, we can conclude from Lemma \ref{log gamma log tilde gamma} that there exists $\omega\in\Psi$ such that 
\begin{align}
\min_{j\in I_n}\psi\left(\left\lfloor\log j\right\rfloor\right)&\geq\min_{j\in I_n}\omega\left(\left\lfloor\log_2 j\right\rfloor\right)=\omega\left(n\right).\label{eq: exp(-kappan)1}
\end{align}
Hence, combining this with \eqref{eq: sum in Phink}, \eqref{m qn 2^n}, and \eqref{eq: exp(-kappan)} yields
\begin{align*}
 \sum_{n=k}^{\infty}\sum_{l=\kappa_n}^{2^{n+1}}\mu\left(\left|\max_{j\in I_n}\mathsf{S}_j\overline{f}_{n,l}\right|>\frac{V}{4}\cdot\min_{r\in I_n} c\left(l, r\right)\right)<\infty.
\end{align*}
\eqref{l c tilde} follows from analogous calculations.

\emph{Proof of \ref{en: large dev c}}.
For $n\in\mathbb{N}$ define
the random variables
\begin{align*}
h_n&\coloneqq\mathbbm{1}_{\left\{\chi> F^{\leftarrow}\left(1-\kappa_n/2^n\right)\right\}}
\end{align*}
which imply $\int h_n\mathrm{d}\mu \leq \kappa_n/2^n$. 
To ease notation set 
$\overline{h}_{n}\coloneqq h_n-\int h_n\mathrm{d}\mu$. 
We have for $i\in I_n$ and $p_{i}\leq\kappa_{n}/2^{n}$ that $\int h_{n}\mathrm{d}\mu\geq p_i$ and thus
\begin{align*}
 \mathsf{S}_i\mathbbm{1}_{\left\{\chi>u_i\right\}}-p_i\cdot i>V\cdot c\left(p_i\cdot i, i\right)
\end{align*}
implies 
\begin{align*}
 \mathsf{S}_ih_n-p_i\cdot i> V\cdot c\left(\kappa_{n}, i\right)
\end{align*}
which yields
\begin{align*}
 \mathsf{S}_ih_n-\int h_n\mathrm{d}\mu\cdot i
 &> V\cdot c\left(\kappa_{n}, i\right)-\left(\int h_n\mathrm{d}\mu-p_i\right)\cdot i
 >V\cdot c\left(\kappa_{n}, i\right)-\kappa_n\cdot 2\\
 &\geq \frac{V}{2}\cdot c\left(\kappa_{n}, i\right)
 \geq \frac{V}{2}\cdot\min_{r\in I_n} c\left(\kappa_{n}, r\right),
\end{align*}
if $V\geq 4$.
On the other
hand we have for $i\in I_n$ and $p_i\leq \kappa_{n}/2^{n}$ 
that 
\begin{align*}
 \left\{p_i\cdot i-\mathsf{S}_i\mathbbm{1}_{\left\{\chi>u_i\right\}}>V\cdot c\left(p_i\cdot i, i\right)\right\}
\subset \left\{p_i\cdot i>V\cdot p_i\cdot i\right\}=\emptyset,
\end{align*}
if $V\geq 1$.
Combining this with the above considerations yields 
\begin{align*}
 \left\{\left|\mathsf{S}_i\mathbbm{1}_{\left\{\chi>u_i\right\}}-p_i\cdot i\right|>V\cdot c\left(p_i\cdot i, i\right)\right\}
 \subset 
 \left\{\max_{j\in I_n}\left|\mathsf{S}_jh_n-\int h_n\mathrm{d}\mu\cdot j\right|
 > \frac{V}{2}\cdot \min_{r\in I_n}c\left(\kappa_{n}, r\right)\right\}.
\end{align*}

Likewise as in the calculations leading to \eqref{eq: sum Enlk} replacing $l$ by $\kappa_n$ in \eqref{eq: V/2 min} and \eqref{max Enk}, it follows for sufficiently large $n$ that 
\begin{align*}
\MoveEqLeft\mu\left(\max_{j\in I_n}\left|\mathsf{S}_j\overline{h}_{n}\right|>\frac{V}{2}\cdot \min_{r\in I_n}c\left(\kappa_n, r\right)\right)<K\cdot\exp\left(-\kappa_n\right).
\end{align*}
Applying \eqref{eq: exp(-kappan)} and \eqref{eq: exp(-kappan)1}
yields 
\begin{align*}
\MoveEqLeft\sum_{n=1}^{\infty}\mu\left(\max_{i\in I_n}\left|\mathsf{S}_j\overline{h}_{n}\right|>\frac{V}{2}\cdot \min_{r\in I_n}c\left(\kappa_n, r\right)\right)<\infty.
\end{align*}
The Borel-Cantelli lemma implies the statement or \ref{en: large dev c}.
\end{proof}

\subsection{Proofs of theorems concerning the trimmed sum \texorpdfstring{$\mathsf{S}_n^{b_n}\chi$}{}}
In this section we will give the proof of Theorem \ref{find bn}, its Corollary \ref{S b(n) immer} and Theorem \ref{Sb(n)}.
The main idea of all proofs is to use statements which state that under the given assumptions Properties $\bm{A}$ and $\bm{B}$ are fulfilled 
and consequently Lemma \ref{lem: Prop A B to trimming} can be applied.

\begin{proof}[Proof of Theorem \ref{find bn}]
Our strategy is to show that all properties are fulfilled such that we can apply Lemma \ref{lem: Prop A B to trimming}.

If \eqref{cond 1 find bn} holds, then also \eqref{cond a} for the choice $\psi(n)=n^2$ which fulfills $\psi\in\Psi$.
Thus, Theorem \ref{Thm: Sn* allg} implies that Condition $\bm{A}$ holds.

Furthermore, the definition of $b_n$, the 
above choice of $\psi$ combined with Lemma
\ref{bernoulli}, and setting $W=2V$ yield 
that Condition $\bm{B}$ holds.

Finally, by the definition of $b_n$ we have that $\gamma_n<W\cdot \max\{a_n^{1/2+\epsilon}\cdot\left(\log\log n\right)^{1/2-\epsilon},\log\log n\}+1$.
Hence, \eqref{cond 1 find bn} implies that \eqref{eq: fn gamman} holds.
\end{proof}

\begin{proof}[Proof of Corollary \ref{S b(n) immer}]
The proof of Theorem \ref{S b(n) immer} is basically the same as the proof of \cite[Theorem A]{kessebohmer_strong_2016} applying Theorem \ref{find bn} instead of \cite[Theorem B]{kessebohmer_strong_2016}.
\end{proof}

\begin{proof}[Proof of  Theorem \ref{Sb(n)}]
In the first part of the proof we will show \eqref{eq: lim Snbnchi}.
We define 
\begin{align*}
f_{n}\coloneqq F^{\leftarrow}\left(1-\frac{b_{n}-W\cdot c\left(b_{n},n\right)}{n}\right)-1
\end{align*}
with $c$ as in \eqref{c(n)}.
We have that $F(F^{\leftarrow}(x)-1)\leq x$.
Since $c$ is monotonically increasing
in its first variable, it follows that 
\begin{align*}
 \gamma_n&\coloneqq b_n-n\cdot \mu\left(\chi> f_n\right)
 =b_n-n\cdot \left(1-F\left(f_n\right)\right)\\
 &=b_n-n\cdot \left(1-F\left(F^{\leftarrow}\left(1-\frac{b_{n}-W\cdot c\left(b_{n},n\right)}{n}\right)-1\right)\right)\\
 &\geq b_n-n\cdot \frac{b_{n}-W\cdot c\left(b_{n},n\right)}{n}
 = W\cdot c\left(b_{n},n\right)\geq W\cdot c\left(n\cdot \mu\left(\chi> f_n\right),n\right).
\end{align*}
Hence, Lemma \ref{bernoulli} implies 
that the pair $((f_n),(\gamma_n))$ fulfills Property $\bm{B}$.

In the following we will show that this choice of $(f_n)$ fulfills Property $\bm{A}$.
Assume that $\left(b_{n}\right)$ fulfills \eqref{eq: t cond a2} for some $\psi\in\Psi$.
If we consider 
$c_{\epsilon,\psi}\left(k,n\right)$ 
in \eqref{c(n)} with the same $\psi$, then $c\left(b_{n},n\right)=b_{n}^{1/2+\epsilon}\log\psi\left(\left\lfloor \log n\right\rfloor \right)^{1/2-\epsilon}$
and thus 
$W\cdot c\left(b_{n},n\right)=o\left(b_{n}\right)$.

In the next steps we aim to prove that $(1-F\left(f_n\right))\cdot n\sim b_{n}$.
We have on the one hand that $F(F^{\leftarrow}(x))\geq x$ 
and on the other hand $F(F^{\leftarrow}(x)-1)\leq x$.
This yields
\begin{align}
 b_n&\sim b_n-W\cdot c\left(b_{n},n\right)
 \geq n\cdot \left(1-F\left(F^{\leftarrow}\left(1-\frac{b_{n}-W\cdot c\left(b_{n},n\right)}{n}\right)\right)\right)
 = n\cdot \left(1-F\left(f_n+1\right)\right)\notag\\
 &=n\cdot \frac{L\left(f_n+1\right)}{\left(f_n+1\right)^{\alpha}}
 \sim n\cdot \frac{L\left(f_n\right)}{f_n^{\alpha}}
 =n\cdot \left(1-F\left(f_n\right)\right)\notag\\
 &= n\cdot \left(1-F\left(F^{\leftarrow}\left(1-\frac{b_{n}-W\cdot c\left(b_{n},n\right)}{n}-1\right)\right)\right)
 \geq b_n.\label{eq: bn asymp}
\end{align}
The second asymptotic holds because by assumption $b_n=o(n)$ which implies that $(f_n)$ tends to infinity.

The above observation combined with \eqref{eq: t cond a2} implies
\begin{align*}
 \lim_{n\to\infty}\frac{n\cdot L\left(f_n\right)}{f_n^{\alpha}\cdot \log\psi\left(\left\lfloor\log n\right\rfloor\right)}=\infty
\end{align*}
which is equivalent to \eqref{cond} and Theorem \ref{Thm: Sn* reg var} states that under this condition Property $\bm{A}$ holds.

In the last steps we will prove \eqref{eq: fn gamman}.
By \eqref{E S*} we have that
\begin{align*}
 \frac{\gamma_n\cdot f_n}{\int\mathsf{T}_n^{f_n}\chi\mathrm{d}\mu}
 \sim \frac{\gamma_n\cdot f_n}{n\cdot \frac{\alpha}{1-\alpha}\cdot f_n^{1-\alpha}\cdot L\left(f_n\right)}
 =\frac{\gamma_n\cdot f_n^{\alpha}}{n\cdot \frac{\alpha}{1-\alpha}\cdot L\left(f_n\right)}
 =\frac{\gamma_n}{n\cdot\left(1-F\left(f_n\right)\right)}.
\end{align*}

Since $b_n\sim n\cdot (1-F(f_n))$, it follows that $\gamma_n=o\left(n\cdot \left(1-F\left(f_n\right)\right)\right)$.
This implies \eqref{eq: fn gamman}.
Hence, we can apply Lemma \ref{lem: Prop A B to trimming} and obtain the first part of the theorem. 
\medskip

In the next steps we will show the asymptotic given in \eqref{bn psi exp}
by finding an asymptotic equivalent sequence in
terms of $\left(b_{n}\right)$ for $\int\mathsf{T}_{n}^{f_{n}}\chi\mathrm{d}\mu$.
We have by \eqref{E S*} and the definition
of $F$ and $(f_{n})$ that 
\begin{align}
\int\mathsf{T}_{n}^{f_{n}}\chi\mathrm{d}\mu & \sim\frac{\alpha}{1-\alpha}\cdot n\cdot f_{n}^{1-\alpha}\cdot L\left(f_{n}\right)
  =\frac{\alpha}{1-\alpha}\cdot n\cdot \left(1-F\left(f_{n}\right)\right)\cdot f_{n}\notag\\
 & =\frac{\alpha}{1-\alpha}\cdot n\cdot \left(1-F\left(F^{\leftarrow}\left(1-\frac{b_{n}-W\cdot c\left(b_{n},n\right)}{n}\right)\right)\right)\notag\\
 &\qquad\cdot F^{\leftarrow}\left(1-\frac{b_{n}-W\cdot c\left(b_{n},n\right)}{n}\right).\label{eq: an G}
\end{align}
Set for the following $u_n\coloneqq n/\left(b_{n}-W\cdot c\left(b_{n},n\right)\right)$
and the function $G:\mathbb{R}^{+}\rightarrow\mathbb{R}^{+}$ with 
$G\left(x\right)\coloneqq1/\left(1-F\left(x\right)\right)=x^{\alpha}/L\left(x\right)$.
Then 
\begin{align*}
 F^{\leftarrow}\left(1-1/u_n\right)
 &=\inf\left\{ y\in\left[0,\infty\right)\colon F\left(y\right)>1-\frac{1}{u_n}\right\}\\
 &=\inf\left\{ y\in\left[0,\infty\right)\colon1-\frac{1}{G\left(y\right)}>1-\frac{1}{u_n}\right\}\\
 &=\inf\left\{ y\in\left[0,\infty\right)\colon G\left(y\right)>u_n\right\}
 =G^{\leftarrow}\left(u_n\right).
\end{align*}
Since $u_{n}$ tends to infinity, we can apply Lemma \ref{bingham} on $G$
and obtain 
\begin{align*}
F^{\leftarrow}\left(1-1/u_n\right)\sim u_{n}^{1/\alpha}\cdot\left(L^{1/\alpha}\right)^{\#}\left(u_{n}^{1/\alpha}\right).
\end{align*}

This and the fact that $b_{n}\sim b_{n}-W\cdot c\left(b_{n},n\right)$
yields
\begin{align}
\MoveEqLeft L^{\leftarrow}\left(1-\frac{b_{n}-W\cdot c\left(b_{n},n\right)}{n}\right)\notag\\
 & \sim\left(\frac{n}{b_{n}-W\cdot c\left(b_{n},n\right)}\right)^{1/\alpha}\cdot \left(L^{1/\alpha}\right)^{\#}\left(\left(\frac{n}{b_{n}-W\cdot c\left(b_{n},n\right)}\right)^{1/\alpha}\right)\notag\\
 & \sim\left(\frac{n}{b_{n}}\right)^{1/\alpha}\cdot\left(L^{1/\alpha}\right)^{\#}\left(\left(\frac{n}{b_{n}}\right)^{1/\alpha}\right).\label{G 2}
\end{align}
Hence, applying \eqref{eq: bn asymp} and \eqref{G 2} to \eqref{eq: an G} yields \eqref{bn psi exp}.
\end{proof}

\subsection{Proof of main example}\label{subsec: proof ex}
\begin{proof}[Proof of Propositions \ref{prop: Ex1} and \ref{prop: Ex2}]
We have for all $f,g\in BV$ that 
\begin{align*}
\left\|f\cdot g\right\|_{BV}&=\left|f\cdot g\right|_{\infty}+\mathsf{V}\left(f\cdot g\right)
\leq \left|f\right|_{\infty}\cdot \left|g\right|_{\infty}+\left|f\right|_{\infty}\cdot \mathsf{V}\left(g\right)+\left|g\right|_{\infty}\cdot\mathsf{V}\left(f\right)
\leq \left\|f\right\|_{BV}\cdot\left\|g\right\|_{BV}
\end{align*}
and thus, it can be deduced that $BV$ is a Banach algebra which contains the constant functions.

Further define $h:\Omega\to\mathbb{R}_{\geq 0}$ as a function fulfilling $\left|h\right|_{\infty}\leq 1$, $\mathsf{V}\left(h\right)<\infty$, and $h\lvert_{\Omega\backslash \Omega'}=0$. 
With this function we define an operator $\mathcal{P}:\mathcal{L}^1\to \mathcal{L}^1$ by 
\begin{align*}
 \mathcal{P}f\left(x\right)\coloneqq \sum_{y\in T^{-1}\left(x\right)}h\left(y\right)f\left(y\right).
\end{align*}
Furthermore, let $h$ be such that the adjoint operator $\mathcal{P}^*$ preserves the Lebesgue measure $\lambda$, i.e.
$\mathcal{P}^*\left(\lambda\right)\left(f\right)=\lambda \left(\mathcal{P} f\right)=\lambda \left( f\right)$ for all $f\in\mathcal{L}^1$.

Now, given all the properties of $T$, \cite[Theorem 1]{rychlik_bounded_1983} states that the operator $\mathcal{P}$
fulfills all the properties for a spectral gap given in Definition \ref{def spec gap} except that there might be finitely many mutually orthogonal one dimensional projections.
Combining this with \cite[Remark 4b]{rychlik_bounded_1983} gives that in case the system is topologically mixing there is only one projection corresponding to the eigenvalue $1$.
Furthermore, \cite[Theorem 3]{rychlik_bounded_1983} implies in the topologically mixing case the existence of a function $f$ such that $\mathcal{P}f=f$ 
and thus the existence of a measure $\mu$ absolutely continuous with respect to $\lambda$.

Finally, \eqref{cond C 1} and \eqref{cond C 2} ensure that \eqref{C 1} and \eqref{C 2} are fulfilled
with $K_1\coloneqq 1+\widetilde{K}_1$ and $K_2\coloneqq 1+\widetilde{K}_2$.
\end{proof}

\subsection*{Acknowledgements}
This research was supported by the German Research Foundation (DFG) grant {\emph Renewal Theory and Statistics of Rare Events in Infinite Ergodic Theory} (Geschäftszeichen KE 1440/2-1).
TS was partly supported by the Studienstiftung des Deutschen Volkes.

We thank Gerhard Keller for referring us to the paper of Rychlik.

\subsection*{Acknowledgements}
We thank Gerhard Keller for referring us to the paper of Rychlik.

\end{document}